\documentclass[11pt]{article}
 
\usepackage[sort,numbers]{natbib}
\usepackage[bookmarks=true,bookmarksnumbered=true,colorlinks=true,allcolors=black]{hyperref}

\usepackage{graphicx}
\usepackage{amsmath,amssymb,amsfonts,amsthm}
\usepackage{lscape}
\usepackage{arydshln}
\renewcommand*{\baselinestretch}{1.2}

\usepackage{geometry}
\usepackage{indentfirst}

\usepackage{enumitem,color}

\newtheorem{theorem}{Theorem}[section]
\newtheorem{lemma}{Lemma}[section]
\newtheorem{proposition}{Proposition}[section]
\newtheorem{corollary}{Corollary}[section]
\theoremstyle{definition}

\newtheorem*{rmk*}{Remark}
\newtheorem{rmk}{Remark}[section]

\DeclareMathOperator{\variance}{Var}

\DeclareMathOperator{\trace}{tr}
\DeclareMathOperator{\diag}{diag}
\DeclareMathOperator{\E}{E}

\DeclareMathOperator{\rank}{rank}

\allowdisplaybreaks

\numberwithin{equation}{section}
\counterwithin{figure}{section}
\counterwithin{table}{section}

\makeatletter
    \renewcommand*{\section}{\@startsection{section}{1}{\z@}%
    {6pt}{3pt}{\reset@font\normalsize\bfseries}}
\makeatother
    
\makeatletter
    \renewcommand*{\subsection}{\@startsection{subsection}{2}{\z@}%
    {3pt}{3pt}{\reset@font\normalsize\mdseries\itshape}}
\makeatother

\makeatletter
    \renewcommand*{\subsubsection}{\@startsection{subsubsection}{3}{\z@}%
    {3pt}{3pt}{\reset@font\normalsize\mdseries\itshape}}
\makeatother

\makeatletter
\def\@listi{\leftmargin\leftmargini
  \topsep=.5\baselineskip 
  \partopsep=0pt \parsep=0pt \itemsep=0pt}
\let\@listI\@listi
\@listi
\def\@listii{\leftmargin\leftmarginii
  \labelwidth\leftmarginii \advance\labelwidth-\labelsep
  \topsep=0pt \partopsep=0pt \parsep=0pt \itemsep=0pt}
\def\@listiii{\leftmargin\leftmarginiii
  \labelwidth\leftmarginiii \advance\labelwidth-\labelsep
  \topsep=0pt \partopsep=0pt \parsep=0pt \itemsep=0pt}
\def\@listiv{\leftmargin\leftmarginiv
  \labelwidth\leftmarginiv \advance\labelwidth-\labelsep
  \topsep=0pt \partopsep=0pt \parsep=0pt \itemsep=0pt}
\makeatother

\makeatletter
\newcommand{\opnorm}{\@ifstar\@opnorms\@opnorm}
\newcommand{\@opnorms}[1]{%
  \left|\mkern-1.5mu\left|\mkern-1.5mu\left|
   #1
  \right|\mkern-1.5mu\right|\mkern-1.5mu\right|
}
\newcommand{\@opnorm}[2][]{%
  \mathopen{#1|\mkern-1.5mu#1|\mkern-1.5mu#1|}
  #2
  \mathclose{#1|\mkern-1.5mu#1|\mkern-1.5mu#1|}
}
\makeatother

\def\be#1{\begin{equation*}#1\end{equation*}}
\def\ben#1{\begin{equation}#1\end{equation}}

\def\besn#1{\begin{equation}\begin{split}#1\end{split}\end{equation}}

\def\bm#1{\begin{multline*}#1\end{multline*}}

\def\ba#1{\begin{align*}#1\end{align*}}
\def\ban#1{\begin{align}#1\end{align}}

\newcommand{\lpa}{\left(}
\newcommand{\rpa}{\right)}
\newcommand{\eps}{\varepsilon}
\newcommand{\sch}{\mathcal S}
\newcommand{\levy}{\lambda}

\newcommand{\bs}[1]{\boldsymbol{#1}}
\newcommand{\ol}[1]{\overline{#1}}
\newcommand{\ul}[1]{\underline{#1}}

\newcommand{\wt}[1]{\widetilde{#1}}
\newcommand{\mf}[1]{\mathfrak{#1}}
\newcommand{\mcl}[1]{\mathcal{#1}}

\newcommand{\expe}[1]{\mathrm{E}[#1]}
\newcommand{\norm}[1]{\left\|#1\right\|}
\newcommand{\bra}[1]{\left(#1\right)}
\newcommand{\cbra}[1]{\left\{#1\right\}}
\newcommand{\sbra}[1]{\left[#1\right]}
\newcommand{\abs}[1]{\left|#1\right|}

\makeatletter
\newcommand{\pushright}[1]{\ifmeasuring@#1\else\omit\hfill$\displaystyle#1$\fi\ignorespaces}
\newcommand{\pushleft}[1]{\ifmeasuring@#1\else\omit$\displaystyle#1$\hfill\fi\ignorespaces}
\makeatother

\title{Spectral norm bounds for high-dimensional realized covariance matrices and application to weak factor models}
\author{Yuta Koike
\thanks{Graduate School of Mathematical Sciences, University of Tokyo}
\thanks{CREST, Japan Science and Technology Agency}
}

\begin{document}

\maketitle

\begin{abstract}

Motivated by statistical analysis of latent factor models for high-frequency financial data, we develop  sharp upper bounds for the spectral norm of the realized covariance matrix of a high-dimensional It\^o semimartingale with possibly infinite activity jumps. 
For this purpose, we develop Burkholder--Gundy type inequalities for matrix martingales with the help of the theory of non-commutative $L^p$ spaces. 
The obtained bounds are applied to estimating the number of (relevant) common factors in a continuous-time latent factor model from high-frequency data in the presence of weak factors. 
\vspace{3mm}


\end{abstract}

\section{Introduction}\label{sec:intro}

Given a stochastic basis $\mf B=(\Omega,\mcl F,(\mcl F_t)_{t\in[0,1]},P)$, we consider a $d$-dimensional semimartingale $Z=(Z_t)_{t\in[0,1]}$ defined on $\mf B$. The aim of this paper is to develop an upper bound for the spectral norm of the realized covariance matrix of $Z$,
\[
{[Z,Z]}^n_1=\sum_{i=1}^n(Z_{i/n}-Z_{(i-1)/n})(Z_{i/n}-Z_{(i-1)/n})^\top,
\]
when the dimension $d$ is comparable to or larger than $n$. 
This is motivated by the following statistical problem in high-frequency financial econometrics. 
Let us consider the situation where the log-price process of $d$ assets $Y=(Y_t)_{t\in[0,1]}$ is modeled as the following continuous-time latent factor model:
\ben{\label{eq:factor-model}
Y_t=\beta F_t+Z_t,
}
where $\beta$ is a non-random $d\times r$ matrix consisting of factor loadings and $F=(F_t)_{t\in[0,1]}$ is an $r$-dimensional semimartingale consisting of factor processes, with $r$ assumed to be much smaller than $d$. The process $Z$ corresponds to the idiosyncratic component. 
We observe the process $Y$ at equidistant times $t_i=i/n$, $i=0,1,\dots,n$. 
Then, we are interested in extracting information about the systematic component, $\beta F_t$, from the observation data $(Y_{t_i})_{i=0}^n$. 
Principal component analysis (PCA) is a standard tool to accomplish this purpose. 
This is based on the idea that the first $r$ largest eigenvalues of ${[Y,Y]}^n_1$ is inflated by those of $\beta{[F,F]}^n_1\beta^\top$ while the remaining ones are dominated by $\|{[Z,Z]}^n_1\|$, the spectral norm of ${[Z,Z]}^n_1$. 
Therefore, we could extract information about the systematic component as long as $\|{[Z,Z]}^n_1\|$ has a moderate size compared to eigenvalues of $\beta{[F,F]}^n_1\beta^\top$. 
In econometrics, it has been conventionally assumed that factors are pervasive in the sense that all the eigenvalues of $\beta{[F,F]}^n_1\beta^\top$ have the same order as the number of assets $d$. 
However, researchers have increasingly recognized that it is more realistic to consider the situation that some eigenvalues of $\beta{[F,F]}^n_1\beta^\top$ diverge at slower rates than $d$. 
Factor models allowing such a situation are referred to as \textit{weak factor models} and have been actively studied in recent years. 
For backgrounds and recent developments in this area, see \cite{BaNg23,Fr22,UeYa23a,UeYa23b,CoKo22,GXZ21} and references therein. 
Also, we refer to \cite{Tr86}, \cite[Section 6.4]{UeYa23a} and \cite[Section 5]{CNP23} for empirical evidence of weak factors in financial data. 
For weak factor models, it becomes more important to establish a sharper bound for $\|{[Z,Z]}^n_1\|$ in order to justify PCA based estimation. 
In view of Bai--Yin's law (cf.~\cite[Theorem 5.31]{Ve12}), we attempt to establish a bound of order $O_p(1+d/n)$ for $\|{[Z,Z]}^n_1\|$ without any extra structural conditions on the cross-sectional dependence of $Z$.

Statistical analysis of a high-dimensional continuous-time latent factor model using high-frequency data has been extensively studied in recent years. 
\cite{AX2017,DLX2019,FK2017} are examples where a continuous-time latent factor model is analyzed in the context of estimating high-dimensional quadratic covariation matrices. 
Estimation for the systematic component of a high-dimensional continuous-time model has been developed in, among others, \cite{Pe19,SXZ23,KLZ19,LCL23} for the constant loading case and \cite{Kong17,Kong18,KLLL23,CMZ20} for the general case. 
To the author's knowledge, however, none of them has formally considered weak factors in the above sense. We remark that \citet{LCL23} consider a different type of weak factors that are caused by the difference between the magnitudes of efficient log-prices and microstructure noise. 
Besides, jumps have been rarely incorporated into the model, especially for the idiosyncratic process where the high-dimensionality particularly matters. 
Two exceptions are the work by \citet{Pe19} and \citet{SXZ23}, both of which allow for finite activity jumps in the idiosyncratic process. 
By contrast, we allow the idiosyncratic process to have infinite activity jumps with possibly infinite variation. 
This is empirically important since there is some evidence that a significant number of individual stocks have infinite activity idiosyncratic jumps; see \cite[pp.~410--412 and 435--437]{AJ2014}.

For a discrete-time model, spectral distribution theory of random matrices is a traditional tool to derive a spectral norm bound for the empirical covariance matrix of the idiosyncratic process; see e.g.~\cite{BaNg06,On10,AhHo13}. 
However, for a continuous-time model observed on a fixed time interval, the corresponding theory is available only under rather restrictive assumptions; see \cite{ZL2011,HP2014}. 
\citet{MoWe17} have developed an alternative way to obtain such bounds for various discrete-time models using \citet{La05}'s inequality for the spectral norm of a random matrix with independent entries, but it is again non-trivial to extend this approach to the continuous-time setting. 
For this reason, to establish bounds for $\|{[Z,Z]}^n_1\|$, most articles rely on bounding the spectral norm by the Frobenius norm or the $\ell^\infty$-operator norm; see the proofs of \cite[Theorem 2]{AX2017}, \cite[Lemma A.1]{Kong17}, \cite[Lemma P.2]{Pe19}, \cite[Lemma 1]{SXZ23} and \cite[Lemma A.1]{LCL23} for instance. 
This scheme yields only coarse bounds of order $O_p(1+d/\sqrt n)$ unless we impose a non-trivial restriction on the cross-sectional dependence of $Z$.  
One exception is the work by \citet{CNP23}, where a Bernstein type inequality for matrix martingales has been developed to tackle the problem. While their approach gives an almost satisfactory bound for $\|{[Z,Z]}^n_1\|$ when $Z$ is continuous, this is not the case when $Z$ has jumps. 
The main aim of this paper is to fill in this gap. 


In view of the decomposition
\[
{[Z,Z]}^n_1-[Z,Z]_1=\sum_{i=1}^n\{(\Delta^n_iZ)(\Delta^n_iZ)^\top-\Delta^n_i[Z,Z]\},
\]
where $\Delta^n_iZ=Z_{t_i}-Z_{t_{i-1}}$ and $\Delta^n_i[Z,Z]=[Z,Z]_{t_i}-[Z,Z]_{t_{i-1}}$, our problem can be classified into the problem of bounding the spectral norm of a sum of random matrices, and this problem has been extensively studied in the literature. We refer to \cite{Ve12,vH17,Tr15} for an overview of this topic. In this context, roughly speaking, a major difficulty specific to our problem is caused by the fact that the magnitude of the summands
\[
\max_{i=1,\dots,n}\|(\Delta^n_iZ)(\Delta^n_iZ)^\top-\Delta^n_i[Z,Z]\|
\]
does not vanish as $n\to\infty$ due to the presence of infinite activity jumps. 
This fact seems to prevent us from using most standard techniques available in the literature. 
In particular, it is unclear whether $[Z,Z]^n_1$ concentrates around $[Z,Z]_1$ in the presence of infinite activity jumps. 
Instead, we directly bound $\|[Z,Z]^n_1\|$ by revisiting the classical approach of \citet{Ru99}. 
Since \citet{Ru99}'s original proof essentially uses the independence of summands, we resolve this issue by developing Burkholder--Gundy type inequalities for matrix martingales with the help of the theory of non-commutative $L^p$ spaces. 
A special case of these inequalities follows from the Burkholder--Gundy inequalities for non-commutative martingales but this is not the case in general (see Remark \ref{rem:BG}), so they would be of independent interest. 

This paper is organized as follows. 
Section \ref{sec:main} presents the main theoretical results of the paper. 
In Section \ref{sec:factor}, we apply the developed theory to determine the number of relevant factors from high-frequency data. 
Section \ref{sec:simulation} conducts a simulation study.
Mathematical proofs are given in Section \ref{sec:proof}. 
In the Appendix, we give additional proofs and simulation results. 
\vspace{-2mm}

\paragraph{Notation} 
Throughout the paper, we assume $d\geq3$. 
For a vector $x\in\mathbb R^n$, $x^j$ denotes the $j$-th component of $x$. 
Also, we write $|x|=\sqrt{\sum_{j=1}^n(x^j)^2}$ and $|x|_\infty=\max_{j=1,\dots,n}|x^j|$. 
$\mathbb R^{m\times n}$ denotes the set of $m\times n$ (real) matrices. 
For a matrix $A\in\mathbb R^{m\times n}$, $A^{jk}$ denotes the $(j,k)$-th entry of $A$. 
Also, we write $\|A\|$ and $\|A\|_F$ for the spectral norm and the Frobenius norm of $A$, respectively. 
If $m=n$, $\diag(A)$ denotes the diagonal matrix with the same diagonal entries as $A$. 
If $A$ is symmetric, we denote by $\lambda_1(A)\geq\cdots\geq\lambda_m(A)$ the eigenvalues of $A$. 
Given two sequences of non-negative numbers $a_n$ and $b_n$, we write $a_n\asymp b_n$ if $a_n=O(b_n)$ and $b_n=O(a_n)$ as $n\to\infty$. 
Similarly, given two sequences of non-negative random variables $\xi_n$ and $\eta_n$, we write $\xi_n\asymp_p \eta_n$ if $\xi_n=O_p(\eta_n)$ and $\eta_n=O_p(\xi_n)$ as $n\to\infty$.

\section{Spectral norm bounds for high-dimensional realized covariance matrices}\label{sec:main}

We assume that $Z$ is a $d$-dimensional It\^o semimartingale with the following Grigelionis decomposition:
\begin{equation}\label{eq:ito}
Z_t=Z_0+\int_0^tb_sds+\int_0^t\sigma_sdW_s
+\kappa(\delta)\star(\mu-\nu)_t
+\kappa'(\delta)\star\mu_t,
\end{equation}
where $W=(W_s)_{s\in[0,1]}$ is a $d'$-dimensional standard Wiener process, $\mu$ is a Poisson random measure on $[0,1]\times E$ with $E$ a Polish space, $\nu$ is the intensity measure of $\mu$ of the form $\nu(dt,dz)=dt\otimes\lambda(dz)$ with $\lambda$ a $\sigma$-finite measure on $E$, $b$ is a progressively measurable $\mathbb{R}^d$-valued process, $\sigma$ is a progressively measurable $\mathbb{R}^{d\times d'}$-valued process, $\delta$ is a predictable $\mathbb{R}^d$-valued function on $\Omega\times[0,1]\times E$, and $\kappa(z)=(z^11_{\{|z^1|\leq1\}},\dots,z^d1_{\{|z^d|\leq1\}})^\top$ and $\kappa'(z)=z-\kappa(z)$ for $z\in\mathbb R^d$. 
Also, $\star$ denotes the integral (either stochastic or ordinary) with respect to some (integer-valued) random measure. 
Here and below we use standard concepts and notation in stochastic calculus, which are described in detail in e.g.~Chapter 2 of \cite{JP2012}. 
Also, we denote by $\delta_j$ the $j$-th component function of $\delta$. 
\begin{rmk}[Choice of the truncation function]
A standard Grigelionis decomposition in the literature uses the function $x\mapsto x1_{\{|x|\leq1\}}$ to separate ``big'' and ``small'' jumps (cf.~Eq.(2.1.30) of \cite{JP2012}). 
We can rewrite our decomposition \eqref{eq:ito} in this standard form as
\[
Z_t=Z_0+\int_0^tb'_sds+\int_0^t\sigma_sdW_s
+(\delta1_{\{|\delta|\leq1\}})\star(\mu-\nu)_t
+(\delta1_{\{|\delta|>1\}})\star\mu_t,
\]
where $b'_s=b_s+\int_E\kappa(\delta(s,z))1_{\{|\delta(s,z)|>1\}}\levy(dz)$. 
However, we prefer starting from \eqref{eq:ito} in order to avoid assuming (local) boundedness for $\max_j\int_E|\delta_j(s,z)|1_{\{|\delta(s,z)|>1\}}\levy(dz)$: 
If we start from the standard one, such an assumption seems inevitable when applying the standard argument of compensating the big jump term to create a martingale term. 
Nevertheless, this assumption seems inappropriate in the high-dimensional setting because the event $\{|\delta(s,z)|>1\}$ can be caused by many small jumps in the cross-section. 
As an illustration, suppose that the jump part of $Z$ is given by a compound Poisson process as
\[
Z_t=Z_0+\int_0^tb_sds+\int_0^t\sigma_sdW_s+Z^J_t\quad\text{with}~~Z^J_t=\frac{2}{\sqrt d}\sum_{k=1}^{N_t}\epsilon_k,
\]
where $N=(N_t)_{t\in[0,1]}$ is a Poisson process with intensity $d$ and $(\epsilon_k)_{k=1}^\infty$ is a sequence of i.i.d.~random vectors in $\mathbb R^d$ such that the components of $\epsilon_k$ are i.i.d.~Rademacher variables. In this case, we always have $|\delta(s,z)|>1$ as soon as $N$ jumps at $s$, so $\max_j\int_E|\delta_j(s,z)|1_{\{|\delta(s,z)|>1\}}\levy(dz)=2\sqrt d\to\infty$ as $d\to\infty$. 
We note that in this example the jump term is already compensated, so the afore-mentioned argument is indeed unnecessary for this specific situation. 
Another remark is that we have $\|[Z^J,Z^J]_1\|=O_p(1)$ as $d\to\infty$ in this example because $\E[[Z^J,Z^J]_1]=4I_d$ and $\|[Z^J,Z^J]_1-\E[[Z^J,Z^J]_1]\|=O_p(1)$ by Eq.(5.23) in \cite{Ve12}. 
\end{rmk}

We first develop non-asymptotic bounds for $\|{[Z,Z]}^n_1\|$. 
For this purpose, we impose the following boundedness condition on the coefficient processes and spectral norm of the quadratic covariation matrix of $Z$: 
\begin{enumerate}[label={\normalfont[SH]}]

\item \label{ass:SH} There are constants $\Lambda\geq1$ and Borel functions $\gamma_j:E\to[0,\infty)$ $(j=1,\dots,d)$ such that
\ba{
&|b_s(\omega)|_\infty^2\leq\Lambda,\qquad
\|\sigma_s(\omega)\|^2\leq\Lambda,\qquad
|\delta_j(\omega,s,z)|\wedge1\leq\gamma_j(z),\\
&\int_E\gamma_j(z)^2\levy(dz)<\infty,\qquad
\|[Z,Z]_1(\omega)\|\leq \Lambda 
}
for all $\omega\in\Omega$, $s\in[0,1]$, $z\in E$ and $j=1,\dots,d$. 
%
%
%
%
%
\end{enumerate}
We emphasize that Assumption \ref{ass:SH} does not impose any special restriction on the cross-sectional dependence of $Z$ except for the natural boundedness assumption for $\|\sigma_s\|$ and $\|[Z,Z]_1\|$. 

When $Z$ is continuous, we have the following concentration inequality for $[Z,Z]^n_1$.
\begin{theorem}\label{cont-bound}
Assume \ref{ass:SH} and $\mu=0$. There is a universal constant $C>0$ such that, for any $u>0$,
\ben{\label{eq:cont-bound}
\norm{{[Z,Z]}^n_1-[Z,Z]_1}\leq C\Lambda\bra{\frac{d+u}{n}+\sqrt{\frac{d+u}{n}}}
}
with probability at least $1-2e^{-u}$. 
\end{theorem}
%
Theorem \ref{cont-bound} can be shown by a rather straightforward modification of the proof of an analogous statement for the sample covariance matrix of independent samples, but the result itself seems new in the literature. 
We note that the coarse estimate by $\norm{{[Z,Z]}^n_1-[Z,Z]_1}_F$ yields a bound of order $O(d/\sqrt n)$ in general. 
Also, Theorem 2.4 in \cite{CNP23} gives a slightly worse bound than \eqref{eq:cont-bound}, i.e.~$d$ is replaced by $d\log d$. Meanwhile, Theorem 2.4 in \cite{CNP23} can lead to a much better bound when $\trace(\sigma_s\sigma_s^\top)$ is relatively smaller than $d$ for all $s$. 

In the general setting, we have the following high probability bound for $\|[Z,Z]^n_1\|$.
\begin{theorem}\label{jump-bound2}
Assume \ref{ass:SH} and $n\geq10$. There is a universal constant $C>0$ such that, for any $\alpha\in[2/n,2/e^2)$,
\[
\norm{{[Z,Z]}^n_1}
\leq C\Lambda\cbra{\frac{d+\log(1/\alpha)}{n}+\ol\gamma_2+\log^2(d/\alpha)\bra{\frac{d\log n}{n}\ol{\gamma}_2+\log^2n}}
\]
with probability at least $1-\alpha$, where
\[
\ol\gamma_2:=\max_{j=1,\dots,d}\int_E\gamma_j(z)^2\levy(dz).
\] 
\end{theorem}
As demonstrated in the introduction, it seems less hopeful in the high-dimensional setting that $[Z,Z]^n_1$ concentrates around $[Z,Z]_1$ in terms of the spectral norm if $Z$ allows for infinite activity jumps. 
Theorem \ref{jump-bound2} claims that we can still derive an adequate bound for $\norm{{[Z,Z]}^n_1}$ without any special restriction on $Z$; see also the following remark. In Section \ref{sec:factor}, we see that this type of bound is enough for some statistical applications. 
\begin{rmk}[Order of $\ol\gamma_2$]
In practice, it is natural to assume that the quantity $\ol\gamma_2$ is bounded by a dimension-free constant. 
To illustrate this point, let us consider a simple case that the jump part of $Z$ is given by a $d$-dimensional L\'evy process $L=(L_t)_{t\in[0,1]}$ as
\[
Z_t=Z_0+\int_0^tb_sds+\int_0^t\sigma_sdW_s+L_t,
\]
where $b$ and $\sigma$ are as in \ref{ass:SH}. 
For each $j=1,\dots,d$, we assume that the L\'evy--It\^o decomposition of $L^j$ is given by
\[
L^j_t=\int_0^t\int_{|x|\leq1}x(\mu_j-\nu_j)(ds,dx)+\int_0^t\int_{|x|>1}x\mu_j(ds,dx),
\]
where $\mu_j$ is the jump measure of $L^j$ and $\nu_j$ is the compensator of $\mu_j$. 
Then, $Z$ is of the form \eqref{eq:ito} with $\delta(s,z)=z$ and $\mu$ the jump measure of $L$, and \ref{ass:SH} is satisfied with $\gamma_j(z)=|z^j|\wedge1$, $j=1,\dots,d$. 
In this case, we have by \cite[Theorem 4.1]{CT2004}
\[
\ol\gamma_2=\max_{j=1,\dots,d}\int_{\mathbb R}(x^2\wedge1)\levy_j(dx),
\] 
where $\levy_j$ is the L\'evy measure of $L^j$. Hence, $\ol\gamma_2$ is determined by the marginal distributions of $L$ and thus we may naturally assume it is bounded by a dimension-free constant. 
\end{rmk}

As usual, the boundedness assumption can be relaxed by a localization procedure if we are concerned with asymptotic results.  
Recall that an increasing sequence of stopping times $(\tau_m)_{m=1}^\infty$ is called a \textit{localizing sequence} if $P(\tau_m=1)\to1$ as $m\to\infty$. 
Also, we say that a process $h=(h_s)_{s\in[0,1]}$ is \textit{locally bounded} if there is a localizing sequence  $(\tau_m)_{m=1}^\infty$ such that $\sup_{\omega\in\Omega,0<s\leq \tau_m(\omega)}|h_s(\omega)|<\infty$ for every $m$. Note that, following \cite{JP2012}, we do not require the boundedness of $h_0$ in the latter definition. 
\begin{enumerate}[label={\normalfont[H]}]

\item \label{ass:H} $Z=Z^{(n)}$ and $d=d_n$ may depend on $n$ such that $\norm{[Z,Z]_1}=O_p(\Lambda_n)$ as $n\to\infty$ for some (deterministic) sequence $\Lambda_n\in[1,\infty)$. Moreover, there are a locally bounded process $h=(h_s)_{s\in[0,1]}$ and a localizing sequence $(\tau_m)_{m=1}^\infty$ such that all of them are independent of $n$ and satisfy the following conditions for all $n\in\mathbb N$:
\begin{enumerate}[label=(\roman*)]

\item $|b_s(\omega)|_\infty^2+\|\sigma_s(\omega)\|^2\leq \Lambda_nh_s(\omega)$ for all $(\omega,s)\in\Omega\times[0,1]$.

\item For every $m\in\mathbb N$, there are Borel functions $\gamma^{(n)}_{m,j}:E\to[0,\infty)$ $(j=1,\dots,d)$ such that $|\delta_j(\omega,t,z)|\wedge1\leq\gamma^{(n)}_{m,j}(z)$ for all $(\omega,t,z)\in\Omega\times[0,1]\times E$ with $t\leq\tau_m(\omega)$. Moreover,
\[
\ol\gamma_{m,2}:=\sup_{n\in\mathbb N}\max_{j=1,\dots,d}\int_E\gamma_{m,j}^{(n)}(z)^2\levy(dz)<\infty.
\]

\end{enumerate}

\end{enumerate}
Assumption \ref{ass:H} can be viewed as a high-dimensional analog of Assumption (H) in \cite{JP2012} that is the most basic assumption in high-frequency financial econometrics. 

\begin{theorem}\label{localization}
Assume \ref{ass:H}. Then
\[
\norm{{[Z,Z]}^n_1}
=O_p\bra{\Lambda_n(\log^2d)\bra{\frac{d}{n}\log (d\wedge n)+\log^2(d\wedge n)}}
\]
as $n\to\infty$. 
\end{theorem}
Theorem \ref{localization} is shown by a standard localization argument that is almost the same as the one in Lemma 4.4.9 in \cite{JP2012}. 
We give a detailed proof in Appendix \ref{pr-local} for the sake of completeness. 
\begin{rmk}[Implication of Assumption \ref{ass:H}]
It is worth mentioning that $\|[Z,Z]_1\|$ gives an upper bound for $\sup_{s\in(0,1]}|\Delta Z_s|^2$. In fact, we have by Weyl's inequality. 
\ben{\label{eq:jump-bound}
\|[Z,Z]_1\|\geq\norm{\sum_{s\in(0,1]}(\Delta Z_s)(\Delta Z_s)^\top}\geq\sup_{s\in(0,1]}\|(\Delta Z_s)(\Delta Z_s)^\top\|=\sup_{s\in(0,1]}|\Delta Z_s|^2.
}
Thus, the components of $Z$ have no simultaneous big jump under \ref{ass:H}. 
In the financial application presented in Section \ref{sec:factor}, $Z$ is used for modeling the idiosyncratic component of assets, so simultaneous big jumps should be attributed to the systematic component. 
\end{rmk}
%


\section{Estimation of the number of relevant factors from high-frequency data}\label{sec:factor}

In this section, we apply the theory developed above to the problem of estimating the number of ``relevant'' latent factors from high-frequency data. 
As in the introduction, let $Y=(Y_t)_{t\in[0,1]}$ be the log-price process of $d$ assets and assume that it is of the form
\[
Y_t=\beta F_t+Z_t,
\]
where $\beta$ is a non-random $d\times r$ matrix consisting of factor loadings, $F=(F_t)_{t\in[0,1]}$ is an $r$-dimensional latent factor process, and $Z=(Z_t)_{t\in[0,1]}$ is a $d$-dimensional idiosyncratic process. 
We have observation data for $Y$ at equidistant time points $t_i=i/n$, $i=0,1,\dots,n$. 

We allow for weak factors in the sense that the matrix $d^{-1}\beta^\top\beta$ may have an eigenvalue tending to 0 as $d\to\infty$. Formally, we impose the following condition:
 
\begin{enumerate}[label={\normalfont[A1]}]

\item \label{ass:F} 
(i) $F$ and $r$ are fixed as $n\to\infty$. Also, $F$ is an $r$-dimensional semimartingale such that $[F,F]_1$ is almost surely invertible.

(ii) There are constants $1\geq\alpha_1\geq\cdots\geq\alpha_r>0$ such that $\lambda_j(\beta^\top\beta)\asymp d^{\alpha_j}$ as $d\to\infty$ for every $j=1,\dots,r$. 

\end{enumerate}
Assumption \ref{ass:F} is a natural continuous-time version of the assumptions commonly used in the literature to model weak factors; see \citet{BaNg23} and \citet{UeYa23a,UeYa23b} for example. 
A simple situation satisfying \ref{ass:F} is the case that $F$ is an $r$-dimensional standard Wiener process and the column vectors $\beta^{\cdot 1},\dots,\beta^{\cdot r}$ of $\beta$ are orthogonal and satisfy $|\beta^{\cdot j}|^2\asymp d^{\alpha_j}$ for $j=1,\dots,r$. 
Intuitively, the latter corresponds to the situation that only $O(d^{\alpha_j})$ assets are significantly exposed to the $j$-th factor. Such a factor is called a \textit{local factor} in \citet{Fr22}. 

As discussed in \cite{Fr22} in detail, in the presence of weak factors, we are often interested in the number of sufficiently strong factors because only those are practically relevant (see also the discussion in \cite[Section 2.3.2]{CoKo22}). 
To be precise, fix a threshold value $\tau\in(0,1)$ and define
\[
r_\tau:=\max\{j\in\{1,\dots,r\}:\alpha_j>\tau\}.
\]
Here and below, we interpret $\max\emptyset=0$ by convention. We are interested in estimating $r_\tau$ from the observation data $(Y_{i/n})_{i=0}^n$.

To ensure that the factor process $F$ captures all the non-negligible factors in $Y$, we need to assume that the largest eigenvalue of $[Z,Z]_1$ is not too large. 
To formulate this assumption, it is convenient to use the notion of slowly varying functions.  
A function $f:(0,\infty)\to(0,\infty)$ is said to be \textit{slowly varying} if $\lim_{t\to\infty}f(tx)/f(t)=1$ for any $x>0$. See e.g.~\cite[Section 0.4]{Res1987} for a detailed exposition. 
The following is the only property of a slowly varying function used in this paper (cf.~\cite[Proposition 0.8]{Res1987}): If $f:(0,\infty)\to(0,\infty)$ is a slowly varying function, then 
\ben{\label{eq:svf}
f(x)=o(x^\eps)\qquad\text{as}\quad x\to\infty\quad\text{for any }\eps>0.
}
Prominent examples of slowly varying functions are the constant function and logarithmic function. 
\begin{enumerate}[label={\normalfont[A2]}]

\item \label{ass:svf} We have \ref{ass:H} and there is a slowly varying function $f:(0,\infty)\to(0,\infty)$ such that $\Lambda_n=O(f(d))$ as $n\to\infty$. 

\end{enumerate}

\begin{lemma}\label{lemma:cov-eigen}
Assume \ref{ass:F}--\ref{ass:svf}. 
Fix a constant $\tau\in(0,1)$ and suppose that 
\ben{\label{ass:d-n}
d\to\infty\quad\text{and}\quad 
d^{1-\tau}\log^2(d)\log(d\wedge n)=O(n)\qquad
\text{as }n\to\infty.
}
Then we have the following results:
\begin{enumerate}[label={\normalfont(\alph*)}]

\item\label{eigen-strong} $\lambda_j\bra{{[Y,Y]}^n_1}\asymp_p d^{\alpha_j}$ as $n\to\infty$ for any $j\leq r_\tau$. 

\item\label{eigen-weak} $\max_{j:j>r_\tau}\lambda_j\bra{{[Y,Y]}^n_1}=O_p(d^\tau)$ as $n\to\infty$. 

\end{enumerate}
\end{lemma}
%
Lemma \ref{lemma:cov-eigen} suggests that we could estimate $r_\tau$ by thresholding the eigenvalues of $[Y,Y]^n_1$ at $d^\tau g(d)$ with $g$ a slowly varying function such that $g(d)\to\infty$ as $d\to\infty$ (recall \eqref{eq:svf}). 
Formally, we define
\[
\hat r^{BN}_\tau:=\max\cbra{j\in\{1,\dots,r_{max}\}:\lambda_j({[Y,Y]}^n_1)>d^\tau g(d)},
\]
where $r_{max}$ is an upper bound of $r$. 
From a theoretical point of view, we can simply set $r_{max}=d$, but we use a much smaller value in practice to avoid getting an economically uninterpretable estimate in finite samples. 
As discussed in \cite[Section 4.1.1]{Fr22}, this estimator can be seen as an extension of the so-called $PC_{p1}$ estimator proposed in \cite{BN2002}. 

\begin{proposition}\label{prop:bn}
Suppose that $g$ is slowly varying and $g(d)\to\infty$ as $d\to\infty$. Then, under the assumptions of Lemma \ref{lemma:cov-eigen}, $P(\hat r^{BN}_\tau=r_\tau)\to1$ as $n\to\infty$. 
\end{proposition}

Although Proposition \ref{prop:bn} shows that the estimator $\hat r^{BN}_\tau$ has strong consistency, it turns out that the finite sample performance of $\hat r^{BN}_\tau$ is very sensitive to the choice of $g(d)$ and it is difficult to find a universally adequate choice of $g(d)$. 
Therefore, following \citet{Pe19}, we consider another estimator based on perturbed eigenvalue ratio statistics,
\[
ER_j({[Y,Y]}^n_1):=\frac{\lambda_j({[Y,Y]}^n_1)+d^\tau g(d)}{\lambda_{j+1}({[Y,Y]}^n_1)+d^\tau g(d)},\qquad j=1,\dots,d-1.
\]
Lemma \ref{lemma:cov-eigen} implies that $ER_j({[Y,Y]}^n_1)$ is larger than 1 if $j= r_\tau$ and close to 1 if $j>r_\tau$. This suggests the following estimator for $r_\tau$:
\[
\hat r^{P}_\tau(\gamma):=\max\cbra{j\in\{1,\dots,r_{max}\}:ER_j({[Y,Y]}^n_1)>1+\gamma}.
\]
Here, $\gamma>0$ is a tuning parameter. 
\begin{proposition}\label{prop:pelger}
Fix $\gamma>0$. 
Under the assumptions of Proposition \ref{prop:bn}, $P(\hat r^{P}_\tau(\gamma)=r_\tau)\to1$ as $n\to\infty$. 
\end{proposition}


We can also construct versions of these estimators using the realized correlation matrix 
\[
{R}_{Y,n}:=\diag\bra{{[Y,Y]}^n_1}^{-\frac{1}{2}}{[Y,Y]}^n_1\diag\bra{{[Y,Y]}^n_1}^{-\frac{1}{2}}
\]
instead of ${[Y,Y]}^n_1$. That is,
\ba{
\hat r^{BN,cor}_\tau&:=\max\cbra{j\in\{1,\dots,r_{max}\}:\lambda_j({R}_{Y,n})>d^\tau g(d)},\\
\hat r^{P,cor}_\tau(\gamma)&:=\max\cbra{j\in\{1,\dots,r_{max}\}:ER_j({R}_{Y,n})>1+\gamma}.
}
An advantage of these versions is that the function $g(d)$ can be chosen in a scale-free manner. 
To establish their consistency, we need to impose additional mild assumptions. 
For a semimartingale $X$, we denote by $X^c$ the continuous martingale part of $X$ (cf.~Proposition 4.27 in \cite[Chapter I]{JS}).
\begin{enumerate}[label={\normalfont[A3]}]

\item \label{ass:R} (i) $\max_{1\leq j\leq d}\bra{[(Y^c)^j,(Y^c)^j]_1}^{-1}=O_p(1)$ as $n\to\infty$.

(ii) We can write the process $F$ of the form
\begin{align*}
F_t=F_0+\int_0^tb^F_sds+\int_0^t\sigma^F_sdW_s
+(\delta^F1_{\{|\delta^F|\leq1\}})\star(\mu-\nu)_t
+(\delta^F1_{\{|\delta^F|>1\}})\star\mu_t,\quad t\in[0,1],
\end{align*}
where $b^F=(b^F_s)_{s\in[0,1]}$ is a locally bounded progressively measurable $\mathbb{R}^r$-valued process, $\sigma^F$ is a progressively measurable $\mathbb{R}^{r\times d'}$-valued process such that $\|\sigma^F\|$ is locally bounded, and $\delta^F$ is a predictable $\mathbb{R}^r$-valued function on $\Omega\times[0,1]\times E$. Moreover, there is a localizing sequence $(\tau^F_m)_{m=1}^\infty$ (independent of $n$) such that, for every $m\in\mathbb N$, there is a Borel function $\gamma^F_m:E\to[0,\infty)$ satisfying $\sup_{n\in\mathbb N}\int_E\gamma^F_m(z)^2\levy(dz)<\infty$ and $|\delta^F(\omega,t,z)|\wedge1\leq\gamma^F_{m}(z)$ for all $(\omega,t,z)\in\Omega\times[0,1]\times E$ with $t\leq\tau^F_m(\omega)$. 
%
%
%

(iii) $\max_{1\leq j\leq d,1\leq k\leq r}|\beta^{jk}|=O(\Lambda_n)$ as $n\to\infty$. 

\end{enumerate}
\ref{ass:R}(ii) just says that $F$ is an $r$-dimensional It\^o semimartingale satisfying canonical conditions. 
\ref{ass:R}(iii) is also standard, which ensures that $\max_{1\leq j\leq d}[Y^j,Y^j]$ is moderate. 
\ref{ass:R}(i) is slightly stronger than the more natural assumption $\max_{1\leq j\leq d}\bra{[Y^j,Y^j]_1}^{-1}=O_p(1)$. We need this slightly stronger assumption because it is unclear in the high-dimensional setting whether the variations $[Y,Y]^n_1$ coming from infinite activity jumps concentrate around those of $[Y,Y]_1$. 
It is acceptable for stock data in view of the empirical results in \cite[pp.~446--447]{AJ2014}. 

\begin{lemma}\label{lemma:cor-eigen}
Assume \ref{ass:F}--\ref{ass:R}. 
Fix a constant $\tau\in(0,1)$ and assume \eqref{ass:d-n}. Then $\lambda_j({R}_{Y,n})/d^\alpha\to^p\infty$ as $n\to\infty$ for any $j\leq r_\tau$ and $\alpha<\alpha_j$. Moreover, $\max_{j:j>r_\tau}\lambda_j({R}_{Y,n})=O_p(d^\tau)$ as $n\to\infty$.
\end{lemma}

The next proposition can be shown in the same way as Propositions \ref{prop:bn} and \ref{prop:pelger}, so we omit the proof. 
\begin{proposition}\label{prop:rhat-cor}
Suppose that $g$ is slowly varying and $g(d)\to\infty$ as $d\to\infty$. Then, under the assumptions of Lemma \ref{lemma:cor-eigen}, $P(\hat r^{BN,cor}_\tau=r_\tau)\to1$ as $n\to\infty$. 
Also, for any fixed $\gamma>0$, $P(\hat r^{P,cor}_\tau(\gamma)=r_\tau)\to1$ as $n\to\infty$. 
\end{proposition}

\section{Simulation study}\label{sec:simulation}

In this section we conduct a simulation study to assess the finite sample performance of the estimators proposed in the previous section. 
We focus on the case $\tau=1/2$ since this case is practically most relevant as demonstrated in Section 3.2 of \cite{Fr22}.

\subsection{Implementation of estimators}


To implement the proposed estimators, we need to specify the function $g(d)$ and the parameter $\gamma$ for the case of the perturbed eigenvalue ratio based estimator. We first discuss this issue. 
We set $r_{max}=20$ throughout this section. 

For the estimator $\hat r_\tau^{BN}$, let us recall that it can be seen as a modification of the $PC_{p1}$ estimator of \cite{BN2002}. Motivated by this, we set $g(d)=\hat\sigma^2\sqrt{\log\log d}$ with
\ben{\label{eq:sigma2hat}
\hat\sigma^2=\sum_{j=r_{max}+1}^d\lambda_j([Y,Y]_1^n).
} 
Here, the choice of the coefficient $\sqrt{\log\log d}$ is motivated by \cite{Fr22}: For most relevant values of $d$, the values of $\sqrt{\log\log d}$ are only slightly larger than 1, so it has little impact on the effect of $g(d)$. 
Therefore, this choice is useful to focus on the influence of other parameters. 
We use the covariance matrix based version $\hat r_\tau^{BN}$ because preliminary (unreported) simulation results suggest that it would outperform the correlation matrix based one. 

For the estimator $\hat r_\tau^{P}(\gamma)$, preliminary experiments suggest that the correlation matrix based version would outperform, so we focus on the estimator $\hat r_\tau^{P,cor}(\gamma)$. 
Regarding the choice of $g(d)$, \citet{Pe19} suggests using (with $\tau=1/2$)
\ben{\label{eq:pelger-tune}
g(d)=\mathrm{median}\{\lambda_1(R_{Y,n}),\dots,\lambda_d(R_{Y,n})\},
}
and we will later see that this choice performs very well when the sample size is moderate. 
Unfortunately, however, this choice is problematic when $d>2n$ because the rank of $R_{Y,n}$ is smaller than $d/2$ in this case and thus we always have $g(d)=0$. 
For this reason, we simply set $g(d)=\sqrt{\log\log d}$. 
Indeed, as already discussed in \cite{Pe19}, the choice of $g(d)$ is less important; the choice of $\gamma$ is more important. We use $\gamma=0.05$ in our experiments. 
We show the sensitivity of $\hat r^{P,cor}_\tau(\gamma)$ to the choice of these tuning parameters in Figures \ref{fig:perturb} and \ref{fig:gamma} of Appendix \ref{sec:a-sim}.

For comparison, we also compute the following three estimators:
\begin{itemize}

\item \textbf{\citet{BN2002}'s $\bs{PC_{p1}}$ estimator:} This is defined by the following formula (see Eq.(14) of \cite{Fr22}):
\[
PC_{p1}:=\max\cbra{j\in\{1,\dots,r_{max}\}:\lambda_j([Y,Y]^n_1)>\hat\sigma^2\bra{1+\frac{d}{n}}\log\bra{\frac{dn}{d+n}}}.
\]
Here, $\hat\sigma^2$ is defined as \eqref{eq:sigma2hat}. 

\item \textbf{\citet{Pe19}'s estimator:} This is the estimator $\hat r_\tau^{P,cor}(\gamma)$ but we use the same tuning parameters as in \cite{Pe19}. That is, $g(d)$ is defined by \eqref{eq:pelger-tune} and $\gamma=0.2$. 
Note that we use the correlation matrix based one because unreported simulation results suggest its superior performance. 

\item \textbf{\citet{On10}'s estimator:} 
This is constructed by thresholding the differences between successive eigenvalues:
\[
\hat r^O(\delta):=\max\cbra{j\in\{1,\dots,r_{max}\}:\lambda_j(R_{Y,n})-\lambda_{j+1}(R_{Y,n})\geq\delta}.
\]
Here, the threshold parameter $\delta$ is calibrated by the ED algorithm described in \cite[Section IV]{On10}. 
Note that we again use the correlation matrix based estimator by the same reasoning as above. 

\end{itemize}
There is another popular estimator proposed by \citet{AhHo13} that is based on (non-perturbed) eigenvalue ratio statistics. 
However, the recent studies show that it severely underestimates the number of factors in the presence of a dominant factor; see e.g.~\cite{Pe19,Fr22}. We also observed the same phenomenon in our simulation setting, so we omit this estimator from tables below. 
%
%
%
%

\subsection{Simulation design}\label{sec:sim-design}

We basically follow the setting of \citet{Fr22} but adapt some parts to our continuous-time setting. 
First, we employ the same loading matrix $\beta$ as in \cite{Fr22}: The number of latent factors is $r=9$. 
For each $j\in\{1,\dots,9\}$, we randomly select $d_j$ entries of the $j$-th column of $\beta$ and fill them with i.i.d.~normal variables with mean 1 and variance 1. All other entries are set to 0. Here,
\ba{
d_1&=d,&
d_2&=d^{0.85},&
d_3&=d^{0.75},&
d_4&=d^{2/3},&
d_5&=d^{2/3},\\
d_6&=d^{0.6},&
d_7&=d^{1/3},&
d_8&=d^{1/4},&
d_9&=\log d.
}
If $d_j$ is not an integer, it is rounded to the nearest integer. 
Since we consider the case $\tau=1/2$, the number of relevant factors is $r_\tau=6$. 

For the factor process $F$, we assume 9 factors are independent, and each factor is generated by the following stochastic volatility model:
\ba{
\begin{cases}
dF^j_t=\mu_jdt+\rho_j\sigma^j_tdB^j_t+\sqrt{1-\rho_j^2}\sigma^j_tdW^j_t,\\
\sigma^j_t=\exp(a_{j}+b_{j}\varrho^j_t),\quad
d\varrho^j_t=-\kappa_j\varrho^j_tdt+dB^j_t,
\end{cases}
t\in[0,1],~j=1,\dots,9,
}
where $B^j$ and $W^j$ are two independent standard Wiener processes and $\mu_j,\rho_j,a_j,b_j$ and $\kappa_j$ are parameters specified below. 
This model has been conventionally used in the literature of high-frequency financial econometrics; see e.g.~\cite{HT2005,BNHLS2011}. 
The initial value $\varrho^j_0$ is generated from $N(0,(2\kappa_j)^{-1})$, the stationary distribution of $\varrho^j$. 
We use the same parameter values for all factors that are given by
\[
\mu_j=0.03,\quad
a_j=-5/16,\quad
b_j=1/8,\quad
\kappa_j=1/40,\quad
\rho_j=-0.3,
\]
following \cite{BNHLS2011}. In particular, we have $\E\sbra{[F^j,F^j]_1}=\E\sbra{\int_0^1(\sigma^j_t)^2dt}=1$ for all $j$. 


Finally, we model the idiosyncratic process $Z$ as $Z_t=\sqrt{\theta}AL_t$, where $\theta>0$ is a constant,  $A$ is a lower triangular matrix such that $AA^\top$ is the Toeplitz matrix $(\phi^{|j-k|})_{1\leq j,k\leq d}$ with $\phi\in(0,1)$ a constant, and $L=(L_t)_{t\in[0,1]}$ is either a $d$-dimensional standard Wiener process or a $d$-dimensional centered symmetric normal tempered stable (NTS) process (cf.~\cite[Section 4.4]{CT2004}) independent of the factor process. 
More concretely, for the latter case, $L$ is a $d$-dimensional L\'evy process with $L_1$ having the same law as $\sqrt{V}\zeta$, where $V$ and $\zeta$ are independent, $\zeta\sim N(0,1)$ and $V$ follows the positive tempered stable distribution $PTS(\alpha,c,\lambda)$ with some constants $c,\lambda>0$ and $\alpha\in(0,1)$, i.e. $\log\E[e^{-uV}]=c\Gamma(-\alpha)\{(\lambda+u)^\alpha-\lambda^\alpha\}$ for every $u>0$. 
The main advantage of using NTS processes is that they allow for Blumenthal--Getoor (BG) indices arbitrarily close to 2 while their exact simulation is possible. 
In fact, one can see that the BG index of $L$ is equal to $2\alpha$ from the expression of its L\'evy measure (cf. Eq.(4.25) of \cite{CT2004}). 
Moreover, $L_{1/n}$ has the same law as $\sqrt{V_n}\zeta$ where $V_n\sim PTS(\alpha,c/n,\lambda)$ and $\zeta\sim N(0,1)$ are independent and $V_n$ can be simulated exactly (cf.~\cite{KM2011}); hence we can exactly simulate the process $L$ as well. This is implemented in the function \texttt{rnts()} of the R package \texttt{yuima} which we used for simulating $L$. 

We vary $\alpha\in\{0.25,0.50,0.75\}$ to assess the effect of jump activity. 
We set $c=\lambda^{1-\alpha}/\Gamma(1-\alpha)$ and $\lambda=1-\alpha$ so that $\E[V]=\variance[V]=1$. This particularly implies $\E\sbra{[L^j,L^j]_1}=1$ for all $j$; hence the parameter $\theta$ can be interpreted as a signal-to-noise ratio as in \cite{Fr22}, so we set $\theta=1.5$ following \cite{Fr22}. 
Also, the parameter $\phi$ plays the same role as $\beta$ in \cite{Fr22}, so we set $\phi=0.1$. 
The effects of changes in $\theta$ and $\phi$ are shown in Figures \ref{fig:snr} and \ref{fig:phi} of Appendix \ref{sec:a-sim}. 

We vary the sample size $n$ as $n\in\{26,78,390\}$. These numbers correspond to sampling data at 15, 5, 1 minute frequencies, respectively, in a market open for 6.5 hours. 
The dimension $d$ is varied as $d\in\{100,500,1000,1500\}$, corresponding to typical numbers of constituents of stock indices. 

\subsection{Results}

Tables \ref{table:wiener}--\ref{table:alpha=0.75} report the simulation results: For each estimator $\hat r$, we report the empirical mean of $\hat r$ and the empirical probability that $\hat r=6$ (in parentheses) based on 1,000 Monte Carlo iterations. 
Table \ref{table:wiener} is for the case that $L$ is a $d$-dimensional standard Wiener process and Tables \ref{table:alpha=0.25}--\ref{table:alpha=0.75} are for the case that $L$ is a $d$-dimensional NTS process with stability index $\alpha\in\{0.25,0.50,0.75\}$. 

First we focus on $\hat r^{P,cor}_\tau(\gamma)$. 
As the tables reveal, its performance becomes better as $d$ and/or $n$ increase. This is in line with our asymptotic theory. 
Although it tends to underestimate the number of (relevant) factors, it gives reasonable estimates when $d$ and $n$ are relatively large. This underestimation seems to be caused by stochastic volatilities of factor processes. That is, the volatility of a factor process could be very low at some simulation runs, and such a factor would be difficult to be identified. In Appendix \ref{sec:a-sim} we report the results for the case that factors are Wiener processes, where we find much better performance. 
Finally, jump activity of the idiosyncratic process seems to have little impact on the performance. 

Next we focus on $\hat r^{BN}_\tau$. Although it shows better performance for larger $d$ and $n$ as predicted by our asymptotic theory, its performance is unstable even for relatively large values of $d$ and $n$. We might be able to resolve this problem by selecting the function $g(d)$ more carefully, but we have no theoretical guidance, making this task difficult. 
We remark that such sensitivity is a common issue in thresholding based estimators; see \cite[p.94]{Fr22}. 

Finally, we discuss the performance of the remaining estimators. 
Although they are not designed to estimate the number of relevant factors, they perform reasonably well in some situations. 
In particular, \citet{Pe19}'s and \citet{On10}'s estimators are competitive with $\hat r^{P,cor}_\tau(\gamma)$ when $n=78$. However, their performance becomes worse when $n=390$, suggesting their inconsistency. 
We remark that \citet{On10}'s estimator is designed to estimate the number of \textit{all} factors; hence, in our situation, it might detect the factors $F^7,F^8$ and $F^9$ when both $n$ and $d$ are sufficiently large. While this is not necessarily the case as our model does not satisfy the mathematical assumptions of \cite{On10}, the simulation results reveal that the estimate of \citet{On10}'s increases as $n$ and $d$ increase. 
Turning to the $PC_{p1}$ estimator, it relatively performs well when $L$ is a Wiener process, but it becomes unreliable when $L$ is a NTS process. This is of course not surprising as this estimator is not theoretically justified for our model. 

Overall, we recommend the use of $\hat r^{P,cor}_\tau(\gamma)$ with reasonably large sample sizes, although we need to remember that it tends to underestimate the number of relevant factors with the tuning parameters suggested above. 
Also, despite the lack of theoretical validity, the conventional estimators proposed by \citet{On10} and \citet{Pe19} seem still reliable for moderate sample sizes. 
Interestingly, this parallels to the observation made by \citet{Pe19,Pe20} where \citet{On10}'s estimator is found to be reliable even for high-frequency data. 
Due to the presence of microstructure noise, it would be the most realistic situation that the sample size of intraday high-frequency data is around $n=78$. 
Hence these estimators would be useful in empirical studies for the purpose of robustness check.

\begin{table}[h]
\centering
\small
\caption{Simulation results when $L$ is a Wiener process}\label{table:wiener}
\begin{tabular}{*{11}{r}}
  \hline
$d$ & \multicolumn{2}{c}{$\hat r^{BN}_\tau$} & \multicolumn{2}{c}{$\hat r^{P,cor}_\tau(\gamma)$} & \multicolumn{2}{c}{$PC_{p1}$} & \multicolumn{2}{c}{\citet{Pe19}} & \multicolumn{2}{c}{\citet{On10}} \\ 
  \hline
  & \multicolumn{10}{c}{$n=26$} \\
  100 & 20.00 & (0.00) & 4.63 & (0.17) & 20.00 & (0.00) & 9.90 & (0.11) & 2.37 & (0.02) \\ 
  500 & 20.00 & (0.00) & 4.74 & (0.20) & 20.00 & (0.00) & 4.34 & (0.14) & 3.47 & (0.09) \\ 
  1000 & 20.00 & (0.00) & 4.85 & (0.25) & 20.00 & (0.00) & 4.26 & (0.13) & 3.86 & (0.14) \\ 
  1500 & 20.00 & (0.00) & 4.91 & (0.28) & 20.00 & (0.00) & 4.20 & (0.12) & 4.14 & (0.19) \\ 
  & \multicolumn{10}{c}{$n=78$} \\
  100 & 5.88 & (0.41) & 4.73 & (0.21) & 6.98 & (0.23) & 4.32 & (0.17) & 3.75 & (0.17) \\ 
  500 & 5.81 & (0.46) & 5.29 & (0.49) & 5.73 & (0.45) & 5.52 & (0.58) & 5.36 & (0.55) \\ 
  1000 & 6.01 & (0.49) & 5.44 & (0.58) & 5.29 & (0.36) & 5.53 & (0.62) & 5.62 & (0.67) \\ 
  1500 & 6.10 & (0.55) & 5.51 & (0.62) & 5.03 & (0.28) & 5.54 & (0.65) & 5.72 & (0.73) \\ 
  & \multicolumn{10}{c}{$n=390$} \\
  100 & 4.75 & (0.20) & 5.09 & (0.30) & 6.62 & (0.33) & 4.23 & (0.19) & 5.51 & (0.33) \\ 
  500 & 4.81 & (0.20) & 5.65 & (0.61) & 6.43 & (0.43) & 5.66 & (0.62) & 6.57 & (0.43) \\ 
  1000 & 4.94 & (0.26) & 5.75 & (0.71) & 6.20 & (0.48) & 6.33 & (0.58) & 6.65 & (0.43) \\ 
  1500 & 4.97 & (0.27) & 5.78 & (0.76) & 5.97 & (0.53) & 6.21 & (0.69) & 6.59 & (0.49) \\ 
   \hline
\end{tabular}\vspace{2mm}

\parbox{13cm}{\small
\textit{Note}. For each estimator $\hat r$, this table reports the empirical mean of $\hat r$ and the empirical probability that $\hat r=6$ (in parentheses) based on 1,000 Monte Carlo iterations. 
} 
\end{table}

\begin{table}[ht]
\centering
\small
\caption{Simulation results when $L$ is a NTS process with $\alpha=0.25$}
\label{table:alpha=0.25}
\begin{tabular}{rrrrrrrrrrr}
  \hline
  $d$ & \multicolumn{2}{c}{$\hat r^{BN}_\tau$} & \multicolumn{2}{c}{$\hat r^{P,cor}_\tau(\gamma)$} & \multicolumn{2}{c}{$PC_{p1}$} & \multicolumn{2}{c}{\citet{Pe19}} & \multicolumn{2}{c}{\citet{On10}} \\ 
  \hline
  & \multicolumn{10}{c}{$n=26$} \\
100 & 20.00 & (0.00) & 5.14 & (0.26) & 20.00 & (0.00) & 13.77 & (0.07) & 2.70 & (0.05) \\ 
  500 & 20.00 & (0.00) & 5.26 & (0.41) & 20.00 & (0.00) & 5.14 & (0.37) & 4.06 & (0.22) \\ 
  1000 & 20.00 & (0.00) & 5.34 & (0.48) & 20.00 & (0.00) & 5.05 & (0.37) & 4.47 & (0.32) \\ 
  1500 & 20.00 & (0.00) & 5.38 & (0.50) & 20.00 & (0.00) & 4.97 & (0.35) & 4.69 & (0.35) \\ 
  & \multicolumn{10}{c}{$n=78$} \\
  100 & 10.93 & (0.00) & 5.17 & (0.31) & 14.61 & (0.00) & 5.31 & (0.32) & 4.09 & (0.21) \\ 
  500 & 9.04 & (0.03) & 5.61 & (0.65) & 8.66 & (0.04) & 5.87 & (0.67) & 5.60 & (0.61) \\ 
  1000 & 9.37 & (0.01) & 5.70 & (0.73) & 6.31 & (0.38) & 5.83 & (0.79) & 5.79 & (0.71) \\ 
  1500 & 10.21 & (0.00) & 5.76 & (0.78) & 5.51 & (0.42) & 5.83 & (0.82) & 5.86 & (0.80) \\ 
  & \multicolumn{10}{c}{$n=390$} \\
  100 & 7.76 & (0.14) & 5.45 & (0.38) & 15.00 & (0.00) & 5.13 & (0.33) & 5.53 & (0.31) \\ 
  500 & 5.49 & (0.35) & 5.83 & (0.70) & 12.02 & (0.00) & 6.07 & (0.69) & 6.41 & (0.52) \\ 
  1000 & 5.28 & (0.34) & 5.89 & (0.81) & 9.12 & (0.04) & 6.42 & (0.59) & 6.57 & (0.53) \\ 
  1500 & 5.19 & (0.32) & 5.91 & (0.85) & 7.59 & (0.18) & 6.31 & (0.67) & 6.54 & (0.57) \\ 
   \hline
\end{tabular}\vspace{2mm}

\parbox{13cm}{\small
\textit{Note}. For each estimator $\hat r$, this table reports the empirical mean of $\hat r$ and the empirical probability that $\hat r=6$ (in parentheses) based on 1,000 Monte Carlo iterations. 
} 
\end{table}

\begin{table}[ht]
\centering
\small
\caption{Simulation results when $L$ is a NTS process with $\alpha=0.50$}
\label{table:alpha=0.5}
\begin{tabular}{rrrrrrrrrrr}
  \hline
  $d$ & \multicolumn{2}{c}{$\hat r^{BN}_\tau$} & \multicolumn{2}{c}{$\hat r^{P,cor}_\tau(\gamma)$} & \multicolumn{2}{c}{$PC_{p1}$} & \multicolumn{2}{c}{\citet{Pe19}} & \multicolumn{2}{c}{\citet{On10}} \\ 
  \hline
  & \multicolumn{10}{c}{$n=26$} \\
100 & 20.00 & (0.00) & 5.03 & (0.26) & 20.00 & (0.00) & 13.01 & (0.09) & 2.59 & (0.04) \\ 
  500 & 20.00 & (0.00) & 5.16 & (0.36) & 20.00 & (0.00) & 5.04 & (0.32) & 3.94 & (0.19) \\ 
  1000 & 20.00 & (0.00) & 5.26 & (0.44) & 20.00 & (0.00) & 4.95 & (0.31) & 4.41 & (0.28) \\ 
  1500 & 20.00 & (0.00) & 5.33 & (0.48) & 20.00 & (0.00) & 4.87 & (0.30) & 4.64 & (0.33) \\ 
  & \multicolumn{10}{c}{$n=78$} \\
  100 & 10.10 & (0.00) & 5.08 & (0.30) & 13.62 & (0.00) & 5.08 & (0.30) & 3.99 & (0.21) \\ 
  500 & 8.82 & (0.03) & 5.52 & (0.59) & 8.47 & (0.05) & 5.81 & (0.65) & 5.58 & (0.59) \\ 
  1000 & 9.22 & (0.03) & 5.65 & (0.71) & 6.34 & (0.38) & 5.80 & (0.76) & 5.73 & (0.71) \\ 
  1500 & 9.94 & (0.01) & 5.73 & (0.77) & 5.55 & (0.41) & 5.81 & (0.80) & 5.86 & (0.80) \\ 
  & \multicolumn{10}{c}{$n=390$} \\
  100 & 7.32 & (0.18) & 5.39 & (0.35) & 13.57 & (0.00) & 4.98 & (0.31) & 5.46 & (0.32) \\ 
  500 & 5.51 & (0.35) & 5.80 & (0.68) & 11.68 & (0.00) & 5.97 & (0.69) & 6.41 & (0.51) \\ 
  1000 & 5.33 & (0.33) & 5.88 & (0.79) & 9.17 & (0.04) & 6.42 & (0.57) & 6.51 & (0.52) \\ 
  1500 & 5.23 & (0.33) & 5.89 & (0.85) & 7.85 & (0.15) & 6.29 & (0.68) & 6.48 & (0.58) \\ 
   \hline
\end{tabular}\vspace{2mm}

\parbox{13cm}{\small
\textit{Note}. For each estimator $\hat r$, this table reports the empirical mean of $\hat r$ and the empirical probability that $\hat r=6$ (in parentheses) based on 1,000 Monte Carlo iterations. 
} 
\end{table}

\begin{table}[ht]
\centering
\small
\caption{Simulation results when $L$ is a NTS process with $\alpha=0.75$}
\label{table:alpha=0.75}
\begin{tabular}{rrrrrrrrrrr}
  \hline
  $d$ & \multicolumn{2}{c}{$\hat r^{BN}_\tau$} & \multicolumn{2}{c}{$\hat r^{P,cor}_\tau(\gamma)$} & \multicolumn{2}{c}{$PC_{p1}$} & \multicolumn{2}{c}{\citet{Pe19}} & \multicolumn{2}{c}{\citet{On10}} \\ 
  \hline
  & \multicolumn{10}{c}{$n=26$} \\
100 & 20.00 & (0.00) & 4.86 & (0.22) & 20.00 & (0.00) & 12.01 & (0.10) & 2.52 & (0.04) \\ 
  500 & 20.00 & (0.00) & 5.07 & (0.32) & 20.00 & (0.00) & 4.83 & (0.26) & 3.77 & (0.16) \\ 
  1000 & 20.00 & (0.00) & 5.13 & (0.37) & 20.00 & (0.00) & 4.73 & (0.25) & 4.22 & (0.23) \\ 
  1500 & 20.00 & (0.00) & 5.20 & (0.41) & 20.00 & (0.00) & 4.67 & (0.24) & 4.49 & (0.28) \\ 
  & \multicolumn{10}{c}{$n=78$} \\
  100 & 8.75 & (0.04) & 4.97 & (0.28) & 11.68 & (0.00) & 4.82 & (0.26) & 3.83 & (0.19) \\ 
  500 & 8.34 & (0.06) & 5.49 & (0.59) & 8.06 & (0.08) & 5.71 & (0.65) & 5.53 & (0.60) \\ 
  1000 & 8.88 & (0.04) & 5.60 & (0.67) & 6.41 & (0.36) & 5.71 & (0.71) & 5.75 & (0.71) \\ 
  1500 & 9.80 & (0.01) & 5.66 & (0.71) & 5.62 & (0.40) & 5.72 & (0.75) & 5.82 & (0.77) \\ 
  & \multicolumn{10}{c}{$n=390$} \\
  100 & 6.58 & (0.23) & 5.34 & (0.34) & 11.27 & (0.00) & 4.80 & (0.28) & 5.49 & (0.33) \\ 
  500 & 5.61 & (0.33) & 5.77 & (0.67) & 10.72 & (0.01) & 5.89 & (0.66) & 6.49 & (0.47) \\ 
  1000 & 5.45 & (0.33) & 5.84 & (0.76) & 9.19 & (0.05) & 6.39 & (0.57) & 6.57 & (0.49) \\ 
  1500 & 5.33 & (0.35) & 5.86 & (0.81) & 8.03 & (0.12) & 6.28 & (0.68) & 6.56 & (0.51) \\ 
   \hline
\end{tabular}\vspace{2mm}

\parbox{13cm}{\small
\textit{Note}. For each estimator $\hat r$, this table reports the empirical mean of $\hat r$ and the empirical probability that $\hat r=6$ (in parentheses) based on 1,000 Monte Carlo iterations. 
} 
\end{table}

\clearpage

\section{Proofs}\label{sec:proof}

We begin by introducing additional notation used in the proofs. 
\begin{itemize}

\item For a random variable $\xi$ and $p>0$, we set $\|\xi\|_p:=(\E[|\xi|^p])^{1/p}$. 

\item For a random matrix $X$ and $p>0$, we set $\|X\|_{L^p(\sch^\infty)}:=\|\|X\|\|_p$.

\item For a process $V=(V_t)_{t\in[0,1]}$ taking values in $\mathbb R^d$ or $\mathbb R^{d\times d}$, we set $\Delta^n_iV:=V_{i/n}-V_{(i-1)/n}$ for $i=1,\dots,n$. If $V$ takes values in $\mathbb R^d$, we further set $\Delta^nV:=(\Delta^n_1V,\dots,\Delta^n_nV)$ which is a $d\times n$ random matrix. Further, given two $d$-dimensional processes $U$ and $V$, we set
\[
{[U,V]}^n_1:=\Delta^nU(\Delta^nV)^\top=\sum_{i=1}^n\Delta^n_iU(\Delta^n_iV)^\top.
\]

\item For a $d\times n$ matrix $A$, we write its singular values as $\mf s_1(A)\geq\cdots\geq\mf s_{d\wedge n}(A)$. 

\item For two numbers or random variables $\xi$ and $\eta$, we write $\xi\lesssim\eta$ if there is a positive universal constant $C$ such that $\xi\leq C\eta$. 

\end{itemize}

\subsection{Proofs for Section \ref{sec:main}}\label{sec:proof2}


\subsubsection{Proof of Theorem \ref{cont-bound}}\label{sec:pr-cont}

The proof relies on two general martingale inequalities. 
The first one is the Burkholder--Davis--Gundy inequalities with sharp constants:
\begin{lemma}\label{sharp-BDG}
Let $M=(M_t)_{t\in[0,1]}$ be a continuous local martingale with $M_0=0$. For any $p\in[1,\infty)$ and $t\in[0,1]$, 
\[
\left\|M_t\right\|_p\leq 2\sqrt{p}\left\|\langle M,M\rangle_t^{1/2}\right\|_p.
\]
\end{lemma}

\begin{proof}
This immediately follows from \cite[Theorem A]{CK1991} and a standard localization argument.  
\end{proof}

The second one is a Bernstein-type inequality for discrete-time martingales:
\begin{lemma}\label{bernstein}
Let $(\xi_i)_{i=1}^n$ be a martingale difference sequence with respect to a filtration $(\mcl{G}_i)_{i=0}^n$. 
Suppose that there are constants $B,K>0$ such that $\sum_{i=1}^n\expe{|\xi_i|^p\mid\mcl{G}_{i-1}}\leq p!K^{p-2}B^2/2$ a.s.~for any integer $p\geq 2$. Then, for any $u\geq0$,
\[
P\bra{\abs{\sum_{i=1}^n\xi_i}\geq u}\leq2\exp\bra{-\frac{u^2}{2(B^2+Ku)}}.
\]
\end{lemma}

\begin{proof}
See \cite[Lemma 8.9]{vdG00}.
\end{proof}

In addition, we use the Bernstein--Orlicz norm introduced in \cite{vdGL13} to simplify our argument. 
For each $L>0$, define the function $\Psi_L:[0,\infty)\to[0,\infty)$ by
\[
\Psi_L(u)=\exp\bra{\frac{\sqrt{1+2Lu}-1}{L}}^2-1,\qquad u\geq0.
\]
The \textit{$L$-Bernstein--Orlicz norm} of a random variable $\xi$ is then defined as
\[
\|\xi\|_{\Psi_L}:=\inf\cbra{c>0:\E\sbra{\Psi_L\bra{\frac{|\xi|}{c}}}\leq1}.
\]
The following result taken from \cite{We17} is useful.
\begin{lemma}\label{lem:bo}
Let $\xi$ be a random variable. Suppose that there are constants $B,K>0$ such that
\[
P(|\xi|>u)\leq2\exp\bra{-\frac{u^2}{2(B^2+Ku)}}
\]
for all $u>0$. Then $\|\xi\|_{\Psi_{\sqrt 3L}}\leq\sqrt 3\tau$ with $L=\sqrt{2}K/B$ and $\tau=\sqrt{8}B$. 
\end{lemma}

\begin{proof}
This follows from Lemma 2.2 and Proposition 2.1 in \cite{We17}. 
\end{proof}

\begin{proof}[\bf Proof of Theorem \ref{cont-bound}]
We divide the proof into two steps. 

\noindent\textbf{Step 1}. Define two processes $D=(D_{t})_{t\in[0,1]}$ and $M=(M_{t})_{t\in[0,1]}$ as
\ba{
D_{t}=\int_0^tb_{s}ds,\quad\text{and}\quad
M_{t}=\int_0^t\sigma_{s}dW_s\qquad
\text{for }t\in[0,1].
}
We have
\ba{
\norm{{[Z,Z]}^n_1-[Z,Z]_1}
\leq\norm{{[D,D]}^n_1}+2\norm{{[D,M]}^n_1}+\norm{R},
}
where $R:={[M,M]}^n_1-[M,M]_1$. 
We have by \ref{ass:SH}
\ba{
\norm{{[D,D]}^n_1}
\leq\sum_{i=1}^n\norm{\Delta^n_iD(\Delta^n_iD)^\top}
=\sum_{i=1}^n|\Delta^n_iD|^2
\leq \Lambda\frac{d}{n}
}
and
\ba{
\norm{{[M,M]}^n_1}
\leq \|R\|+\norm{[M,M]_1}
\leq \|R\|+\Lambda.
}
Hence we obtain
\ba{
\norm{{[D,M]}^n_1}
&\leq\norm{\Delta^nD}\norm{\Delta^nM}
=\sqrt{\norm{{[D,D]}^n_1}\norm{{[M,M]}^n_1}}\\
&\leq\sqrt{\Lambda\frac{d}{n}\|R\|}
+\Lambda\sqrt{\frac{d}{n}}
\leq\frac{\Lambda}{2}\frac{d}{n}+\frac{\|R\|}{2}+\Lambda\sqrt{\frac{d}{n}}.
}
Consequently, it holds that
\ben{\label{eq:basic-est}
\norm{{[Z,Z]}^n_1-[Z,Z]_1}
\leq2\Lambda\frac{d}{n}+2\Lambda\sqrt{\frac{d}{n}}
+2\norm{R}.
}
In the next step, we will show that there is a universal constant $C>0$ such that, for any $u>0$,
\ben{\label{eq:D-est}
\norm{R}
\leq C\Lambda\bra{\frac{d+u}{n}+\sqrt{\frac{d+u}{n}}}
}
with probability at least $1-2e^{-u}$. Combining this with \eqref{eq:basic-est} gives the desired result.

\noindent\textbf{Step 2}. 
We prove \eqref{eq:D-est} by the so-called $\eps$-net argument. 
First, by Lemmas 5.2 and 5.4 in \cite{Ve12} with $\eps=1/4$, there is a set $\mathcal N\subset\{x\in\mathbb R^d:|x|=1\}$ such that $\#\mathcal N\leq9^{d}$ and
\ben{\label{eq:net-arg}
\norm{A}\leq2\max_{x\in\mathcal N}|x^\top Ax|
}
for any symmetric $d\times d$ matrix $A$. Here, $\#\mcl N$ denotes the cardinality of $\mcl N$. 
Next, fix $x\in\mcl N$ and define the process $M^x=(M_{t}^x)_{t\in[0,1]}$ by $M^x_{t}=x^\top M_{t}$, $t\in[0,1]$. For every $i=1,\dots,n$, set
\[
\xi_i:=x^\top\cbra{(\Delta_i^nM)(\Delta_i^nM)^\top-\Delta^n_i[M,M]} x
=(\Delta_i^nM^x)^2-\Delta^n_i[M^x,M^x].
\]
Then we have
\ba{
x^\top R x=\sum_{i=1}^n\xi_i.
}
We are going to apply Lemma \ref{bernstein}. First, by construction, $(\xi_{i})_{i=1}^n$ is a martingale difference with respect to $(\mathcal{F}_{t_i})_{i=0}^n$. 
Next, for any integer $p\geq2$ and $i=1,\dots,n$, we have 
\ba{
\bra{\E[|\xi_i|^p\mid \mcl F_{t_{i-1}}]}^{1/p}
&\leq\bra{\E[(\Delta_i^nM^x)^{2p}\mid \mcl F_{t_{i-1}}]}^{1/p}
+\bra{\E\sbra{\abs{\Delta^n_i[M^x,M^x]}^p\mid \mcl F_{t_{i-1}}}}^{1/p}\\
&\leq\{(2\sqrt{2p})^2+1\}\bra{\E[(\Delta_i^n[M^x,M^x])^{p}\mid \mcl F_{t_{i-1}}]}^{1/p}\\
&\leq9p\Lambda/n,
}
where the second line follows from Lemma \ref{sharp-BDG} and the last line from \ref{ass:SH}. 
Thus, using Stirling's formula, we obtain
\[
\sum_{i=1}^n\E\sbra{|\xi_{i}|^p\mid\mcl{F}_{t_{i-1}}}\leq \bra{\frac{9\Lambda}{n}}^{p}\frac{e^p}{\sqrt{2\pi p}}p!
\leq\frac{p!}{2}\bra{\frac{K_0\Lambda}{n}}^{p-2}\bra{\frac{B_0\Lambda}{\sqrt n}}^2,
\]
where $B_0=9e(2/\pi)^{1/4}$ and $K_0=9e$. 
Consequently, we obtain by Lemma \ref{bernstein}
\ben{\label{eq:bernstein}
P(|x^\top R x|\geq u)\leq2\exp\lpa-\frac{u^2}{2(B_0^2\Lambda^2/n+(K_0\Lambda/n) u)}\rpa
}
for all $u\geq0$. This and Lemma \ref{lem:bo} imply that 
\ben{\label{eq:bo-bound}
\norm{x^\top R x}_{\Psi_{L_0/\sqrt n}}\leq \tau_0\frac{\Lambda}{\sqrt n}\qquad\text{with }L_0=\sqrt{6}K_0/B_0\quad\text{and}\quad \tau_0=\sqrt{24}B_0. 
}

Finally, combining \eqref{eq:net-arg} and \eqref{eq:bo-bound} with Lemma 4 in \cite{vdGL13}, we obtain
\ba{
P\bra{\norm{R}\geq \frac{\tau_0}{2}\frac{\Lambda}{\sqrt n}\sbra{\sqrt{\log(1+\#\mcl N)}+\frac{L_0}{2\sqrt n}\log(1+\#\mcl N)+\sqrt u+\frac{L_0u}{2\sqrt n}}}\leq 2e^{-u}
}
for all $u>0$. Since $\log(1+\#\mcl N)\leq\log2+d\log 9\leq4d$, this completes the proof of \eqref{eq:D-est}. 
\end{proof}

\subsubsection{Proof of Theorem \ref{jump-bound2}}

We begin by a simple lemma that allows us to impose extra conditions on $\delta$. 
\begin{lemma}\label{lemma:bdd-jump}
Under \ref{ass:SH}, there is a predictable function $\tilde\delta:\Omega\times[0,1]\times E\to\mathbb R^d$ satisfying the following conditions:
\begin{enumerate}[label=(\roman*)]

\item $\kappa(\delta)\star(\mu-\nu)=\kappa(\tilde\delta)\star(\mu-\nu)$ and $\kappa'(\delta)\star\mu=\kappa'(\tilde\delta)\star\mu$. 

\item $|\tilde\delta(\omega,s,z)|^2\leq\Lambda$ and $\tilde\delta_j(\omega,s,z)^2\leq\min\{\delta_j(\omega,s,z)^2,\Lambda\gamma_j(z)^2\}$ for all $(\omega,s,z)\in\Omega\times[0,1]\times E$ and $j=1,\dots,d$. Here, $\tilde\delta_j$ denotes the $j$-th component function of $\tilde\delta$. 

\end{enumerate}
\end{lemma}

\begin{proof}
Define $\tilde\delta:=\delta1_{\{|\delta|^2\leq\Lambda\}}$. By construction we have $|\tilde\delta|^2\leq\Lambda$. In particular, $\tilde\delta_j^2\leq\Lambda$ for every $j=1,\dots,d$. Now, if $(\omega,s,z)\in\Omega\times[0,1]\times E$ satisfies $\delta_j(\omega,s,z)^2>1$, then $\gamma_j(z)^2\geq1$, so $\tilde\delta_j(\omega,s,z)^2\leq\Lambda\gamma_j(z)^2$. Otherwise, $\tilde\delta_j(\omega,s,z)^2=\delta_j(\omega,s,z)^2\wedge1\leq\gamma_j(z)^2\leq\Lambda\gamma_j(z)^2$. 
Moreover, we evidently have $|\tilde\delta_j|\leq|\delta_j|$. Hence $\tilde\delta$ satisfies condition (ii). 

It remains to check condition (i). By Proposition 1.14 in \cite[Chapter II]{JS}, there are a thin random set $\mcl T$ and an $E$-valued optional process $\theta=(\theta_s)_{s\in[0,1]}$ such that
\ben{\label{eq:prm}
\mu=\sum_{s\in[0,1]}1_{\mcl T}(s)\eps_{(s,\theta_s)},
}
where $\eps_a$ denotes the Dirac measure at point $a$. Then we have $\Delta Z_s=\delta(s,\theta_s)1_{\mcl T}(s)$ for all $s$ by the definition of (stochastic or ordinary) integrals with respect to $\mu$. Meanwhile, we have $\sup_{s\in[0,1]}|\Delta Z_s|^2\leq\Lambda$ by \eqref{eq:jump-bound} and \ref{ass:SH}. Consequently, $\delta(s,\theta_s)1_{\mcl T}(s)=\tilde\delta(s,\theta_s)1_{\mcl T}(s)$ for all $s\in[0,1]$. Now condition (i) follows from the definition of (stochastic or ordinary) integrals with respect to $\mu$. 
\end{proof}

Thanks to Lemma \ref{lemma:bdd-jump}, when assuming \ref{ass:SH}, we may also assume that
\ben{\label{eq:bdd-jump}
|\delta(\omega,s,z)|^2\leq\Lambda\qquad\text{and}\qquad
\delta_j(\omega,s,z)^2\leq\Lambda\gamma_j(z)^2
}
for all $(\omega,s,z)\in\Omega\times[0,1]\times E$ and $j=1,\dots,d$.

Before proceeding to the formal proof, we give a brief description of our proof strategy.  
First, under \eqref{eq:bdd-jump}, we can define the process $J=(J_t)_{t\in[0,1]}$ by
\ben{\label{def:J}
J_t=\delta\star(\mu-\nu)_t,\qquad t\in[0,1].
}
It is not difficult to see that the major problem is to bound ${[J,J]}^n_1$; see Step 1 in the proof of Theorem \ref{jump-bound2} for the formal argument. 

To bound ${[J,J]}^n_1$, we follow a similar strategy to \citet{Ru99}. 
For $p=\log (2/\alpha)$, we will derive the following bound corresponding to Eq.(2.1) in \cite{Ru99}:
\besn{\label{eq:key}
&\norm{{[J,J]}^n_1-[J,J]_1}_{L^p(\sch^\infty)}\\
&\leq C\log(d/\alpha)\bra{\norm{\sqrt{\max_{1\leq i\leq n}|\Delta^n_iJ|^2\norm{{[J,J]}^n_1}}}_p
+\norm{\sqrt{\max_{1\leq i\leq n}\norm{\Delta^n_i[J,J]}\norm{[J,J]_1}}}_p},
}
where $C$ is a universal constant. Since
\ba{
&C\log(d/\alpha)\norm{\sqrt{\max_{1\leq i\leq n}|\Delta^n_iJ|^2\norm{{[J,J]}^n_1}}}_p\\
&\leq C\log(d/\alpha)\norm{\sqrt{\max_{1\leq i\leq n}|\Delta^n_iJ|^2\norm{{[J,J]}^n_1-[J,J]_1}}}_p
+C\log(d/\alpha)\norm{\sqrt{\max_{1\leq i\leq n}|\Delta^n_iJ|^2\norm{[J,J]_1}}}_p\\
&\leq \frac{C^2\log^2(d/\alpha)}{2}\norm{\max_{1\leq i\leq n}|\Delta^n_iJ|^2}_p+\frac{1}{2}\norm{{[J,J]}^n_1-[J,J]_1}_{L^p(\sch^\infty)}
+C\log(d/\alpha)\norm{\sqrt{\max_{1\leq i\leq n}|\Delta^n_iJ|^2\norm{[J,J]_1}}}_p,
}
we obtain 
\bm{
\norm{{[J,J]}^n_1-[J,J]_1}_{L^p(\sch^\infty)}
\leq C^2\log^2(d/\alpha)\norm{\max_{1\leq i\leq n}|\Delta^n_iJ|^2}_p\\
+2C\log(d/\alpha)\bra{\norm{\sqrt{\max_{1\leq i\leq n}|\Delta^n_iJ|^2\norm{[J,J]_1}}}_p
+\norm{\sqrt{\max_{1\leq i\leq n}\norm{\Delta^n_i[J,J]}\norm{[J,J]_1}}}_p}.
}
In the setting of \cite{Ru99}, quantities corresponding to $\max_{i}|\Delta^n_iJ|^2$ and $\max_i\norm{\Delta^n_i[J,J]}$ become smaller as $n\to\infty$, so this gives a concentration inequality for ${[J,J]}^n_1$. 
By contrast, these quantities cannot be small in our setting because $J$ is discontinuous, so this does not give a vanishing error bound for ${[J,J]}^n_1-[J,J]_1$. However, we can still show that they are of order $O_p(\Lambda(1+d/n))$ up to a log factor, so we can deduce an adequate bound for $\|{[J,J]}^n_1\|$ from this argument and Markov's inequality. 


A key technical step in the above argument is the derivation of \eqref{eq:key}. This follows from the following bound along with a few elementary matrix inequalities:
\ba{
\norm{{[J,J]}^n_1-[J,J]_1}_{L^p(\sch^\infty)}
\lesssim \log(d/\alpha)\norm{\bra{\sum_{i=1}^n\cbra{(\Delta^n_iJ)(\Delta^n_iJ)^\top-\Delta^n_i[J,J]}^2}^{1/2}}_{L^p(\sch^\infty)}.
}
In \cite{Ru99}, the corresponding bound follows from a standard symmetrization argument and the non-commutative Khinchine inequalities. In our setting, the former is not applicable since ${[J,J]}^n_1-[J,J]_1$ is not a sum of independent matrices but a matrix martingale, so some modification is necessary. It turns out that a symmetrization inequality is still available in the current situation with an extra log factor; see Remark \ref{rem:UMD}. 
With an additional effort, we can even avoid the extra log factor. This is the content of the following theorem:
\begin{theorem}[Burkholder--Gundy type inequalities for a matrix martingale]\label{matrix-burkholder}
Let $\mcl Y=(Y_i)_{i=1}^n$ be an $\mathbb R^{d\times d}$-valued martingale difference sequence with respect to a filtration $(\mcl G_i)_{i=0}^n$. Then, for any $p\in(1,\infty)$,
\ben{\label{eq:BG}
\norm{\sum_{i=1}^nY_i}_{L^p(\sch^\infty)}
\leq Cp\bra{\frac{\log d}{p-1}+1}
\left\{\norm{\bra{\sum_{i=1}^nY_i^\top Y_i}^{1/2}}_{L^p(\sch^\infty)}\vee
\norm{\bra{\sum_{i=1}^nY_i Y_i^\top}^{1/2}}_{L^p(\sch^\infty)}\right\},
}
where $C>0$ is a universal constant. 
\end{theorem}

We give a proof of this theorem in Appendix \ref{sec:mBDG}. 
\begin{rmk}[Relation to the non-commutative Burkholder--Gundy inequalities]\label{rem:BG}
When $p\geq 2\vee\log d$, Theorem \ref{matrix-burkholder} can be derived from the so-called non-commutative Burkholder--Gundy inequalities. 
To see this, in this remark we will freely use the standard terminology of non-commutative probability theory; see \cite[Chapter 14]{Pi16} for details. 

Let $\mcl M=L^\infty(\Omega,\mcl F,P;\mathbb C^{d\times d})$, i.e.~the space of $d\times d$ complex random matrices on $(\Omega,\mcl F,P)$ whose entries are essentially bounded. Then, identifying $Y\in\mcl M$ with the operator $\xi\mapsto Y\xi$ on the complex Hilbert space $L^2(\Omega,\mcl F,P;\mathbb C^{d})$, we can regard $\mcl M$ as a von Neumann algebra. 
We can easily check that the map $\tau:\mcl M\to\mathbb C$ defined as $\tau(Y)=\E[\trace Y]$ is a normal faithful finite trace on $\mcl M$. 
Then, the norm of the non-commutative $L^p$ space $L^p(\mcl M,\tau)$ associated with $(\mcl M,\tau)$ is given by $\|Y\|_{L^p(\mcl M,\tau)}=(\E[\|Y\|_{\sch^p}^p])^{1/p}$ for $Y\in\mcl M$, where $\|\cdot\|_{\sch^p}$ denotes the Schatten $p$-norm (cf.~Appendix \ref{sec:mBDG}).  
Now, in the setting of Theorem \ref{matrix-burkholder}, $\mcl Y$ is a non-commutative martingale difference sequence with respect to the filtration $(L^\infty(\Omega,\mcl G_i,P;\mathbb C^{d\times d}))_{i=0}^n$. 
Therefore, by the non-commutative Burkholder--Gundy inequalities (e.g.~Theorem 5.1 in \cite{Ra02}), 
\ba{
\left(\E\norm{\sum_{i=1}^nY_i}^q_{\sch^q}\right)^{1/q}
&\leq C_q
\left\{\left(\E\norm{\lpa\sum_{i=1}^nY_i^\top Y_i\rpa^{1/2}}^q_{\sch^q}\right)^{1/q}\vee
\left(\E\norm{\lpa\sum_{i=1}^nY_i Y_i^\top\rpa^{1/2}}^q_{\sch^q}\right)^{1/q}\right\}
}
for any $q\geq2$, where $C_q>0$ is a constant depending only on $q$. It is known that $C_q\lesssim q$ (see Remark 5.4 in \cite{Ra02}). Since $\|A\|\leq\|A\|_{\sch^q}\leq d^{1/q}\|A\|$ for any $A\in\mathbb R^{d\times d}$ and $q\geq1$, the claim of the theorem follows from the above inequality with $q=p$ as long as $p\geq2\vee\log d$. 

Since we can get a slightly sharper bound in our main results by applying Theorem \ref{matrix-burkholder} with $p$ independent of $d$, we prove the version stated above later. 
Although our proof is essentially a rephrasing of the proof of \cite[Theorem 5.1]{Ra02}, it is written without explicit use of language from non-commutative probability theory and may be of independent interest. 

Finally, we refer to the work by \cite{MJCFT14} that establishes Burkholder--Gundy type inequalities for a random matrix via the method of exchangeable pairs (see Theorem 7.1 ibidem). However, this result seems inapplicable to matrix martingales. 
\end{rmk}

%



It remains to bound $\max_{i}|\Delta^n_iJ|^2$ and $\max_i\norm{\Delta^n_i[J,J]}$. Under \ref{ass:SH}, the latter can be bounded by $\Lambda$ because
\ben{\label{eq:jqv-bound}
\max_{1\leq i\leq n}\norm{\Delta^n_i[J,J]}=\max_{1\leq i\leq n}\norm{\sum_{s\in(t_{i-1},t_i]}(\Delta J_s)(\Delta J_s)^\top}\leq\norm{[J,J]_1}\leq\norm{[Z,Z]_1}
}
by Weyl's inequality. To bound the former, we will use the following Rosenthal type inequalities for multi-dimensional locally square integrable martingales with dimension-free constants. 
Given a locally square integrable martingale $M$ in $\mathbb R^d$, we write $\langle M\rangle=\sum_{j=1}^d\langle M^j,M^j\rangle$. 
%
\begin{theorem}\label{thm:rosenthal}
Let $M=(M_t)_{t\in[0,1]}$ be a local martingale in $\mathbb R^d$ with bounded jumps and $M_0=0$ (hence it is locally square integrable). Suppose that $\langle M\rangle$ is continuous. 
Then, there is a universal constant $C>0$ such that
\ben{\label{eq:rosenthal}
\norm{|M_t|}_p\leq C\bra{\sqrt p\norm{\langle M\rangle_t^{1/2}}_p+p\norm{\sup_{s\in(0,t]}|\Delta M_s|}_p}
}
for any $p\in[2,\infty)$ and $t\in(0,1]$. 
\end{theorem}
We give a proof of this theorem in Appendix \ref{sec:rosenthal}. 
\begin{rmk}[Alternative version]
The bound \eqref{eq:rosenthal} immediately follows from the Rosenthal type inequality of \cite{ReTi03} and Theorem 3.1 in \cite{KaSz91} if both $\sqrt p$ and $p$ on the right hand side are replaced by $p/\log p$, and this version already yields an almost satisfactory bound for $\max_{i}|\Delta^n_iJ|^2$. 
We use the above version because it slightly weakens the assumptions on $d$ and $n$ necessary in Lemmas \ref{lemma:cov-eigen} and \ref{lemma:cor-eigen}. 
Besides, Theorem \ref{thm:rosenthal} may hold for any locally square integrable martingale $M$ in $\mathbb R^d$ with $M_0=0$ because this is the case if we do not care about the dependence on $p$ as above. The ``extra'' assumptions imposed in the theorem are convenient for the proof and not problematic for our applications. 
\end{rmk}

Finally, we prove two elementary inequalities for matrices. 
\begin{lemma}\label{lem:mat1}
Let $A_1,\dots,A_n,B_1,\dots,B_n$ are $d\times d$ symmetric matrices. Then
\[
\norm{\sum_{i=1}^n(A_i-B_i)^2}
\leq2\norm{\sum_{i=1}^nA_i^2}
+2\norm{\sum_{i=1}^nB_i^2}.
\] 
\end{lemma}

\begin{proof}
Noting that $\|U\|=\|U^\top\|$ for any $U\in\mathbb R^{d\times d}$, we obtain
\ba{
\norm{\sum_{i=1}^n(A_i-B_i)^2}
&\leq\norm{\sum_{i=1}^nA_i^2}
+\norm{\sum_{i=1}^nB_i^2}
+2\norm{\sum_{i=1}^nA_iB_i}.
}
Let
\ben{\label{eq:expand}
A=\begin{pmatrix}
A_1 & \cdots & A_n
\end{pmatrix}
\in\mathbb R^{d\times nd}
\quad\text{and}\quad
B=\begin{pmatrix}
B_1 \\ \vdots \\ B_n
\end{pmatrix}\in\mathbb R^{nd\times d}.
}
Then
\ba{
\norm{\sum_{i=1}^nA_iB_i}
&=\norm{AB}
\leq\norm{A}\norm{B}
=\sqrt{\norm{AA^\top}\norm{B^\top B}}\\
&=\sqrt{\norm{\sum_{i=1}^nA_i^2}\norm{\sum_{i=1}^nB_i^2}}
\leq\frac{1}{2}\lpa\norm{\sum_{i=1}^nA_i^2}+\norm{\sum_{i=1}^nB_i^2}\rpa.
}
Combining these two bounds gives the desired result. 
\end{proof}

\begin{lemma}\label{lem:mat2}
Let $A_1,\dots,A_n$ are $d\times d$ positive semidefinite matrices. Then
\[
\norm{\sum_{i=1}^nA_i^2}
\leq\max_{1\leq i\leq n}\|A_i\|\norm{\sum_{i=1}^nA_i}.
\] 
\end{lemma}

\begin{proof}
Write $K=\max_{1\leq i\leq n}\|A_i\|$. It is not difficult to check that $KA_i-A_i^2$ is positive semidefinite for all $i=1,\dots,n$. Hence $K\sum_{i=1}^nA_i-\sum_{i=1}^nA_i^2$ is also positive semidefinite. Now the desired result follows from Weyl's inequality.
\end{proof}

\begin{proof}[\bf Proof of Theorem \ref{jump-bound2}]
We divide the proof into two steps. 

\noindent\textbf{Step 1}. Define the process $J=(J_t)_{t\in[0,1]}$ by \eqref{def:J}. Also, define the process $\wt D=(\wt D_{t})_{t\in[0,1]}$ by
\ba{
\wt D_{t}&=\int_0^t\wt b_sds\quad\text{with}\quad\wt b_s=b_s+\int_E\kappa'\bra{\delta(s,z)}\levy(dz),
}
and set $X=\wt D+M$. Then we have
\ban{
\norm{{[Z,Z]}^n_1}&=\|\Delta^nZ\|^2
=\|\Delta^nX+\Delta^nJ\|^2
\nonumber\\
&\leq2(\|\Delta^nX\|^2+\|\Delta^nJ\|^2)
=2\bra{\norm{{[X,X]}^n_1}+\norm{{[J,J]}^n_1}}.\label{xx-bound}
}
For $j=1,\dots,d$, we have by \eqref{eq:bdd-jump}
\[
\abs{\int_E\delta_j(s,z)1_{\{|\delta_j(s,z)|>1\}}\levy(dz)}
\leq\int_E\delta_j(s,z)^2\levy(dz)
\leq\Lambda\ol\gamma_2.
\]
Thus, $|\wt b_s|_\infty\leq\Lambda(1+\ol\gamma_2)$. 
Hence, applying Theorem \ref{cont-bound} with $u=\log(4/\alpha)$, we have
\be{
\norm{{[X,X]}^n_1}\lesssim \Lambda\bra{1+\ol\gamma_2}\bra{1+\frac{d+\log(1/\alpha)}{n}}
}
with probability at least $1-\alpha/2$. Therefore, we complete the proof once we show that the bound
\ben{\label{IA}
\norm{{[J,J]}^n_1}
\lesssim \Lambda\log^2(d/\alpha)\bra{\frac{d\log n}{n}\ol{\gamma}_2+\log^2n}
}
holds with probability at least $1-\alpha/2$.

\noindent\textbf{Step 2}. Let $p=\log (2/\alpha)$. Observe that $p\geq2$ and
\[
{[J,J]}^n_1-[J,J]_1=\sum_{i=1}^n\cbra{(\Delta^n_iJ)(\Delta^n_iJ)^\top-\Delta^n_i[J,J]}.
\]
Therefore, we have by Theorem \ref{matrix-burkholder}
\ba{
\norm{{[J,J]}^n_1-[J,J]_1}_{L^p(\sch^\infty)}
\leq c_0\log(d/\alpha)\norm{\bra{\sum_{i=1}^n\cbra{(\Delta^n_iJ)(\Delta^n_iJ)^\top-\Delta^n_i[J,J]}^2}^{1/2}}_{L^p(\sch^\infty)},
}
where $c_0>0$ is a universal constant. 
Noting that $\|A^{1/2}\|=\|A\|^{1/2}$ for any positive semidefinite $d\times d$ matrix $A$, we have by Lemmas \ref{lem:mat1} and \ref{lem:mat2}
\ba{
&\norm{\lpa\sum_{i=1}^n\{(\Delta^n_iJ)(\Delta^n_iJ)^\top-\Delta^n_i[J,J]\}^2\rpa^{1/2}}\\
&\leq\sqrt{2\max_{1\leq i\leq n}|\Delta^n_iJ|^2\norm{{[J,J]}^n_1}}
+\sqrt{2\max_{1\leq i\leq n}\norm{\Delta^n_i[J,J]}\norm{[J,J]_1}}.
}
Combining these bounds with \eqref{eq:jqv-bound} and \ref{ass:SH}, we obtain
\ba{
\norm{{[J,J]}^n_1}_{L^p(\sch^\infty)}
&\leq\norm{{[J,J]}^n_1-[J,J]_1}_{L^p(\sch^\infty)}+\norm{[J,J]_1}_{L^p(\sch^\infty)}\\
&\leq \sqrt2c_0\log(d/\alpha)\norm{\sqrt{\max_{1\leq i\leq n}|\Delta^n_iJ|^2\norm{{[J,J]}^n_1}}}_{p}
+\{\sqrt2c_0\log(d/\alpha)+1\}\Lambda\\
&\leq c_0^2\log^2(d/\alpha)\norm{\max_{1\leq i\leq n}|\Delta^n_iJ|^2}_{p}
+\frac{1}{2}\norm{{[J,J]}^n_1}_{L^p(\sch^\infty)}
+\sqrt2\{c_0\log(d/\alpha)+1\}\Lambda,
}
where the last line follows from the AM-GM inequality. Thus we conclude
\[
\norm{{[J,J]}^n_1}_{L^p(\sch^\infty)}
\leq 2c_0^2\log^2(d/\alpha)\norm{\max_{1\leq i\leq n}|\Delta^n_iJ|^2}_{p}
+2\sqrt2\{c_0\log(d/\alpha)+1\}\Lambda.
\]
Since $p\leq\log n$, we have
\[
\norm{\max_{1\leq i\leq n}|\Delta^n_iJ|^2}_{p}
\leq\norm{\max_{1\leq i\leq n}|\Delta^n_iJ|^2}_{\log n}
\leq\bra{\sum_{i=1}^n\E\left[|\Delta^n_iJ|^{2\log n}\right]}^{1/\log n}
\leq e\max_{1\leq i\leq n}\norm{\Delta^n_iJ}_{2\log n}^2.
\]
Theorem \ref{thm:rosenthal} and \eqref{eq:bdd-jump} yield
\ba{
\|\Delta^n_iJ\|_{2\log n}&\lesssim\bra{\sqrt{\log n}\norm{\sqrt{\int_{t_{i-1}}^{t_i}\int_E|\delta(s,z)|^2\nu(ds,dz)}}_{2\log n}+(\log n)\norm{\sup_{s\in(t_{i-1},t_i]}|\Delta J_s|}_{2\log n}}\\
&\leq\sqrt\Lambda\bra{\sqrt{(\log n)\sum_{j=1}^d\int_{t_{i-1}}^{t_i}\int_E\gamma_j(z)^2\nu(ds,dz)}+\log n}\\
&\leq\sqrt\Lambda\bra{\sqrt{\frac{d\log n}{n}\ol{\gamma}_2}+\log n}.
}
Thus we obtain
\ba{
\norm{{[J,J]}^n_1}_{L^p(\sch^\infty)}
&\leq c_1\log^2(d/\alpha)\Lambda\bra{\frac{d\log n}{n}\ol{\gamma}_2+\log^2 n}
+c_1\log(d/\alpha)\Lambda\\
&\leq 2c_1\Lambda\log^2(d/\alpha)\bra{\frac{d\log n}{n}\ol{\gamma}_2+\log^2n},
}
where $c_1>0$ is a universal constant. 
Finally, we have by Markov's inequality
\ba{
P\bra{\norm{{[J,J]}^n_1}>e\cdot 2c_1\Lambda\log^2(d/\alpha)\bra{\frac{d\log n}{n}\ol{\gamma}_2+\log^2n}}
\leq e^{-p}=\frac{\alpha}{2}.
}
This completes the proof of \eqref{IA}.
\end{proof}

\subsection{Proofs for Section \ref{sec:factor}}\label{sec:proof3}

\subsubsection{Proof of Lemma \ref{lemma:cov-eigen}}

\begin{lemma}\label{lem:ostrowski}
Let $A\in\mathbb R^{r\times r}$ and $B\in\mathbb R^{d\times r}$. Assume that $A$ is positive semidefinite. Then, for any $j=1,\dots,d$,
\[
\lambda_r(A)\lambda_j(BB^\top)\leq\lambda_j(BAB^\top)\leq\lambda_1(A)\lambda_j(BB^\top).
\]
\end{lemma}

\begin{proof}
Note that $\lambda_r(A)\geq0$ by assumption. Then, we have by Theorems 11.4 and 11.8 in \cite{MaNe19}
\ba{
\lambda_j(BAB^\top)&=\max_{C\in\mathbb R^{d\times(d-j)}:C^\top C=I_{d-j}}\min_{x\in\mathbb R^d\setminus\{0\}:C^\top x=0}\frac{x^\top BAB^\top x}{x^\top x}\\
&\geq\max_{C\in\mathbb R^{d\times(d-j)}:C^\top C=I_{d-j}}\min_{x\in\mathbb R^d\setminus\{0\}:C^\top x=0}\frac{\lambda_r(A)x^\top BB^\top x}{x^\top x}\\
&=\lambda_r(A)\lambda_j(BB^\top),
}
and similarly $\lambda_j(BAB^\top)\leq\lambda_1(A)\lambda_j(BB^\top)$.
\end{proof}

\begin{proof}[\bf Proof of Lemma \ref{lemma:cov-eigen}]
We may assume $n>r$ without loss of generality. 
First, note that $\lambda_j({[Y,Y]}^n_1)=\lambda_j(\Delta^nY(\Delta^nY)^\top)=\mf s_j(\Delta^n_jY)^2$ for $j=1,\dots,d\wedge n$. 
Next, we have by Corollary 7.3.5 in \cite{HJ2013}
\[
|\mf{s}_j(\Delta^nY)-\mf{s}_j(\beta\Delta^nF)|\leq\|\Delta^nZ\|
\]
for $j=1,\dots,d\wedge n$. 
Therefore, for $j=r_\tau+1,\dots,d$,
\ba{
\sqrt{\lambda_j({[Y,Y]}^n_1)}
&\leq\sqrt{\lambda_{r_1+1}({[Y,Y]}^n_1)}
=\mf{s}_{r_\tau+1}(\Delta^nY)
\leq\mf{s}_{r_\tau+1}(\beta\Delta^nF)+\|\Delta^nZ\|\\
&=\sqrt{\lambda_{r_\tau+1}(\beta{[F,F]}^n_1\beta^\top)}+\sqrt{\|{[Z,Z]}^n_1\|}.
}
Similarly, for $j=1,\dots,r_\tau$,
\ba{
\sqrt{\lambda_{j}(\beta{[F,F]}^n_1\beta^\top)}-\sqrt{\|{[Z,Z]}^n_1\|}
\leq\sqrt{\lambda_j({[Y,Y]}^n_1)}
\leq\sqrt{\lambda_{j}(\beta{[F,F]}^n_1\beta^\top)}+\sqrt{\|{[Z,Z]}^n_1\|}.
}
Now, by Theorem \ref{localization} and \eqref{ass:d-n}, we have $\|{[Z,Z]}^n_1\|=O_p(\Lambda_nd^\tau)$. Hence, by \ref{ass:svf} and \eqref{eq:svf}, we obtain $\|{[Z,Z]}^n_1\|=o_p(d^\alpha)$ for any $\alpha>\tau$. 
Meanwhile, since $\|{[F,F]}^n_1-[F,F]_1\|\to^p0$ as $n\to\infty$ by Theorem 4.47a in \cite[Chapter I]{JS}, we have $\lambda_{j}(\beta{[F,F]}^n_1\beta^\top)\asymp_p d^{\alpha_j}$ for $j=1,\dots,r$ by Lemma \ref{lem:ostrowski} and \ref{ass:F}. 
In addition, since $\rank(\beta{[F,F]}^n_1\beta^\top)\leq r$, we have $\lambda_{j}(\beta{[F,F]}^n_1\beta^\top)=0$ for $j>r$. Combining these results completes the proof. 
\end{proof}

\subsubsection{Proof of Proposition \ref{prop:bn}}

First, since
\ba{
P(\hat r^{BN}_\tau>r_\tau)
&=P\bra{\lambda_j({[Y,Y]}^n_1)>d^\tau g(d)\text{ for some $j>r_\tau$}}
\leq P\bra{\lambda_{r_\tau+1}({[Y,Y]}^n_1)>d^\tau g(d)},
}
we have $P(\hat r^{BN}_\tau>r_\tau)\to0$ by Lemma \ref{lemma:cov-eigen}\ref{eigen-weak} and $\lim_{d\to\infty}g(d)=\infty$. 
If $r_\tau=0$, this completes the proof. Otherwise, since
\ba{
P(\hat r^{BN}_\tau<r_\tau)
&=P\bra{\lambda_j({[Y,Y]}^n_1)\leq d^\tau g(d)\text{ for all $j\geq r_\tau$}}
\leq P\bra{\lambda_{r_\tau}({[Y,Y]}^n_1)\leq d^\tau g(d)},
}
we have $P(\hat r^{BN}_\tau<r_\tau)\to0$ by Lemma \ref{lemma:cov-eigen}\ref{eigen-strong} and \eqref{eq:svf}. 
Consequently, $P(\hat r^{BN}_\tau\neq r_\tau)=P(\hat r^{BN}_\tau<r_\tau)+P(\hat r^{BN}_\tau>r_\tau)\to0$.
\qed

\subsubsection{Proof of Proposition \ref{prop:pelger}}

First, since
\ba{
P(\hat r^{P}_\tau(\gamma)>r_\tau)
&=P\bra{ER_j({[Y,Y]}^n_1)>1+\gamma\text{ for some $j>r_\tau$}}\\
&=P\bra{\lambda_j({[Y,Y]}^n_1)>(1+\gamma)\lambda_{j+1}({[Y,Y]}^n_1)+\gamma d^\tau g(d)\text{ for some $j>r_\tau$}}\\
&\leq P\bra{\lambda_{r_\tau+1}({[Y,Y]}^n_1)>\gamma d^\tau g(d)},
}
we have $P(\hat r^{P}_\tau(\gamma)>r_\tau)\to0$ by Lemma \ref{lemma:cov-eigen}\ref{eigen-weak} and $\lim_{d\to\infty}g(d)=\infty$. 
If $r_\tau=0$, this completes the proof. Otherwise, since
\ba{
P(\hat r^{P}_\tau(\gamma)<r_\tau)
&\leq P\bra{ER_{r_\tau}({[Y,Y]}^n_1)\leq 1+\gamma}\\
&\leq P\bra{\lambda_{r_\tau}({[Y,Y]}^n_1)\leq (1+\gamma)\lambda_{r_\tau+1}({[Y,Y]}^n_1)+\gamma d^\tau  g(d)},
}
we have $P(\hat r^{P}_\tau(\gamma)<r_\tau)\to0$ by Lemma \ref{lemma:cov-eigen} and \eqref{eq:svf}. 
Consequently, $P(\hat r^{P}_\tau(\gamma)\neq r_\tau)=P(\hat r^{P}_\tau(\gamma)<r_\tau)+P(\hat r^{P}_\tau(\gamma)>r_\tau)\to0$.
\qed

\subsubsection{Proof of Lemma \ref{lemma:cor-eigen}}

\begin{lemma}\label{lemma:diag}
Under the assumptions of Lemma \ref{lemma:cor-eigen},
\[
\max_{1\leq j\leq d}{[Y^j,Y^j]}^n_1=O_p\bra{\Lambda_n\bra{\frac{\log d}{\log\log d}}^2}
\quad\text{and}\quad
\max_{1\leq j\leq d}\bra{{[Y^j,Y^j]}^n_1}^{-1}=O_p(1)
\]
as $n\to\infty$.
\end{lemma}

The proof of Lemma \ref{lemma:diag} is more or less routine but lengthy, so we postpone it to Appendix \ref{sec:proof-diag}. 

\begin{proof}[\bf Proof of Lemma \ref{lemma:cor-eigen}]
By Ostrowski's theorem (cf.~Theorem 4.5.9 in \cite{HJ2013}), for any $j=1,\dots,d$,
\[
\lambda_j\bra{{[Y,Y]}^n_1}\bra{\max_{1\leq j\leq d}{[Y^j,Y^j]}^n_1}^{-1}\leq\lambda_j\bra{{R}_{Y,n}}\leq\lambda_j\bra{{[Y,Y]}^n_1}\max_{1\leq j\leq d}\bra{{[Y^j,Y^j]}^n_1}^{-1}.
\]
Hence, the desired result follows from Lemmas \ref{lemma:cov-eigen} and \ref{lemma:diag} as well as \eqref{eq:svf}.
\end{proof}

\appendix

\section*{Appendix}

\section{Additional proofs}\label{sec:a-proof}

\subsection{Proof of Theorem \ref{localization}}\label{pr-local}

\begin{lemma}\label{lemma:l2}
Assume \ref{ass:SH}. There is a universal constant $C>0$ such that
\[
\norm{{[Z,Z]}^n_1}_{L^2(\sch^\infty)}\leq C\Lambda\bra{1+\ol\gamma_2\bra{1+\frac{d}{\sqrt n}}}.
\]
\end{lemma}

\begin{proof}
In view of \eqref{xx-bound}, the desired result follows once we show that
\ben{\label{eq:l2-aim}
\norm{{[X,X]}^n_1}_{L^2(\sch^\infty)}\lesssim\Lambda\bra{1+\ol\gamma_2}\bra{1+\frac{d}{n}}
\quad\text{and}\quad
\norm{{[J,J]}^n_1}_{L^2(\sch^\infty)}
\lesssim\Lambda\bra{1+\ol\gamma_2\frac{d}{\sqrt n}}.
}
Applying \eqref{eq:basic-est} and \eqref{eq:bernstein} to $X$ and then using \eqref{eq:net-arg} and \cite[Lemma 2.2.10]{VW1996}, we obtain
\[
\norm{{[X,X]}^n_1-[X,X]_1}_{L^2(\sch^\infty)}\lesssim\Lambda\bra{1+\ol\gamma_2}\bra{\frac{d}{n}+\sqrt{\frac{d}{n}}}.
\]
This implies the first bound in \eqref{eq:l2-aim}. 
Next, for all $j,k\in\{1,\dots,d\}$, we have by integration by parts 
\ba{
{[J^j,J^k]}^n_1-[J^j,J^k]_1=\sum_{i=1}^n\bra{\int_{t_{i-1}}^{t_i}(J^j_{s-}-J^j_{t_{i-1}})dJ^k_s+\int_{t_{i-1}}^{t_i}(J^k_{s-}-J^k_{t_{i-1}})dJ^j_s}.
}
Using the It\^o isometry, \eqref{eq:bdd-jump} and Jensen's inequality repeatedly, we obtain
\ba{
\sum_{j,k=1}^d\E\sbra{\abs{\int_{t_{i-1}}^{t_i}(J^j_{s-}-J^j_{t_{i-1}})dJ^k_s}^2}
&=\sum_{j,k=1}^d\E\sbra{\int_{t_{i-1}}^{t_i}\int_E(J^j_{s-}-J^j_{t_{i-1}})^2\delta_k(s,z)^2\nu(ds,dz)}\\
&\leq\Lambda\sum_{j,k=1}^d\int_{t_{i-1}}^{t_i}\E\sbra{(J^j_{s-}-J^j_{t_{i-1}})^2}ds\int_E\gamma_k(z)^2\levy(dz)\\
&\leq\Lambda\ol\gamma_2d\sum_{j=1}^d\int_{t_{i-1}}^{t_i}\E\sbra{\int_{t_{i-1}}^{s}\int_E\delta_j(u,z)^2\nu(du,dz)}ds\\
&\leq\Lambda^2\ol\gamma_2^2\frac{d^2}{n^2}.
}
Consequently, we deduce
\ba{
\E\sbra{\norm{{[J,J]}^n_1-[J,J]_1}^2}
\leq\E\sbra{\norm{{[J,J]}^n_1-[J,J]_1}_F^2}
\leq2\Lambda^2\ol\gamma_2^2\frac{d^{2}}{n}.
}
This implies the second bound in \eqref{eq:l2-aim}.
\end{proof}

\begin{proof}[\bf Proof of Theorem \ref{localization}]
We divide the proof into four steps. 

\noindent\textbf{Step 1}. First, since the process $Z$ is not affected by the values $b_0$ and $\sigma_0$, we may assume $b_0=0$ and $\sigma_0=0$ and hence $h_0=0$ without loss of generality. Then, there is a localizing sequence $(T_m)_{m=1}^\infty$ such that, for each $m$, $h_t\leq m$ if $0\leq t\leq T_m$. 
Next, for any $m,n\in\mathbb N$, define the stopping time $R_{m,n}$ by $R_{m,n}:=\inf\{t\in[0,1]:\|[Z,Z]_t\|\geq m\Lambda_n\}\wedge1$ and set $S_{m,n}:=T_m\wedge R_{m,n}\wedge\tau_m$. 
We have by assumption
\[
\limsup_{m\to\infty}\limsup_{n\to\infty}P\bra{S_{m,n}<1}=0.
\]
Finally, we define the processes $b^{(m)}=(b^{(m)}_t)_{t\in[0,1]}$, $\sigma^{(m)}=(\sigma^{(m)}_t)_{t\in[0,1]}$ and function $\delta^{(m)}$ on $\Omega\times[0,1]\times E$ by
\ba{
b_t^{(m)}=b_{t\wedge S_{m,n}},\quad
\sigma_t^{(m)}=\sigma_{t\wedge S_{m,n}},\quad
\delta^{(m)}(t,z)=\delta(t\wedge S_{m,n},z)1_{\{|\delta(t\wedge S_{m,n},z)|^2\leq m\Lambda_n\}}.
}
Then we define the process $Z(m)=(Z(m)_t)_{t\in[0,1]}$ by
\ba{
Z(m)_t&=Z_0+\int_0^tb^{(m)}_sds+\int_0^t\sigma^{(m)}_sdW_s
+\kappa(\delta^{(m)})\star(\mu-\nu)_t
+\kappa'(\delta^{(m)})\star\mu_t.
}
Note that $\|[Z(m),Z(m)]_1\|\leq\|[Z(m),Z(m)]_{S_{m,n}-}\|+|\Delta Z(m)_{S_{m,n}}|^2\leq2m\Lambda_n$. Hence $Z(m)$ satisfies \ref{ass:SH} with $\Lambda=2m\Lambda_n$ and $\gamma_j=\gamma^{(n)}_{m,j}$. 

\noindent\textbf{Step 2}. We prove $Z_t=Z(m)_t$ for all $t<S_{m,n}$. First, $\int_0^tb_sds=\int_0^tb^{(m)}_sds$ for all $t<S_{m,n}$ by construction. 
Next, $\int_0^t\sigma_sdW_s=\int_0^t\sigma^{(m)}_sdW_s$ for all $t<S_{m,n}$ by Claim 4.38 in \cite[Chapter I]{JS}. 
Now, let us consider the representation \eqref{eq:prm} for $\mu$, and recall that we have $\Delta Z_s=\delta(s,\theta_s)1_{\mcl T}(s)$ for all $s$. 
Then, if $s<S_{m,n}$, we also have $s<R_{m,n}$ and thus $|\Delta Z_s|^2\leq\norm{[Z,Z]_s}\leq m\Lambda_n$, so $\delta(s,\theta_s)1_{\mcl T}(s)=\delta^{(m)}(s,\theta_s)1_{\mcl T}(s)$ for all $s<S_{m,n}$. This implies $\kappa'(\delta^{(m)})\star\mu_t=\kappa'(\delta)\star\mu_t$ for all $t<S_{m,n}$. 
Also, since $\kappa(\delta)1_{\{|\delta|^2>m\Lambda_n\}}\star(\mu-\nu)=\kappa(\delta)1_{\{|\delta|^2>m\Lambda_n\}}\star\mu-\kappa(\delta)1_{\{|\delta|^2>m\Lambda_n\}}\star\nu$ by Proposition 1.28 in \cite[Chapter II]{JS}, $\kappa(\delta)1_{\{|\delta|^2>m\Lambda_n\}}\star(\mu-\nu)_t=0$ for all $t<S_{m,n}$. Thus $\kappa(\delta)\star(\mu-\nu)_t=\kappa(\delta)1_{\{|\delta|^2\leq m\Lambda_n\}}\star(\mu-\nu)_t$ for all $t<S_{m,n}$. 
Since $\kappa(\delta(s,\theta_s))1_{\{|\delta(s,\theta_s)|^2\leq m\Lambda_n\}}1_{\mcl T}(s)=\kappa(\delta^{(m)}(s,\theta_s))1_{\mcl T}(s)$ for all $s\leq S_{m,n}$, two purely discontinuous local martingales $\kappa(\delta)1_{\{|\delta|^2\leq m\Lambda_n\}}\star(\mu-\nu)$ and $\kappa(\delta^{(m)})\star(\mu-\nu)$ have the same jumps when they are stopped at $S_{m,n}$. Hence they are almost surely equal on $[0,S_{m,n}]$ by Corollary 4.19 in \cite[Chapter I]{JS}. Consequently, $\kappa(\delta)\star(\mu-\nu)_t=\kappa(\delta^{(m)})\star(\mu-\nu)_t$ for all $t<S_{m,n}$. Overall, $Z_t=Z(m)_t$ for all $t<S_{m,n}$.

\noindent\textbf{Step 3}. Set
\ben{\label{def:an}
u_n:=\Lambda_n(\log^2d)\bra{\frac{d}{n}\log n+\log^2n},\qquad
v_n:=\Lambda_n\bra{1+\frac{d}{\sqrt n}}.
}
In this step, we prove
\ben{\label{localization-aim}
\|[Z,Z]^n_1\|=O_p(u_n\wedge v_n).
}
Fix $m\in\mathbb N$ and $\alpha\in(0,2/e^2)$. By Step 1 and Theorem \ref{jump-bound2}, there is a universal constant $C>0$ such that
\[
\norm{{[Z(m),Z(m)]}^n_1}
\leq Cm\Lambda_n\cbra{\frac{d+\log(1/\alpha)}{n}+\ol\gamma_{m,2}+\log^2(d/\alpha)\bra{\frac{d\log n}{n}\ol{\gamma}_{m,2}+\log^2n}}
\]
with probability $1-\alpha$ when $\alpha\geq2/n$. Hence
\[
\limsup_{K\to\infty}\limsup_{n\to\infty}P\bra{\|{[Z(m),Z(m)]}^n_1\|>Ku_n}\leq\alpha.
\]
Also, by Markov's inequality and Lemma \ref{lemma:l2},
\ba{
\limsup_{K\to\infty}\limsup_{n\to\infty}P\bra{\|{[Z(m),Z(m)]}^n_1\|>Kv_n}
&\leq\limsup_{K\to\infty}K^{-2}\limsup_{n\to\infty}(v_n^{-1}\|{[Z(m),Z(m)]}^n_1\|_{L^2(\sch^\infty)})^2\\
&=0.
}
Consequently,
\[
\limsup_{K\to\infty}\limsup_{n\to\infty}P\bra{\|{[Z(m),Z(m)]}^n_1\|>K(u_n\wedge v_n)}\leq\alpha.
\]
Moreover, noting that $Z_1=Z_{1-}$ a.s.~by Corollary 1.19 in \cite[Chapter II]{JS}, we have ${[Z,Z]}^n_1={[Z(m),Z(m)]}^n_1$ on the event $\{S_{m,n}=1\}$ by Step 2. Hence
\ba{
&\limsup_{K\to\infty}\limsup_{n\to\infty}P\bra{{[Z,Z]}^n_1>K(u_n\wedge v_n)}\\
&\leq \limsup_{K\to\infty}\limsup_{n\to\infty}P\bra{{[Z(m),Z(m)]}^n_1>K(u_n\wedge v_n)}+\limsup_{n\to\infty}P(S_{m,n}<1)\\
&\leq\alpha+\limsup_{n\to\infty}P(S_{m,n}<1).
}
Letting $m\to\infty$ and $\alpha\downarrow0$, we obtain \eqref{localization-aim}.

\noindent\textbf{Step 4}. When $d\leq\sqrt n$, $u_n\wedge v_n\leq\Lambda_n(1+d/\sqrt n)\leq2\Lambda_n$. Otherwise, we have $\log n\leq2\log d$, so
\[
u_n\wedge v_n=O\bra{\Lambda_n(\log^2d)\bra{\frac{d}{n}\log (d\wedge n)+\log^2(d\wedge n)}}.
\]
Hence the desired result follows from \eqref{localization-aim}.
\end{proof}

\subsection{Proof of Theorem \ref{matrix-burkholder}}\label{sec:mBDG}

Before starting the proof, we introduce some notation. 
For numbers $a_1,\dots,a_d$, $\diag(a_1,\dots,a_d)$ denotes the diagonal matrix with the diagonal entries $a_1,\dots,a_d$.  
For a $d\times d$ matrix $A$, the Schatten $p$-norm of $A$ is defined by 
\[
\|A\|_{\sch^p}=
\begin{cases}
\left(\sum_{j=1}^d\mf s_j(A)^p\right)^{1/p} & \text{if }p\in[1,\infty),\\
\mf s_1(A) & \text{if }p=\infty.
\end{cases}
\]
For a $d\times d$ random matrix $Y$ and $p\in[1,\infty),q\in[1,\infty]$, we set $\|Y\|_{L^p(\sch^q)}:=\|\|Y\|_{\sch^q}\|_p$. 
The following properties are well-known:
\begin{itemize}

\item For any $A\in\mathbb R^{d\times d}$,
\ben{\label{eq:trace}
|\trace(A)|\leq\|A\|_{\sch^1}.
}

\item For any $A,B\in\mathbb R^{d\times d}$ and any $p,q\in[1,\infty]$ such that $1/p+1/q=1$,
\ben{\label{eq:holder}
\|AB\|_{\sch^1}\leq\|A\|_{\sch^p}\|A\|_{\sch^q}.
}

\end{itemize}

The proof of Theorem \ref{matrix-burkholder} relies on two far-reaching results in the theory of non-commutative $L^p$-spaces. 
The first one is the so-called UMD property of non-commutative $L^p$-spaces. 
\begin{lemma}[\cite{Ra02}, Corollary 4.5]\label{lem:UMD}
For any $p\in(1,\infty)$, the Banach space $\mathbb R^{d\times d}$ equipped with the norm $\norm{\cdot}_{\sch^p}$ has the UMD property. More precisely, there is a universal constant $C>0$ such that
\[
\norm{\sum_{i=1}^n\epsilon_iY_i}_{L^p(\sch^p)}\leq C\frac{p}{1-1/p}\norm{\sum_{i=1}^nY_i}_{L^p(\sch^p)}
\]
for any $\mathbb R^{d\times d}$-valued martingale difference sequence $(Y_i)_{i=1}^n$ and any constants $\epsilon_1,\dots,\epsilon_n\in\{-1,1\}$. 
\end{lemma}
We refer to \cite[Chapter 4]{HvNVW16} for a detailed exposition of the UMD property of Banach spaces. 
We will use Lemma \ref{lem:UMD} in the following form.
\begin{corollary}\label{coro:UMD}
Let $\epsilon=(\epsilon_i)_{i=1}^n$ be a sequence of i.i.d.~Rademacher variables. 
There is a universal constant $C>0$ such that
\[
\norm{\sum_{i=1}^n\epsilon_iY_i}_{L^q(\sch^p)}\leq C\bra{\frac{q}{p}+\frac{1-1/p}{1-1/q}}\frac{p}{1-1/p}\norm{\sum_{i=1}^nY_i}_{L^q(\sch^p)}
\]
for any $p,q\in(1,\infty)$ and any $\mathbb R^{d\times d}$-valued martingale difference sequence $(Y_i)_{i=1}^n$ independent of $\epsilon$.
\end{corollary}

\begin{proof}
The claim follows from Lemma \ref{lem:UMD}, \cite[Theorem 4.2.7]{HvNVW16} and Jensen's inequality. 
\end{proof}

\begin{rmk}[Relation to symmetrization inequalities]\label{rem:UMD}
In the setting of Corollary \ref{coro:UMD}, observe that $(\epsilon_iY_i)_{i=1}^n$ is also an $\mathbb R^{d\times d}$-valued martingale difference sequence with respect to an appropriate filtration. 
So the role of $Y_i$ and $\epsilon_iY_i$ can be switched in the above inequality. 
Moreover, since $\|A\|\leq\|A\|_{\sch^{\log d}}\leq e\|A\|$ for any $A\in\mathbb R^{d\times d}$, we have 
\[
\norm{\sum_{i=1}^nY_i}_{L^p(\sch^\infty)}\leq C'p\bra{1+\frac{\log d}{p-1}}\norm{\sum_{i=1}^n\epsilon_iY_i}_{L^p(\sch^\infty)},
\]
where $C'$ is a universal constant. 
Consequently, matrix-valued martingales satisfy a symmetrization inequality with respect to the spectral norm up to a constant depending on $d$ and $p$. 
\end{rmk}
Combining the symmetrization inequality in the above remark with the matrix Khintchine inequality (see e.g.~Theorem 4.1.1 in \cite{Tr15}), we can derive the bound of Theorem \ref{matrix-burkholder} with an additional multiplicative factor $\sqrt{p+\log d}$. 
To avoid getting such an extra factor, we will use a duality argument. 
A key result is the non-commutative Khintchine inequality for $p=1$ originally obtained in \cite{LPP91}. We will state a version given in \cite{HaMu07} with an explicit constant. 
To state this result, we introduce some additional notation. 
Given a sequence of $d\times d$ matrices $\mcl A=(A_i)_{i=1}^n$ and $p\in[1,\infty]$, define
\[
\norm{\mcl A}_{\mcl H^p_C}:=\norm{\lpa\sum_{i=1}^nA_i^\top A_i\rpa^{1/2}}_{\sch^p},\qquad
\norm{\mcl A}_{\mcl H^p_R}:=\norm{\lpa\sum_{i=1}^nA_i A_i^\top\rpa^{1/2}}_{\sch^p}
\]
and
\[
\norm{\mcl A}_{\mcl H^p}:=
\begin{cases}
\inf\left\{\norm{(B_i)_{i=1}^n}_{\mcl H^p_C}+\norm{(C_i)_{i=1}^n}_{\mcl H^p_R}:A_i=B_i+C_i,~B_i,C_i\in\mathbb R^{d\times d}~(i=1,\dots,n)\right\} & \text{if }p<2,\\
\norm{\mcl A}_{\mcl H^p_C}\vee\norm{\mcl A}_{\mcl H^p_R} & \text{if }p\geq2.
\end{cases}
\]

\begin{lemma}[\cite{HaMu07}, Theorem 2.11]\label{lem:khintchine}
Let $\epsilon_1,\dots,\epsilon_n$ be i.i.d.~Rademacher variables. Then
\[
\frac{1}{\sqrt 3}\norm{\mcl A}_{\mcl H^1}
\leq\norm{\sum_{i=1}^n\epsilon_iA_i}_{L^1(\sch^1)}
\leq \norm{\mcl A}_{\mcl H^1}
\]
for any sequence of $d\times d$ matrices $\mcl A=(A_i)_{i=1}^n$.
\end{lemma}
An important feature of the above inequality is that the involved constants are dimension-free. 
This contrasts with the matrix Khintchine inequality with respect to the spectral norm. 

%

We additionally need a few elementary results for matrices which would be well-known at least for general non-commutative $L^p$-spaces. We directly prove them for the sake of completeness. 
\begin{lemma}[Duality formula]\label{lem:duality}
Let $Y$ be a $d\times d$ random matrix such that $\|Y\|_{L^p(\sch^\infty)}<\infty$ for some $p\in(1,\infty)$. Let $q\in(1,\infty)$ satisfy $1/p+1/q=1$. Then
\ben{\label{eq:duality}
\|Y\|_{L^p(\sch^\infty)}=\sup_{\|Z\|_{L^q(\sch^1)}\leq1}|\E[\trace(YZ)]|,
}
where the supremum is taken over all $d\times d$ random matrices $Z$ with $\|Z\|_{L^q(\sch^1)}\leq1$.
\end{lemma}

\begin{proof}
By \eqref{eq:trace}--\eqref{eq:holder} and H\"older's inequality, the right hand side of \eqref{eq:duality} is bounded by $\|Y\|_{L^p(\sch^\infty)}$. 
To prove the opposite inequality, let $Y=UDV$ be a singular decomposition of $Y$ such that $U,V\in\mathbb R^{d\times d}$ are (random) orthogonal matrices and $D=\diag(\mf s_1(Y),\dots,\mf s_d(Y))$. Note that $U$ and $V$ can be taken to be measurable (cf.~\cite{Az74}). Define 
\[
Z:=V^\top\diag((\E[\|Y\|^p])^{-1/q}\|Y\|^{p-1},0,\dots,0)U^\top.
\]
Then we have
\ba{
\E[\trace(YZ)]=\E[\trace(DVZU)]
=(\E[\|Y\|^p])^{-1/q}\E[\|Y\|^{p}]
=\|Y\|_{L^p(\sch^\infty)}.
}
Also, since $\mf s_j(Z)=(\E[\|Y\|^p])^{-1/q}\|Y\|^{p-1}$ if $j=1$ and $\mf s_j(Z)=0$ otherwise, we obtain
\ba{
\|Z\|_{L^q(\sch^1)}^q=\E[\|Z\|_{\sch^1}^q]
=(\E[\|Y\|^p])^{-1}\E[\|Y\|^{q(p-1)}]
=1.
} 
Consequently, $\|Y\|_{L^p(\sch^\infty)}$ is bounded by the right hand side of \eqref{eq:duality}. 
\end{proof}

\begin{lemma}\label{lem:schwarz}
Let $\mcl A=(A_i)_{i=1}^n$ and $\mcl B=(B_i)_{i=1}^n$ be two sequences of $d\times d$ matrices. Also, let $p,q\in[1,\infty]$ satisfy $1/p+1/q=1$. Then, 
\ba{
\norm{\sum_{i=1}^nA_iB_i}_{\sch^1}
\leq\norm{\mcl A}_{\mcl H^p_R}\norm{\mcl B}_{\mcl H^q_C}.
}
\end{lemma}

\begin{proof}
Define the matrices $A\in\mathbb R^{d\times nd}$ and $B\in\mathbb R^{nd\times d}$ as in \eqref{eq:expand}. 
Then, by \eqref{eq:holder},
\ba{
\norm{\sum_{i=1}^nA_iB_i}_{\sch^1}
=\norm{AB}_{\sch^1}
\leq\norm{A}_{\sch^p}\norm{B}_{\sch^q}.
}
Since $AA^\top=\sum_{i=1}^nA_iA_i^\top$ and $B^\top B=\sum_{i=1}^nB_i^\top B_i$, we obtain
$\norm{A}_{\sch^p}=\norm{\mcl A}_{\mcl H^p_R}$ and $\norm{B}_{\sch^q}=\norm{\mcl B}_{\mcl H^q_C}$. This completes the proof. 
\end{proof}

\begin{corollary}\label{coro:schwarz}
Under the assumptions of Lemma \ref{lem:schwarz}, we have 
\[
\left|\trace\left(\sum_{i=1}^nA_iB_i\right)\right|
\leq\norm{\mcl A}_{\mcl H^p}\norm{\mcl B}_{\mcl H^q}.
\]
\end{corollary}

\begin{proof}
When $p=q=2$, the claim immediately follows from Lemma \ref{lem:schwarz}. Moreover, since $\trace\left(\sum_{i=1}^nA_iB_i\right)=\trace\left(\sum_{i=1}^nB_iA_i\right)$, it suffices to consider the case $p>2>q$. 

Let $\mcl C=(C_i)_{i=1}^n$ and $\mcl D=(D_i)_{i=1}^n$ be two sequences of $d\times d$ matrices such that $B_i=C_i+D_i$ for $i=1,\dots,n$. Then we have by \eqref{eq:trace} and Lemma \ref{lem:schwarz}
\ba{
\left|\trace\left(\sum_{i=1}^nA_iC_i\right)\right|
\leq\norm{\sum_{i=1}^nA_iC_i}_{\sch^1}
\leq\norm{\mcl A}_{\mcl H^p_R}\norm{\mcl C}_{\mcl H^q_C}
}
and
\ba{
\left|\trace\left(\sum_{i=1}^nA_iD_i\right)\right|
=\left|\trace\left(\sum_{i=1}^nD_iA_i\right)\right|
\leq\norm{\sum_{i=1}^nD_iA_i}_{\sch^1}
\leq\norm{\mcl D}_{\mcl H^q_R}\norm{\mcl A}_{\mcl H^p_C}.
}
Hence we obtain
\ba{
\left|\trace\left(\sum_{i=1}^nA_iB_i\right)\right|
=\left|\trace\left(\sum_{i=1}^nA_iC_i\right)+\trace\left(\sum_{i=1}^nA_iD_i\right)\right|
\leq\norm{\mcl A}_{\mcl H^p}\lpa\norm{\mcl C}_{\mcl H^q_C}+\norm{\mcl D}_{\mcl H^q_R}\rpa.
}
Taking the infimum with respect to $\mcl C$ and $\mcl D$, we obtain the desired result. 
\end{proof}

\begin{proof}[\bf Proof of Theorem \ref{matrix-burkholder}]
If $\|Y_i\|_{L^p(\sch^\infty)}=\infty$ for some $i$, then the right hand side of \eqref{eq:BG} is infinite because $\|(\sum_{j=1}^nY_j^\top Y_j)^{1/2}\|\geq\norm{(Y_i^\top Y_i)^{1/2}}=\|Y_i\|$ by Weyl's inequality. Hence we may assume $\|Y_i\|_{L^p(\sch^\infty)}<\infty$ for all $i$. 
Then, in view of Lemma \ref{lem:duality}, it suffices to show that
\be{
\left|\E\left[\trace\left(\sum_{i=1}^nY_iZ\right)\right]\right|
\lesssim p\bra{\frac{\log d}{p-1}+1}\lpa\E\norm{\mcl Y}_{\mcl H^\infty}^p\rpa^{1/p}
}
for any $d\times d$ random matrix $Z$ such that $\|Z\|_{L^q(\sch^1)}\leq1$, where $q=p/(p-1)$ so that $1/p+1/q=1$. 
For $i=1,\dots,n$, set $Z_i:=\E[Z\mid\mcl G_i]-\E[Z\mid\mcl G_{i-1}]$. Since $\mcl Y$ is a martingale difference sequence with respect to $(\mcl G_j)$, we obtain
\ba{
\E[Y_iZ]=\E[Y_i\E[Z\mid\mcl G_i]]
=\E[Y_iZ_i]+\E[Y_i\E[Z\mid\mcl G_{i-1}]]
=\E[Y_iZ_i].
}
Therefore, 
\ba{
\left|\E\left[\trace\left(\sum_{i=1}^nY_iZ\right)\right]\right|
&=\left|\E\left[\trace\lpa\sum_{i=1}^nY_iZ_i\rpa\right]\right|\\
&\leq\E\left[\norm{\mcl Y}_{\mcl H^\infty}\norm{\mcl Z}_{\mcl H^1}\right]
\leq\lpa\E\left[\norm{\mcl Y}_{\mcl H^\infty}^p\right]\rpa^{1/p}
\lpa\E\left[\norm{\mcl Z}_{\mcl H^1}^q\right]\rpa^{1/q},
}
where $\mcl Z=(Z_i)_{i=1}^n$ and the last line follows from Corollary \ref{coro:schwarz} and H\"older's inequality. Thus, we complete the proof once we prove
\ben{\label{mat-b:aim}
\lpa\E\left[\norm{\mcl Z}_{\mcl H^1}^q\right]\rpa^{1/q}\lesssim p\bra{\frac{\log d}{p-1}+1}.
}
Applying Lemma \ref{lem:khintchine} conditional on $\mcl Z$, we obtain
\[
\norm{\mcl Z}_{\mcl H^1}\leq \sqrt 3\E\left[\norm{\sum_{i=1}^n\epsilon_iZ_i}_{\sch^1}\mid\mcl Z\right],
\]
where $\epsilon_1,\dots,\epsilon_n$ are i.i.d.~Rademacher variables independent of $\mcl Z$. 
This and Jensen's inequality imply that
\ben{\label{res:khintchine}
\lpa\E\left[\norm{\mcl Z}_{\mcl H^1}^q\right]\rpa^{1/q}
\leq\sqrt{3}\norm{\sum_{i=1}^n\epsilon_iZ_i}_{L^q(\sch^1)}.
}
Now, let $r=(\log d)/(\log d -1)$ so that $1/r+1/\log d=1$. By the $C_r$-inequality and H\"older's inequality, we have for any $A\in\mathbb R^{d\times d}$
\[
\norm{A}_{\sch^r}\leq\norm{A}_{\sch^1}\leq e\norm{A}_{\sch^r}.
\]
Moreover, note that $\mcl Z$ is a martingale difference sequence with respect to $(\mcl G_i)_{i=0}^n$ by construction. Combining these results with Corollary \ref{coro:UMD}, we obtain
\ba{
\norm{\sum_{i=1}^n\epsilon_iZ_i}_{L^q(\sch^1)}
&\leq e\norm{\sum_{i=1}^n\epsilon_iZ_i}_{L^q(\sch^r)}
\lesssim\bra{\frac{q}{r}+\frac{p}{\log d}}(r\log d)\norm{\sum_{i=1}^nZ_i}_{L^q(\sch^r)}\\
&\lesssim p\bra{\frac{\log d}{p-1}+1}\norm{\E[Z\mid\mcl G_n]-\E[Z\mid\mcl G_0]}_{L^q(\sch^1)}\\
&\leq 2p\bra{\frac{\log d}{p-1}+1}\|Z\|_{L^q(\sch^1)}
\leq 2p\bra{\frac{\log d}{p-1}+1},
}
where the last line follows from Jensen's inequality and $\|Z\|_{L^q(\sch^1)}\leq1$. Combining this with \eqref{res:khintchine} gives \eqref{mat-b:aim}.
where the last line follows from Jensen's inequality and $\|Z\|_{L^q(\sch^1)}\leq1$. Combining this with \eqref{res:khintchine} gives \eqref{mat-b:aim}.
\end{proof}

\subsection{Proof of Theorem \ref{thm:rosenthal}}\label{sec:rosenthal}

The proof is done by approximating the desired inequality by the following discrete-time analog: 
\begin{lemma}
\label{disc-rosenthal}
Let $(\xi_i)_{i=1}^n$ be an $\mathbb R^d$-valued martingale difference sequence with respect to a filtration $(\mcl{G}_i)_{i=0}^n$. Then there is a universal constant $C>0$ such that
\[
\norm{\abs{\sum_{i=1}^n\xi_i}}_p\leq C\bra{\sqrt p\norm{\sqrt{\sum_{i=1}^n\E[|\xi_i|^2\mid\mcl G_{i-1}]}}_p
+p\norm{\max_{i=1,\dots,n}|\xi_i|}_p}
\]
for all $p\geq2$. 
\end{lemma}

\begin{proof}
This follows from Eq.(4.2) of \cite{Pi94} since any Hilbert space is $(2,1)$-smooth. 
\end{proof}

\begin{proof}[\bf Proof of Theorem \ref{thm:rosenthal}]
We divide the proof into two steps. 

\noindent\textbf{Step 1}. 
First we prove the claim when both $M$ and $\langle M\rangle$ are bounded. 
Take $n\in\mathbb N$ arbitrarily and set $\xi^n_i:=M_{it/n}-M_{(i-1)t/n}$ for $i=1,\dots,n$. 
Since $M_t=\sum_{i=1}^n\xi^n_i$ and $(\xi^n_i)_{i=1}^n$ is an $\mathbb R^d$-valued martingale difference sequence with respect to $(\mathcal F_{it/n})_{i=0}^n$, we have by Lemma \ref{disc-rosenthal}
\[
\norm{\abs{M_t}}_p\lesssim\sqrt p\norm{\sqrt{\sum_{i=1}^n\E[|\xi^n_i|^2\mid\mcl F_{(i-1)t/n}]}}_p
+p\norm{\max_{i=1,\dots,n}|\xi^n_i|}_p.
\]
Set $\eta^n_i:=\langle M\rangle_{it/n}-\langle M\rangle_{(i-1)t/n}$ for $i=1,\dots,n$. Then we have
\ba{
\sum_{i=1}^n\E[|\xi^n_i|^2\mid\mcl F_{(i-1)t/n}]=\sum_{i=1}^n\E[\eta^n_i\mid\mcl F_{(i-1)t/n}].
}
Hence
\ba{
\norm{\sqrt{\sum_{i=1}^n\E[|\xi^n_i|^2\mid\mcl F_{(i-1)t/n}]}}_p
&\leq\norm{\sqrt{\sum_{i=1}^n\eta^n_i}}_p
+\norm{\sqrt{\abs{\sum_{i=1}^n\bra{\eta^n_i-\E[\eta^n_i\mid\mcl F_{(i-1)t/n}]}}}}_p\\
&=\norm{\langle M\rangle_t^{1/2}}_p+\norm{\sum_{i=1}^n\bra{\eta^n_i-\E[\eta^n_i\mid\mcl F_{(i-1)t/n}]}}_{p/2}^{1/2}.
}
Therefore, we obtain the desired result once we prove the following claims:
\ban{
&\limsup_{n\to\infty}\norm{\sum_{i=1}^n\bra{\eta^n_i-\E[\eta^n_i\mid\mcl F_{(i-1)t/n}]}}_{p/2}=0,
\label{rosenthal-aim1}\\
&\limsup_{n\to\infty}\norm{\max_{i=1,\dots,n}|\xi^n_i|}_p\leq\norm{\sup_{s\in(0,t]}|\Delta M_s|}_p.
\label{rosenthal-aim2}
}

\noindent\ul{Proof of \eqref{rosenthal-aim1}.} 
By the Burkholder--Davis--Gundy inequality,
\ben{\label{ros-aim1-bdg}
\norm{\sum_{i=1}^n\bra{\eta^n_i-\E[\eta^n_i\mid\mcl F_{(i-1)t/n}]}}_{p/2}
\leq C_p\norm{\sqrt{\sum_{i=1}^n\bra{\eta^n_i-\E[\eta^n_i\mid\mcl F_{(i-1)t/n}]}^2}}_{p/2},
}
where $C_p>0$ is a constant depending only on $p$. 
Noting that $\eta_i^n\geq0$ for all $i$, we have
\ba{
\norm{\sqrt{\sum_{i=1}^n\bra{\eta^n_i-\E[\eta^n_i\mid\mcl F_{(i-1)t/n}]}^2}}_{p/2}
\leq\norm{\sqrt{\sum_{i=1}^n\bra{\eta^n_i}^2}}_{p/2}
+\norm{\sqrt{\sum_{i=1}^n\E[\eta^n_i\mid\mcl F_{(i-1)t/n}]^2}}_{p/2}.
}
By Stein's inequality (cf.~Theorem 3.1 in \cite{AsMo93}), 
\ba{
\norm{\sqrt{\sum_{i=1}^n\E[\eta^n_i\mid\mcl F_{(i-1)t/n}]^2}}_{p/2}
\leq\norm{\sqrt{\sum_{i=1}^n\E[\eta^n_i\mid\mcl F_{(i-1)t/n}]^2}}_{p}
\leq C_p'\norm{\sqrt{\sum_{i=1}^n\bra{\eta^n_i}^2}}_{p},
}
where $C_p'$ is a constant depending only on $p$. Consequently, 
\ba{
\norm{\sum_{i=1}^n\bra{\eta^n_i-\E[\eta^n_i\mid\mcl F_{(i-1)t/n}]}}_{p/2}
\leq C_p(1+C_p')\norm{\sqrt{\sum_{i=1}^n\bra{\eta^n_i}^2}}_{p}.
}
Since
\[
\sum_{i=1}^n\bra{\eta^n_i}^2\leq\bra{\max_{i=1,\dots,n}\eta^n_i}\sum_{i=1}^n\eta^n_i
=\bra{\max_{i=1,\dots,n}\eta^n_i}\langle M\rangle_t,
\]
$\sum_{i=1}^n\bra{\eta^n_i}^2$ is bounded by some constant thanks to the boundedness assumption on $\langle M\rangle$. Moreover, since $\langle M\rangle$ is continuous by assumption, $\max_{i=1,\dots,n}\eta^n_i\to0$ as $n\to\infty$ a.s. 
All together, the bounded convergence theorem yields \eqref{rosenthal-aim1}.

\noindent\ul{Proof of \eqref{rosenthal-aim2}.} 
Define the process $M^n=(M^n_s)_{s\in[0,1]}$ by $M^n_s=M_{\lfloor(n/t)s\rfloor t/n}$ for $s\in[0,1]$. 
Observe that $\sup_{s\in(0,t]}|\Delta M^n_s|=\max_{i=1,\dots,n}|\xi^n_i|$. 
Also, by Proposition 6.37d in \cite[Chapter VI]{JS}, the sequence $M^n$ converges a.s.~to $M$ for the Skorohod topology as $n\to\infty$. Hence Lemma 2.5 in \cite[Chapter VI]{JS} yields
\[
\limsup_{n\to\infty}\max_{i=1,\dots,n}|\xi^n_i|\leq\sup_{s\in(0,t]}|\Delta M_s|\quad\text{a.s.}
\]
Since $\max_{i=1,\dots,n}|\xi^n_i|$ is bounded by some constant thanks to the boundedness assumption on $M$, we obtain \eqref{rosenthal-aim2} by the reverse Fatou lemma. 

\noindent\textbf{Step 2}. 
Next we drop the boundedness assumption on $M$ and $\langle M\rangle$ in Step 1 by a standard localization argument. 
For every $m\in\mathbb N$, define the stopping time $T_m$ by $T_m:=\inf\{s\in[0,1]:|M_s|+\langle M\rangle_s>m\}\wedge1$. 
If $0\leq s<T_m$, we have $|M_s|+\langle M\rangle_s\leq m$. Since $\langle M\rangle$ is continuous, this also means $\langle M\rangle_{T_m}\leq m$. 
In addition, recall that jumps of $M$ are bounded by assumption; hence there is a constant $c>0$ such that $\sup_{s\in(0,1]}|\Delta M_s|\leq c$. Thus, $|M_{T_m}|\leq m+c$. 
As a result, the stopped process $M^{T_m}=(M_{s\wedge T_m})_{s\in[0,1]}$ satisfies the boundedness assumption in Step 1. Hence we have
\ba{
\norm{M_{t\wedge T_m}}_p\lesssim\sqrt p\norm{\langle M^{T_m}\rangle_t^{1/2}}_p+p\norm{\sup_{s\in(0,t]}|\Delta M_s^{T_m}|}_p
\leq\sqrt p\norm{\langle M\rangle_t^{1/2}}_p+p\norm{\sup_{s\in(0,t]}|\Delta M_s|}_p.
}
Meanwhile, on the event $\{T_m<1\}$, we have $\sup_{s\in[0,1]}|M_s|+\langle M\rangle_1>m$. Hence $P(T_m<1)\to0$ as $m\to\infty$. This means $M_{t\wedge T_m}\to M_t$ as $m\to\infty$ a.s. Consequently, by Fatou's lemma,
\[
\norm{M_t}_p\leq\liminf_{m\to\infty}\norm{M_{t\wedge T_m}}_p
\lesssim\sqrt p\norm{\langle M\rangle_t^{1/2}}_p+p\norm{\sup_{s\in(0,t]}|\Delta M_s|}_p.
\]
This completes the proof. 
\end{proof}


\subsection{Proof of Lemma \ref{lemma:diag}}\label{sec:proof-diag}

For a semimartingale $X$ and a stopping time $\tau$, denote by $X^\tau$ the semimartingale obtained from $X$ stopped at $\tau$. Then, using Lemma 4.13c in \cite[Chapter I]{JS}, we can check $(X^\tau)^c=(X^c)^\tau$. 
Hence, we may use a localization argument for the proof. In particular, by the same localization procedure as in the proof of Theorem \ref{localization}, we may assume without loss of generality the following: There is a constant $K\geq1$ and Borel functions $\gamma^{(n)}_j:E\to[0,\infty)$ $(n\in\mathbb N; j=1,\dots,d)$ such that
\ben{\label{resid-sh}
|b_s(\omega)|_\infty\leq K\sqrt{\Lambda_n},\quad
\|\sigma_s(\omega)\|\leq K\sqrt{\Lambda_n},\quad
|\delta_j(\omega,s,z)|\wedge1\leq\gamma^{(n)}_j(z),\quad
\|[Z,Z]_1(\omega)\|\leq K^2\Lambda_n 
}
for all $n\in\mathbb N,\omega\in\Omega$, $s\in[0,1]$, $z\in E$ and $j=1,\dots,d$. Moreover, 
\ben{\label{resid-gamma}
\sup_{n\in\mathbb N}\max_{1\leq j\leq d}\int_E\gamma_{j}^{(n)}(z)^2\levy(dz)<\infty.
}
Then, Lemma \ref{lemma:bdd-jump} allows us to additionally assume that
\ben{\label{resid-bdd}
|\delta(\omega,s,z)|\leq K\sqrt\Lambda_n\qquad\text{and}\qquad
|\delta_j(\omega,s,z)|\leq K\sqrt\Lambda_n\gamma^{(n)}_j(z)
}
for all $n\in\mathbb N,\omega\in\Omega$, $s\in[0,1]$, $z\in E$ and $j=1,\dots,d$. 
Furthermore, by a localization procedure similar to the one used in the proof of \cite[Lemma 4.4.9]{JP2012}, we may also assume without loss of generality that there is a Borel function $\gamma_F:E\to[0,\infty)$ such that
\ben{\label{factor-sh}
|b_{s}^F(\omega)|\leq \sqrt K,\quad
\|\sigma_{s}^F(\omega)\|\leq \sqrt K,\quad
|\delta^F(\omega,s,z)|\leq\gamma^F(z)\leq \sqrt K,\quad
\int_E\gamma_F(z)^2\levy(dz)\leq K
}
for all $n\in\mathbb N,\omega\in\Omega$, $s\in[0,1]$ and $z\in E$, where we increase the value of $K$ if necessary.  
In addition, since we may focus only on sufficiently large $n$, we may assume, with $\beta^{j\cdot}$ denoteing the $j$-th row of $\beta$,
\ben{\label{beta-bound}
\max_{1\leq j\leq d}|\beta^{j\cdot}|\leq \sqrt{K\Lambda_n}
}
for all $n$ with increasing the value of $K$ if necessary (recall that $r$ is fixed).  

Next, under \eqref{resid-sh}--\eqref{factor-sh}, we can define the $d$-dimensional processes $J=(J_t)_{t\in[0,1]}$ and $\wt b=(\wt b_t)_{t\in[0,1]}$ by
\[
J_t=\delta^Y\star(\mu-\nu)_t\quad\text{with}\quad\delta^Y=\beta \delta^F+\delta
\]
and
\[
\wt b_t=\beta\bra{b^F_t+\int_{|\delta^F(t,z)|>1}\delta^F(t,z)\levy(dz)}+b_t+\int_{E}\kappa'(\delta(t,z))\levy(dz).
\]
Then we have $Y_t=Y_0+\wt D_t+M_t+J_t$, where
\ba{
\wt D_t=\int_0^t\wt b_sds\qquad\text{and}\qquad
M_t=\int_0^t\sigma^Y_sdW_s\quad\text{with}\quad\sigma^Y_s=\beta\sigma^F_s+\sigma_s.
}
In particular, we have $Y^c=M$. Now, let 
\[
R:=\max_{1\leq j\leq d}\abs{{[Y^j,Y^j]}^n_1-{[J^j,J^j]}^n_1-[M^j,M^j]_1}.
\] 
Then, for all $j=1,\dots,d$,
\be{
[M^j,M^j]_1-R
\leq{[Y^j,Y^j]}^n_1
\leq{[J^j,J^j]}^n_1+[M^j,M^j]_1+R,
}
where we used ${[J^j,J^j]}^n_1\geq0$ to deduce the first inequality. Therefore, the proof of the lemma is completed once we prove
\ban{
\max_{1\leq j\leq d}{[J^j,J^j]}^n_1&=O_p\bra{\Lambda_n\bra{\frac{\log d}{\log\log d}}^2},\label{JJ-bound}\\
R&=o_p(1).
\label{R-bound}
}

\noindent\ul{Proof of \eqref{JJ-bound}}. Set $\gamma^Y_j:=\gamma^F+\gamma_j^{(n)}$ for $j=1,\dots,d$. Define
\[
\ol{\gamma}^Y_2:=\max_{1\leq j\leq d}\int_E\gamma_{j}^{Y}(z)^2\levy(dz).
\]
We have $\ol{\gamma}^Y_2=O(1)$ as $n\to\infty$ by \eqref{resid-gamma} and \eqref{factor-sh}. 
Also, by \eqref{resid-bdd}--\eqref{beta-bound}, we have
\ben{\label{jy-bound}
|\delta^Y(s,z)|\leq 2K\sqrt{\Lambda_n}\quad
\text{and}\quad
|\delta^Y_j(s,z)|\leq K\sqrt{\Lambda_n}\gamma^Y_j(z)
}
for all $j=1,\dots,d$ and $(s,z)\in[0,1]\times E$, where $\delta^Y_j$ denotes the $j$-th component function of $\delta^Y$. 
Now, for each $j\in\{1,\dots,d\}$, integration by parts gives 
\ba{
{[J^j,J^j]}^n_1-[J^j,J^j]_1
=2\sum_{i=1}^n\int_{t_{i-1}}^{t_i}(J^j_{s-}-J^j_{t_{i-1}})dJ^j_s
=2\int_{0}^{1}\sum_{i=1}^n1_{(t_{i-1},t_i]}(s)(J^j_{s-}-J^j_{t_{i-1}})dJ^j_s.
}
Therefore, for any $p\geq2$, we have by Theorem 1 in \cite{ReTi03}
\ba{
&\norm{{[J^j,J^j]}^n_1-[J^j,J^j]_1}_p\\
&\lesssim\frac{p}{\log p}\bra{\norm{\sqrt{\sum_{i=1}^n\int_{t_{i-1}}^{t_i}\int_E(J^j_{s-}-J^j_{t_{i-1}})^2\delta^Y_j(s,z)^2\nu(ds,dz)}}_p
+\norm{\max_{1\leq i\leq n}\sup_{s\in(t_{i-1},t_i]}|(J^j_{s-}-J^j_{t_{i-1}})\Delta J^j_s|}_p}.
}
By \eqref{jy-bound}, we obtain
\ban{
&\norm{\sqrt{\sum_{i=1}^n\int_{t_{i-1}}^{t_i}\int_E(J^j_{s-}-J^j_{t_{i-1}})^2\delta^Y_j(s,z)^2\nu(ds,dz)}}_p
\notag\\
&\leq K\sqrt{\Lambda_n\ol\gamma^Y_2}\norm{\sqrt{\sum_{i=1}^n\int_{t_{i-1}}^{t_i}(J^j_{s-}-J^j_{t_{i-1}})^2ds}}_p
\leq K\sqrt{\Lambda_n\ol\gamma^Y_2\sum_{i=1}^n\int_{t_{i-1}}^{t_i}\norm{J^j_{s-}-J^j_{t_{i-1}}}_{p}^2ds},
\label{eq:jy-qv}
}
where the last inequality follows by the integral Minkowski inequality. 
Theorem 1 in \cite{ReTi03} yields
\ba{
\sup_{s\in[t_{i-1},t_i]}\norm{J^j_{s}-J^j_{t_{i-1}}}_{p}
&\lesssim\frac{p}{\log p}\bra{\norm{\sqrt{\int_{t_{i-1}}^{t_i}\int_E\delta^Y_j(s,z)^2\nu(ds,dz)}}_{p}
+\norm{\sup_{s\in[t_{i-1},t_i]}|\Delta J^j_s|}_{p}}.
}
Since $|\Delta J_s|\leq|\delta^Y(s,\Delta J_s)|$, we obtain by \eqref{jy-bound}
\ben{\label{jy-bound2}
\sup_{s\in[t_{i-1},t_i]}\norm{J^j_{s}-J^j_{t_{i-1}}}_{p}
\lesssim\frac{p}{\log p}K\sqrt{\Lambda_n}\bra{\sqrt{\frac{\ol\gamma^Y_2}{n}}+1}.
}
Consequently, 
\ben{\label{eq:jy-qv1}
\norm{\sqrt{\sum_{i=1}^n\int_{t_{i-1}}^{t_i}\int_E(J^j_{s-}-J^j_{t_{i-1}})^2\delta^Y_j(s,z)^2\nu(ds,dz)}}_p
\lesssim\frac{p}{\log p}K^2\Lambda_n\sqrt{\ol\gamma^Y_2}\bra{\sqrt{\frac{\ol\gamma^Y_2}{n}}+1}.
}
Meanwhile, we have by \eqref{jy-bound}
\ba{
&\norm{\max_{1\leq i\leq n}\sup_{s\in(t_{i-1},t_i]}|(J^j_{s-}-J^j_{t_{i-1}})\Delta J^j_s|}_p^p\\
&\leq\E\sbra{\sum_{i=1}^n\int_{t_{i-1}}^{t_i}\int_E|(J^j_{s-}-J^j_{t_{i-1}})\delta_k(s,z)|^p\mu(ds,dz)}\\
&=\E\sbra{\sum_{i=1}^n\int_{t_{i-1}}^{t_i}\int_E|(J^j_{s-}-J^j_{t_{i-1}})\delta_k(s,z)|^p\nu(ds,dz)}\\
&\leq K^p\Lambda_n^{p/2}\ol\gamma^Y_2\E\sbra{\sum_{i=1}^n\int_{t_{i-1}}^{t_i}|J^j_{s-}-J^j_{t_{i-1}}|^pds}
\leq K^p\Lambda_n^{p/2}\ol\gamma^Y_2\sup_{s\in(t_{i-1},t_i]}\|J^j_{s-}-J^j_{t_{i-1}}\|_p^p.
}
Hence, we obtain by \eqref{jy-bound2}
\ben{\label{eq:jy-qv2}
\norm{\max_{1\leq i\leq n}\sup_{s\in(t_{i-1},t_i]}|(J^j_{s-}-J^j_{t_{i-1}})\Delta J^j_s|}_p
\lesssim\frac{p}{\log p}K^2\Lambda_n\bra{\ol{\gamma}^Y_2}^{1/p}\bra{\sqrt{\frac{\ol\gamma^Y_2}{n}}+1}.
}
Combining \eqref{eq:jy-qv}, \eqref{eq:jy-qv1} and \eqref{eq:jy-qv2}, we obtain
\ba{
\norm{{[J^j,J^j]}^n_1-[J^j,J^j]_1}_p
\lesssim\bra{\frac{p}{\log p}}^2K^2\Lambda_n\bra{\sqrt{\ol\gamma^Y_2}+\bra{\ol{\gamma}^Y_2}^{1/p}}\bra{\sqrt{\frac{\ol\gamma^Y_2}{n}}+1}.
}
Applying this inequality with $p=2\vee\log d$, we obtain
\ba{
\E\sbra{\max_{1\leq j\leq d}\abs{{[J^j,J^j]}^n_1-[J^j,J^j]_1}}
&\leq\bra{\E\sbra{\sum_{j=1}^d\abs{{[J^j,J^j]}^n_1-[J^j,J^j]_1}^p}}^{1/p}\\
&\leq d^{1/p}\max_{1\leq j\leq d}\norm{{[J^j,J^j]}^n_1-[J^j,J^j]_1}_p
=O\bra{\Lambda_n\bra{\frac{\log d}{\log\log d}}^2}.
}
This implies \eqref{JJ-bound}.

\noindent\ul{Proof of \eqref{R-bound}}. Define $\Sigma^Y_s:=\sigma^Y_s(\sigma^Y_s)^\top$. 
By \eqref{resid-sh}--\eqref{beta-bound}, we obtain
\ben{\label{cy-bound}
|\wt b_t|_\infty\leq K\sqrt{\Lambda_n}(2+\ol\gamma^Y_2)
\quad\text{and}\quad
\|\Sigma^Y_s\|=\|\sigma^Y_s\|^2\leq4K^2\Lambda_n.
}
Next, define $X:=\wt D+M$ and consider the following decomposition for $j=1,\dots,d$:
\ba{
{[Y^j,Y^j]}^n_1-{[J^j,J^j]}^n_1-[M^j,M^j]_1
=\bra{{[X^j,X^j]}^n_1-[M^j,M^j]_1}
+2{[\wt D^j,J^j]}^n_1+2{[M^j,J^j]}^n_1.
}
A similar argument to the proof of \eqref{eq:bernstein} implies that there are constants $c_1,c_2>0$ independent of $n$ such that
\[
P\bra{\abs{{[X^j,X^j]}^n_1-[M^j,M^j]_1}\geq u}\leq2\exp\bra{-\frac{u^2}{2(c_1\Lambda_n^2/n+(c_2\Lambda_n/n) u)}}
\]
for all $u\geq0$. 
This and \cite[Lemma 2.2.10]{VW1996} yield
\[
\max_{1\leq j\leq d}\abs{{[X^j,X^j]}^n_1-[M^j,M^j]_1}=O_p\bra{\Lambda_n\bra{\sqrt{\frac{\log d}{n}}+\frac{\log d}{n}}}=o_p(1),
\]
where the last equality follows by \eqref{ass:d-n}. 
Also, we obtain by the Schwarz inequality, \eqref{JJ-bound} and \eqref{cy-bound}
\ba{
\max_{1\leq j\leq d}\abs{{[\wt D^j,J^j]}^n_1}
\leq\sqrt{\max_{1\leq j\leq d}\sum_{i=1}^n\bra{\Delta^n_i\wt D^j}^2}
\sqrt{\max_{1\leq j\leq d}\sum_{i=1}^n\bra{\Delta^n_iJ^j}^2}
=O_p\bra{\Lambda_n\frac{\log d}{\sqrt n\log\log d}}=o_p(1).
}
Therefore, we complete the proof of \eqref{R-bound} once we prove
\ben{\label{MJ-bound}
\max_{1\leq j\leq d}\abs{{[M^j,J^j]}^n_1}
=o_p(1).
}
%

Let $p:=2\vee\log d$. For any $j=1,\dots,d$, Rosenthal's inequality with sharp constant gives (see e.g.~\cite[Eq.(4.3)]{Pi94})
\ben{\label{eq:mj1}
\norm{{[M^j,J^j]}^n_1}_p
\lesssim\frac{p}{\log p}\bra{\norm{\sqrt{\sum_{i=1}^n\E[(\Delta^n_iM^j\Delta^n_iJ^j)^2\mid\mcl F_{t_{i-1}}]}}_p
+\norm{\max_{1\leq i\leq n}\abs{\Delta^n_iM^j\Delta^n_iJ^j}}_p}.
}
Integration by parts yields
\ba{
\Delta^n_iM^j\Delta^n_iJ^j
=\int_{t_{i-1}}^{t_i}(M^j_{s}-M^j_{t_{i-1}})dJ^j_s+\int_{t_{i-1}}^{t_i}(J^j_{s-}-J^j_{t_{i-1}})dM^j_s.
}
Therefore, by \eqref{jy-bound} and \eqref{cy-bound}
\ban{
&\E[(\Delta^n_iM^j\Delta^n_iJ^j)^2\mid\mcl F_{t_{i-1}}]
\notag\\
&\leq2\E\sbra{\int_{t_{i-1}}^{t_i}\int_E(M^j_{s}-M^j_{t_{i-1}})^2\delta^Y_j(s,z)^2\nu(ds,dz)\mid\mcl F_{t_{i-1}}}+2\E\sbra{\int_{t_{i-1}}^{t_i}(J^j_{s-}-J^j_{t_{i-1}})^2(\Sigma_s^Y)^{jj}ds\mid\mcl F_{t_{i-1}}}
\notag\\
&\leq2K^2\Lambda_n\ol\gamma^Y_2\E\sbra{\int_{t_{i-1}}^{t_i}(M^j_{s}-M^j_{t_{i-1}})^2ds\mid\mcl F_{t_{i-1}}}+8K^2\Lambda_n\E\sbra{\int_{t_{i-1}}^{t_i}(J^k_{s-}-J^k_{t_{i-1}})^2ds\mid\mcl F_{t_{i-1}}}
\notag\\
&=2K^2\Lambda_n\ol\gamma^Y_2\E\sbra{\int_{t_{i-1}}^{t_i}\bra{\int_{t_{i-1}}^s(\Sigma_u^Y)^{jj}du}ds\mid\mcl F_{t_{i-1}}}
+8K^2\Lambda_n\E\sbra{\int_{t_{i-1}}^{t_i}\bra{\int_{t_{i-1}}^s\int_E\delta^Y_k(u,z)^2\nu(du,dz)}ds\mid\mcl F_{t_{i-1}}}
\notag\\
&\leq\frac{16K^4\Lambda_n^2\ol\gamma^Y_2}{n^2}.\label{eq:mj2}
}
Next, the Schwarz inequality gives
\ben{\label{eq:mj3}
\norm{\max_{1\leq i\leq n}\abs{\Delta^n_iM^j\Delta^n_iJ^j}}_p
\leq\norm{\max_{1\leq i\leq n}\abs{\Delta^n_iM^j}}_{2p}\norm{\max_{1\leq i\leq n}\abs{\Delta^n_iJ^j}}_{2p}.
}
Set $q:=2p\vee\log n$. Then we have
\ban{
\norm{\max_{1\leq i\leq n}\abs{\Delta^n_iM^j}}_{2p}
&\leq\norm{\max_{1\leq i\leq n}\abs{\Delta^n_iM^j}}_{q}
\leq\bra{\sum_{i=1}^n\E\sbra{\abs{\Delta^n_iM^j}^{q}}}^{1/q}
\notag\\
&\leq \sqrt e\max_{1\leq i\leq n}\norm{\Delta^n_iM^j}_{q}
\leq4K\sqrt{e\frac{q\Lambda_n}{n}},
\label{eq:mj4}
}
where the third inequality follows from $q\geq\log n$ and the fourth from Lemma \ref{sharp-BDG} and \eqref{cy-bound}. 
Meanwhile, Theorem 1in \cite{ReTi03} gives
\ba{
\norm{\Delta^n_iJ^j}_{2p}
\lesssim\frac{p}{\log p}\bra{\norm{\sqrt{\int_{t_{i-1}}^{t_i}\int_E\delta^Y_j(s,z)^2\nu(ds,dz)}}_{2p}
+\norm{\sup_{s\in(t_{i-1},t_i]}|\Delta J^j_s|}_{2p}}.
}
By \eqref{jy-bound}, we obtain
\ba{
\sum_{i=1}^n\norm{\sqrt{\int_{t_{i-1}}^{t_i}\int_E\delta^Y_j(s,z)^2\nu(ds,dz)}}_{2p}^{2p}
=\sum_{i=1}^n\E\sbra{\bra{\int_{t_{i-1}}^{t_i}\int_E\delta^Y_j(s,z)^2\nu(ds,dz)}^p}
\leq\frac{\bra{K^2\Lambda_n\ol\gamma^Y_2}^p}{n^{p-1}}
}
and
\ba{
\sum_{i=1}^n\norm{\sup_{s\in[t_{i-1},t_i]}|\Delta J^j_s|}_{2p}^{2p}
&\leq\sum_{i=1}^n\E\sbra{\int_{t_{i-1}}^{t_i}\int_E|\delta^Y_j(s,z)|^{2p}\mu(ds,dz)}\\
&=\E\sbra{\int_{0}^{1}\int_E|\delta^Y_j(s,z)|^{2p}\nu(ds,dz)}
\leq(4K^2\Lambda_n)^{p}\ol\gamma^Y_2.
}
Consequently,
\ben{\label{eq:mj5}
\norm{\max_{1\leq i\leq n}\abs{\Delta^n_iJ^j}}_{2p}
\leq\bra{\sum_{i=1}^n\E\sbra{\abs{\Delta^n_iJ^j}^{2p}}}^{1/(2p)}
\lesssim\bra{\frac{p}{\log p}K^2\sqrt\Lambda_n\bra{\sqrt{\ol\gamma^Y_2}+(\ol\gamma^Y_2)^{1/(2p)}}}.
}
By \eqref{eq:mj1}--\eqref{eq:mj5}, we conclude
\[
\max_{1\leq j\leq d}\norm{{[M^j,J^j]}^n_1}_p
=O\bra{\frac{p}{\log p}\bra{\frac{\Lambda_n}{\sqrt n}+\frac{p}{\log p}\Lambda_n\sqrt{\frac{q}{n}}}}
=o(1),
\]
where the last equality follows by \eqref{ass:d-n}. 
Since $d^p\leq e$, this inequality and Markov's inequality give \eqref{MJ-bound}.
\qed

\section{Additional simulation results}\label{sec:a-sim}

Tables \ref{table:bm-wiener}--\ref{table:bm-alpha=0.75} report the simulation results when the factor process $F$ is an $r$-dimensional Wiener process while other settings are maintained from Section \ref{sec:simulation}. 
Figures \ref{fig:perturb} and \ref{fig:gamma} depict sensitivity of the estimator $\hat r^{P,cor}_\tau(\gamma)$ to the tuning parameters $g(d)$ and $\gamma$. 
Figure \ref{fig:snr} depicts the effect of the signal-to-noise ratio parameter $\theta$. 
Figure \ref{fig:phi} depicts the effect of the idiosyncratic cross-sectional correlation parameter $\phi$.

\begin{table}[p]
\caption{Simulation results when factors are Wiener processes and $L$ is a Wiener process}
\label{table:bm-wiener}
\centering
\small
\begin{tabular}{rrrrrrrrrrr}
  \hline
  $d$ & \multicolumn{2}{c}{$\hat r^{BN}_\tau$} & 
\multicolumn{2}{c}{$\hat r^{P,cor}_\tau(\gamma)$} & 
\multicolumn{2}{c}{$PC_{p1}$} & \multicolumn{2}{c}{\citet{Pe19}} & 
\multicolumn{2}{c}{\citet{On10}}\\ \hline & \multicolumn{10}{c}{$n=26$}\\100 & 20.00 & (0.00) & 5.55 & (0.47) & 20.00 & (0.00) & 11.18 & (0.20) & 2.50 & (0.09) \\ 
  500 & 20.00 & (0.00) & 5.79 & (0.79) & 20.00 & (0.00) & 5.70 & (0.73) & 4.72 & (0.53) \\ 
  1000 & 20.00 & (0.00) & 5.89 & (0.89) & 20.00 & (0.00) & 5.69 & (0.74) & 5.30 & (0.70) \\ 
  1500 & 20.00 & (0.00) & 5.89 & (0.89) & 20.00 & (0.00) & 5.59 & (0.66) & 5.52 & (0.74) \\ 
   & \multicolumn{10}{c}{$n=78$}\\100 & 7.21 & (0.14) & 5.91 & (0.80) & 8.11 & (0.00) & 5.87 & (0.77) & 4.87 & (0.59) \\ 
  500 & 6.39 & (0.63) & 6.00 & (1.00) & 6.32 & (0.69) & 6.01 & (0.98) & 6.04 & (0.96) \\ 
  1000 & 6.32 & (0.69) & 6.00 & (1.00) & 6.02 & (0.98) & 6.00 & (1.00) & 6.03 & (0.97) \\ 
  1500 & 6.34 & (0.67) & 6.00 & (1.00) & 6.00 & (1.00) & 6.00 & (1.00) & 6.02 & (0.98) \\ 
   & \multicolumn{10}{c}{$n=390$}\\100 & 6.33 & (0.66) & 6.13 & (0.80) & 8.01 & (0.02) & 5.96 & (0.80) & 6.33 & (0.53) \\ 
  500 & 6.01 & (0.99) & 6.02 & (0.98) & 7.14 & (0.20) & 6.05 & (0.95) & 6.89 & (0.42) \\ 
  1000 & 6.00 & (1.00) & 6.01 & (0.99) & 6.53 & (0.53) & 6.45 & (0.61) & 6.90 & (0.39) \\ 
  1500 & 6.00 & (1.00) & 6.00 & (1.00) & 6.15 & (0.85) & 6.28 & (0.73) & 6.76 & (0.41) \\ 
   \hline
\end{tabular}\vspace{2mm}

\parbox{13cm}{\small
\textit{Note}. For each estimator $\hat r$, this table reports the empirical mean of $\hat r$ and the empirical probability that $\hat r=6$ (in parentheses) based on 1,000 Monte Carlo iterations. 
} 
\end{table}

\begin{table}[p]
\small
\caption{Simulation results when factors are Wiener processes and $L$ is a NTS process with $\alpha=0.25$}
\label{table:bm-alpha=0.25}
\centering
\begin{tabular}{rrrrrrrrrrr}
  \hline
  $d$ & \multicolumn{2}{c}{$\hat r^{BN}_\tau$} & 
\multicolumn{2}{c}{$\hat r^{P,cor}_\tau(\gamma)$} & 
\multicolumn{2}{c}{$PC_{p1}$} & \multicolumn{2}{c}{\citet{Pe19}} & 
\multicolumn{2}{c}{\citet{On10}}\\ \hline & \multicolumn{10}{c}{$n=26$}\\100 & 20.00 & (0.00) & 5.84 & (0.56) & 20.00 & (0.00) & 14.96 & (0.06) & 2.92 & (0.17) \\ 
  500 & 20.00 & (0.00) & 5.95 & (0.89) & 20.00 & (0.00) & 5.97 & (0.87) & 5.17 & (0.67) \\ 
  1000 & 20.00 & (0.00) & 5.98 & (0.94) & 20.00 & (0.00) & 5.94 & (0.92) & 5.68 & (0.83) \\ 
  1500 & 20.00 & (0.00) & 5.97 & (0.96) & 20.00 & (0.00) & 5.91 & (0.91) & 5.83 & (0.88) \\ 
   & \multicolumn{10}{c}{$n=78$}\\100 & 11.61 & (0.00) & 6.01 & (0.80) & 15.00 & (0.00) & 6.29 & (0.68) & 4.91 & (0.56) \\ 
  500 & 9.63 & (0.01) & 6.01 & (0.99) & 9.28 & (0.01) & 6.12 & (0.89) & 6.08 & (0.93) \\ 
  1000 & 9.72 & (0.00) & 6.00 & (1.00) & 6.72 & (0.46) & 6.03 & (0.97) & 6.06 & (0.95) \\ 
  1500 & 10.47 & (0.00) & 6.00 & (1.00) & 6.17 & (0.84) & 6.01 & (0.99) & 6.06 & (0.94) \\ 
   & \multicolumn{10}{c}{$n=390$}\\100 & 8.86 & (0.02) & 6.18 & (0.75) & 15.64 & (0.00) & 6.26 & (0.70) & 6.48 & (0.48) \\ 
  500 & 6.31 & (0.73) & 6.03 & (0.97) & 12.66 & (0.00) & 6.21 & (0.81) & 6.66 & (0.62) \\ 
  1000 & 6.10 & (0.90) & 6.02 & (0.98) & 9.52 & (0.02) & 6.54 & (0.57) & 6.63 & (0.60) \\ 
  1500 & 6.06 & (0.94) & 6.00 & (1.00) & 7.78 & (0.15) & 6.37 & (0.67) & 6.55 & (0.62) \\ 
   \hline
\end{tabular}\vspace{2mm}

\parbox{13cm}{\small
\textit{Note}. For each estimator $\hat r$, this table reports the empirical mean of $\hat r$ and the empirical probability that $\hat r=6$ (in parentheses) based on 1,000 Monte Carlo iterations. 
} 
\end{table}

\begin{table}[p]
\caption{Simulation results when factors are Wiener processes and $L$ is a NTS process with $\alpha=0.50$}
\label{table:bm-alpha=0.5}
\centering
\small
\begin{tabular}{rrrrrrrrrrr}
  \hline
  $d$ & \multicolumn{2}{c}{$\hat r^{BN}_\tau$} & 
\multicolumn{2}{c}{$\hat r^{P,cor}_\tau(\gamma)$} & 
\multicolumn{2}{c}{$PC_{p1}$} & \multicolumn{2}{c}{\citet{Pe19}} & 
\multicolumn{2}{c}{\citet{On10}}\\ \hline & \multicolumn{10}{c}{$n=26$}\\100 & 20.00 & (0.00) & 5.79 & (0.55) & 20.00 & (0.00) & 14.54 & (0.07) & 2.87 & (0.14) \\ 
  500 & 20.00 & (0.00) & 5.91 & (0.88) & 20.00 & (0.00) & 5.90 & (0.86) & 5.15 & (0.67) \\ 
  1000 & 20.00 & (0.00) & 5.96 & (0.94) & 20.00 & (0.00) & 5.90 & (0.90) & 5.60 & (0.81) \\ 
  1500 & 20.00 & (0.00) & 5.97 & (0.95) & 20.00 & (0.00) & 5.88 & (0.88) & 5.74 & (0.86) \\ 
   & \multicolumn{10}{c}{$n=78$}\\100 & 10.90 & (0.00) & 6.00 & (0.77) & 14.10 & (0.00) & 6.19 & (0.70) & 4.93 & (0.57) \\ 
  500 & 9.36 & (0.00) & 6.00 & (0.99) & 9.06 & (0.01) & 6.12 & (0.89) & 6.07 & (0.93) \\ 
  1000 & 9.69 & (0.00) & 6.00 & (1.00) & 6.82 & (0.42) & 6.02 & (0.98) & 6.06 & (0.95) \\ 
  1500 & 10.34 & (0.00) & 6.00 & (1.00) & 6.22 & (0.80) & 6.00 & (0.99) & 6.05 & (0.96) \\ 
   & \multicolumn{10}{c}{$n=390$}\\100 & 8.61 & (0.03) & 6.16 & (0.78) & 14.47 & (0.00) & 6.20 & (0.74) & 6.34 & (0.51) \\ 
  500 & 6.38 & (0.68) & 6.03 & (0.97) & 12.38 & (0.00) & 6.17 & (0.84) & 6.59 & (0.62) \\ 
  1000 & 6.14 & (0.87) & 6.01 & (0.99) & 9.61 & (0.01) & 6.55 & (0.56) & 6.65 & (0.57) \\ 
  1500 & 6.10 & (0.90) & 6.00 & (1.00) & 7.99 & (0.12) & 6.35 & (0.68) & 6.51 & (0.62) \\ 
   \hline
\end{tabular}\vspace{2mm}

\parbox{13cm}{\small
\textit{Note}. For each estimator $\hat r$, this table reports the empirical mean of $\hat r$ and the empirical probability that $\hat r=6$ (in parentheses) based on 1,000 Monte Carlo iterations. 
} 
\end{table}

\begin{table}[p]
\caption{Simulation results when factors are Wiener processes and $L$ is a NTS process with $\alpha=0.75$}
\label{table:bm-alpha=0.75}
\centering
\small
\begin{tabular}{rrrrrrrrrrr}
  \hline
  $d$ & \multicolumn{2}{c}{$\hat r^{BN}_\tau$} & 
\multicolumn{2}{c}{$\hat r^{P,cor}_\tau(\gamma)$} & 
\multicolumn{2}{c}{$PC_{p1}$} & \multicolumn{2}{c}{\citet{Pe19}} & 
\multicolumn{2}{c}{\citet{On10}}\\ \hline & \multicolumn{10}{c}{$n=26$}\\100 & 20.00 & (0.00) & 5.71 & (0.52) & 20.00 & (0.00) & 13.44 & (0.12) & 2.79 & (0.15) \\ 
  500 & 20.00 & (0.00) & 5.92 & (0.88) & 20.00 & (0.00) & 5.88 & (0.86) & 5.14 & (0.65) \\ 
  1000 & 20.00 & (0.00) & 5.95 & (0.94) & 20.00 & (0.00) & 5.86 & (0.87) & 5.51 & (0.78) \\ 
  1500 & 20.00 & (0.00) & 5.94 & (0.94) & 20.00 & (0.00) & 5.83 & (0.84) & 5.74 & (0.85) \\ 
   & \multicolumn{10}{c}{$n=78$}\\100 & 9.83 & (0.00) & 6.01 & (0.81) & 12.48 & (0.00) & 6.08 & (0.77) & 4.88 & (0.59) \\ 
  500 & 8.90 & (0.02) & 6.00 & (1.00) & 8.61 & (0.03) & 6.07 & (0.94) & 6.07 & (0.94) \\ 
  1000 & 9.33 & (0.01) & 6.00 & (1.00) & 6.87 & (0.39) & 6.01 & (0.99) & 6.05 & (0.95) \\ 
  1500 & 9.97 & (0.00) & 6.00 & (1.00) & 6.34 & (0.71) & 6.00 & (1.00) & 6.03 & (0.97) \\ 
   & \multicolumn{10}{c}{$n=390$}\\100 & 7.91 & (0.08) & 6.16 & (0.76) & 12.30 & (0.00) & 6.12 & (0.76) & 6.35 & (0.51) \\ 
  500 & 6.46 & (0.62) & 6.03 & (0.97) & 11.50 & (0.00) & 6.11 & (0.90) & 6.71 & (0.54) \\ 
  1000 & 6.29 & (0.75) & 6.01 & (0.99) & 9.66 & (0.01) & 6.55 & (0.54) & 6.72 & (0.52) \\ 
  1500 & 6.21 & (0.81) & 6.00 & (1.00) & 8.30 & (0.07) & 6.36 & (0.65) & 6.66 & (0.52) \\ 
   \hline
\end{tabular}\vspace{2mm}

\parbox{13cm}{\small
\textit{Note}. For each estimator $\hat r$, this table reports the empirical mean of $\hat r$ and the empirical probability that $\hat r=6$ (in parentheses) based on 1,000 Monte Carlo iterations. 
} 
\end{table}

\begin{figure}[p]
\centering
\caption{Sensitivity of $\hat{r}^{P,cor}_\tau(\gamma)$ to $g(d)$}
\label{fig:perturb}
\includegraphics[scale=1]{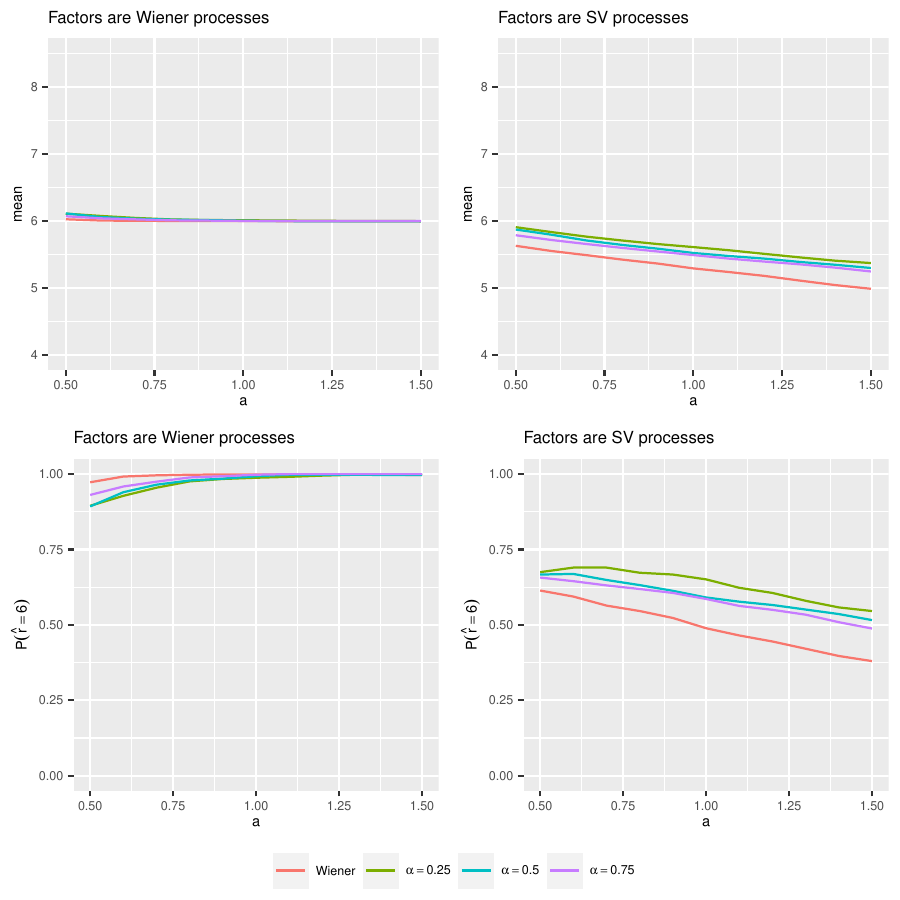}
\medskip

\parbox{15cm}{\small
\textit{Note}. This figure shows the empirical mean of $\hat r^{P,cor}_\tau(\gamma)$ (top panels) and the empirical probability that $\hat r^{P,cor}_\tau(\gamma)=6$ (bottom panels) for various values of $a$ with $g(d)=a\sqrt{\log\log d}$ and $\gamma=0.05$ when $n=78$ and $d=500$. The left panels are for the case that $F$ is an $r$-dimensional Wiener process and the right panels are for the case that $F$ is generated by the stochastic volatility model described in Section \ref{sec:sim-design}. Other settings are the same as in Section \ref{sec:sim-design}. 
Different colors correspond to four different models for the idiosyncratic process $Z$ in the simulation. 
}
\end{figure}

\begin{figure}[p]
\centering
\caption{Sensitivity of $\hat{r}^{P,cor}_\tau(\gamma)$ to $\gamma$}
\label{fig:gamma}
\includegraphics[scale=1]{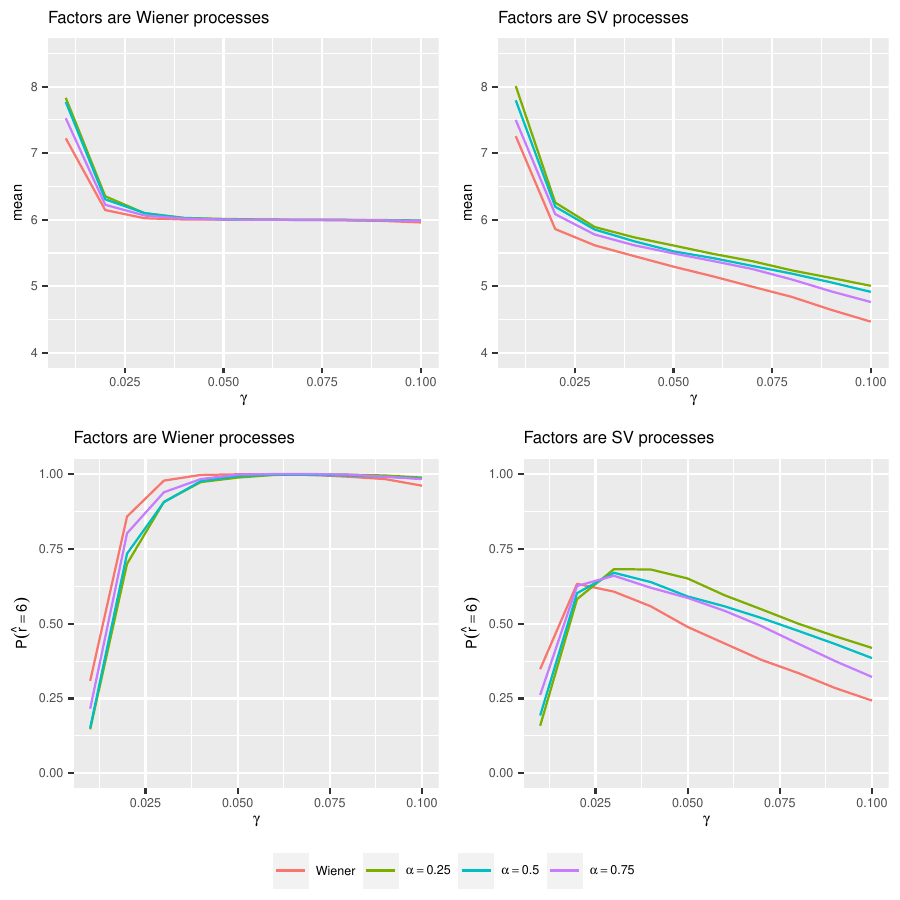}
\medskip

\parbox{15cm}{\small
\textit{Note}. This figure shows the empirical mean of $\hat r^{P,cor}_\tau(\gamma)$ (top panels) and the empirical probability that $\hat r^{P,cor}_\tau(\gamma)=6$ (bottom panels) for various values of $\gamma$ with $g(d)=\sqrt{\log\log d}$ when $n=78$ and $d=500$. The left panels are for the case that $F$ is an $r$-dimensional Wiener process and the right panels are for the case that $F$ is generated by the stochastic volatility model described in Section \ref{sec:sim-design}. Other settings are the same as in Section \ref{sec:sim-design}. 
Different colors correspond to four different models for the idiosyncratic process $Z$ in the simulation. 
}
\end{figure}

\begin{figure}[p]
\centering
\caption{Effect of the signal-to-noise ratio}
\label{fig:snr}
\includegraphics[scale=1]{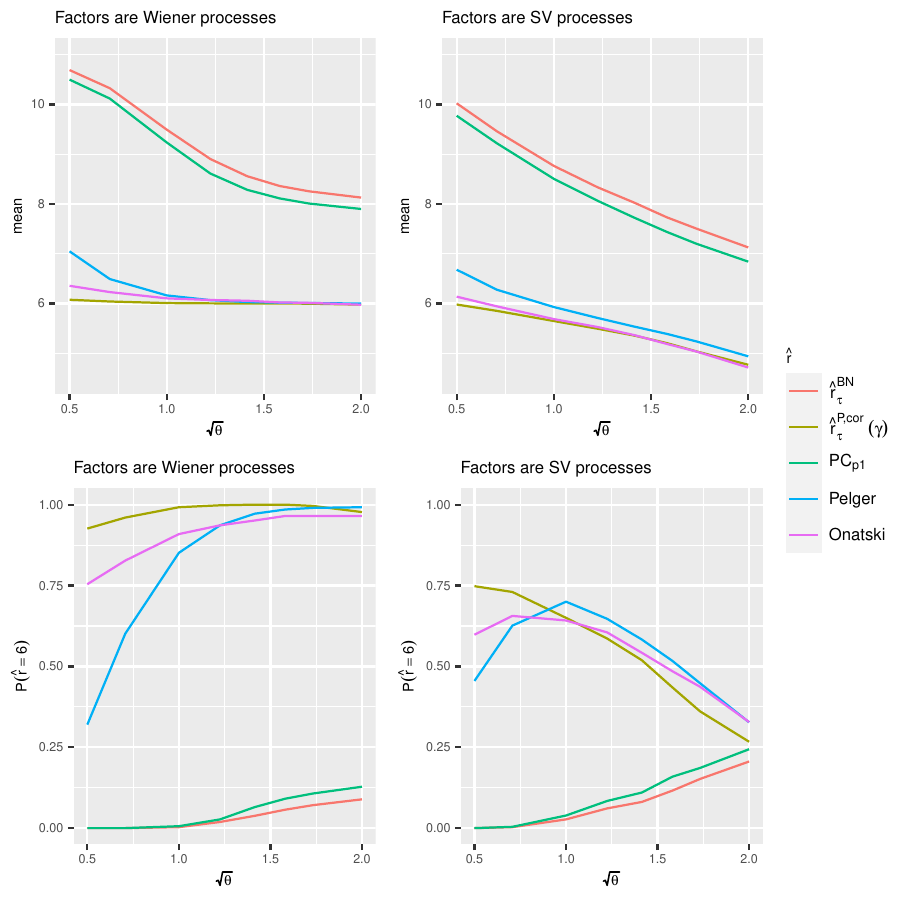}
\medskip

\parbox{15cm}{\small
\textit{Note}. This figure shows the empirical mean of $\hat r$ (top panels) and the empirical probability that $\hat r=6$ (bottom panels) for various values of $\theta$ when $n=78$ and $d=500$. 
The left panels are for the case that $F$ is an $r$-dimensional Wiener process and the right panels are for the case that $F$ is generated by the stochastic volatility model described in Section \ref{sec:sim-design}. 
The idiosyncratic process $Z$ is a $d$-dimensional NTS process with $\alpha=0.75$. 
Other settings are the same as in Section \ref{sec:simulation}. 
Different colors correspond to five different estimators in the simulation. 
}
\end{figure}

\begin{figure}[p]
\centering
\caption{Effect of the cross-sectional correlation of the idiosyncratic process}
\label{fig:phi}
\includegraphics[scale=1]{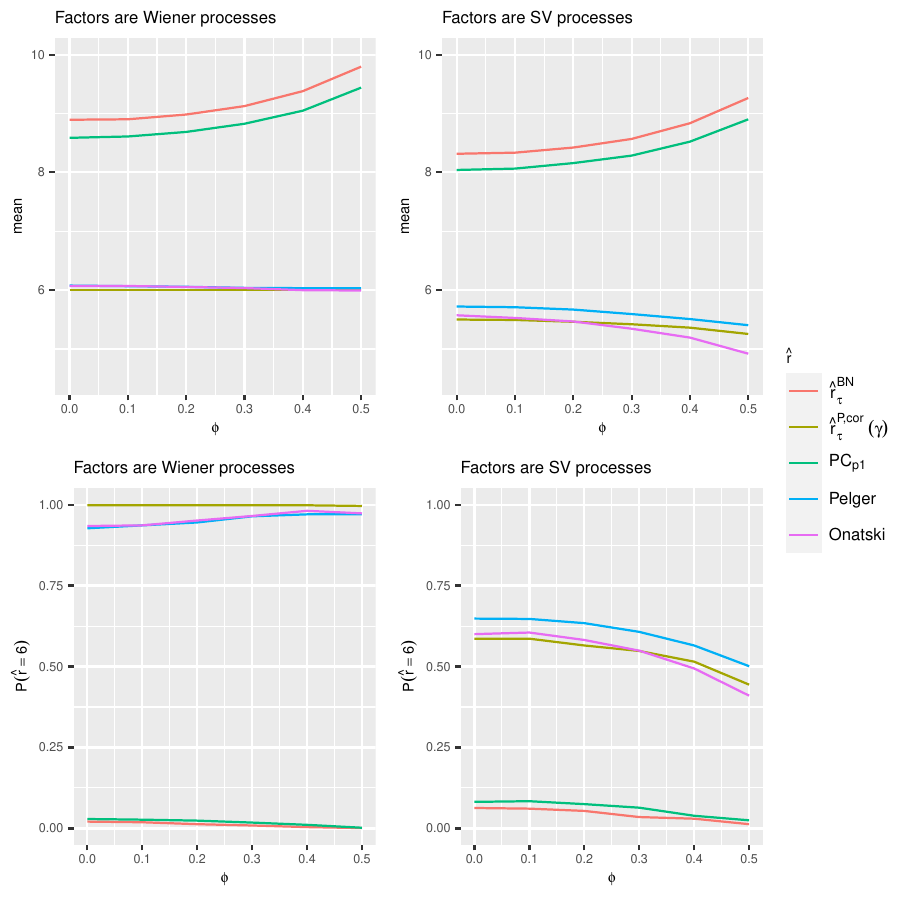}
\medskip

\parbox{15cm}{\small
\textit{Note}. This figure shows the empirical mean of $\hat r$ (top panels) and the empirical probability that $\hat r=6$ (bottom panels) for various values of $\phi$ when $n=78$ and $d=500$. 
The left panels are for the case that $F$ is an $r$-dimensional Wiener process and the right panels are for the case that $F$ is generated by the stochastic volatility model described in Section \ref{sec:sim-design}. 
The idiosyncratic process $Z$ is a $d$-dimensional NTS process with $\alpha=0.75$. 
Other settings are the same as in Section \ref{sec:simulation}. 
Different colors correspond to five different estimators in the simulation. 
}
\end{figure}

\newpage

\paragraph{Acknowledgments}

The author thanks Yuma Uehara for his helpful comments on simulation. 
This work was partly supported by JST CREST Grant Number JPMJCR2115 and JSPS KAKENHI Grant Numbers JP19K13668, JP22H00834, JP22H00889, JP22H01139.

{\small
\renewcommand*{\baselinestretch}{1}\selectfont
\addcontentsline{toc}{section}{References}
\bibliography{base}

\begin{thebibliography}{58}
\expandafter\ifx\csname natexlab\endcsname\relax\def\natexlab#1{#1}\fi

\bibitem[{Ahn \& Horenstein(2013)}]{AhHo13}
Ahn, S.~C. \& Horenstein, A.~R. (2013).
\newblock Eigenvalue ratio test for the number of factors.
\newblock \emph{Econometrica} \textbf{81}, 1203--1227.

\bibitem[{A\"{i}t-Sahalia \& Jacod(2014)}]{AJ2014}
A\"{i}t-Sahalia, Y. \& Jacod, J. (2014).
\newblock \emph{High-frequency financial econometrics}.
\newblock Princeton University Press.

\bibitem[{A{\"\i}t-Sahalia \& Xiu(2017)}]{AX2017}
A{\"\i}t-Sahalia, Y. \& Xiu, D. (2017).
\newblock Using principal component analysis to estimate a high dimensional
  factor model with high-frequency data.
\newblock \emph{J. Econometrics} \textbf{201}, 384--399.

\bibitem[{Asmar \& Montgomery-Smith(1993)}]{AsMo93}
Asmar, N. \& Montgomery-Smith, S. (1993).
\newblock {L}ittlewood--{P}aley theory on solenoids.
\newblock \emph{Colloq. Math.} \textbf{65}, 69--82.

\bibitem[{Azoff(1974)}]{Az74}
Azoff, E.~A. (1974).
\newblock Borel measurability in linear algebra.
\newblock \emph{Proc. Amer. Math. Soc.} \textbf{42}, 346--350.

\bibitem[{Bai \& Ng(2002)}]{BN2002}
Bai, J. \& Ng, S. (2002).
\newblock Determining the number of factors in approximate factor models.
\newblock \emph{Econometrica} \textbf{70}, 191--221.

\bibitem[{Bai \& Ng(2006)}]{BaNg06}
Bai, J. \& Ng, S. (2006).
\newblock Determining the number of factors in approximate factor models,
  {E}rrata.
\newblock Available at
  \url{http://www.columbia.edu/~sn2294/papers/correctionEcta2.pdf}.

\bibitem[{Bai \& Ng(2023)}]{BaNg23}
Bai, J. \& Ng, S. (2023).
\newblock Approximate factor models with weaker loadings.
\newblock \emph{J. Econometrics} \textbf{235}, 1893--1916.

\bibitem[{Barndorff-Nielsen \emph{et~al.}(2011)Barndorff-Nielsen, Hansen, Lunde
  \& Shephard}]{BNHLS2011}
Barndorff-Nielsen, O.~E., Hansen, P.~R., Lunde, A. \& Shephard, N. (2011).
\newblock Multivariate realised kernels: Consistent positive semi-definite
  estimators of the covariation of equity prices with noise and non-synchronous
  trading.
\newblock \emph{J. Econometrics} \textbf{162}, 149--169.

\bibitem[{Carlen \& Kree(1991)}]{CK1991}
Carlen, E. \& Kree, P. (1991).
\newblock {$L^p$} estimates on iterated stochastic integrals.
\newblock \emph{Ann. Probab.} \textbf{19}, 354--368.

\bibitem[{Chen \emph{et~al.}(2020)Chen, Mykland \& Zhang}]{CMZ20}
Chen, D., Mykland, P.~A. \& Zhang, L. (2020).
\newblock The five trolls under the bridge: Principal component analysis with
  asynchronous and noisy high frequency data.
\newblock \emph{J. Amer. Statist. Assoc.} \textbf{115}, 1960--1977.

\bibitem[{Christensen \emph{et~al.}(2023)Christensen, Nielsen \&
  Podolskij}]{CNP23}
Christensen, K., Nielsen, M.~S. \& Podolskij, M. (2023).
\newblock High-dimensional estimation of quadratic variation based on penalized
  realized variance.
\newblock \emph{Stat. Inference Stoch. Process.} \textbf{26}, 331--359.

\bibitem[{Connor \& Korajczyk(2022)}]{CoKo22}
Connor, G. \& Korajczyk, R.~A. (2022).
\newblock Semi-strong factors in asset returns.
\newblock \emph{Journal of Financial Econometrics (forthcoming)} .

\bibitem[{Cont \& Tankov(2004)}]{CT2004}
Cont, R. \& Tankov, P. (2004).
\newblock \emph{Financial modeling with jump processes}.
\newblock Chapman \& Hall/CRC.

\bibitem[{Dai \emph{et~al.}(2019)Dai, Lu \& Xiu}]{DLX2019}
Dai, C., Lu, K. \& Xiu, D. (2019).
\newblock Knowing factors or factor loadings, or neither? {E}valuating
  estimators of large covariance matrices with noisy and asynchronous data.
\newblock \emph{J. Econometrics} \textbf{208}, 43--79.

\bibitem[{Fan \& Kim(2018)}]{FK2017}
Fan, J. \& Kim, D. (2018).
\newblock Robust high-dimensional volatility matrix estimation for
  high-frequency factor model.
\newblock \emph{J. Amer. Statist. Assoc.} \textbf{113}, 1268--1283.

\bibitem[{Freyaldenhoven(2022)}]{Fr22}
Freyaldenhoven, S. (2022).
\newblock Factor models with local factors --- determining the number of
  relevant factors.
\newblock \emph{J. Econometrics} \textbf{229}, 80--102.

\bibitem[{Giglio \emph{et~al.}(2021)Giglio, Xiu \& Zhang}]{GXZ21}
Giglio, S., Xiu, D. \& Zhang, D. (2021).
\newblock Test assets and weak factors.
\newblock Working Paper 21-04, University of Chicago Booth School of Business.

\bibitem[{Haagerup \& Musat(2007)}]{HaMu07}
Haagerup, U. \& Musat, M. (2007).
\newblock On the best constants in noncommutative {K}hintchine-type
  inequalities.
\newblock \emph{J. Funct. Anal.} \textbf{250}.

\bibitem[{Heinrich \& Podolskij(2014)}]{HP2014}
Heinrich, C. \& Podolskij, M. (2014).
\newblock On spectral distribution of high dimensional covariation matrices.
\newblock Working paper. Available at arXiv:
  \url{https://arxiv.org/abs/1410.6764}.

\bibitem[{Horn \& Johnson(2013)}]{HJ2013}
Horn, R.~A. \& Johnson, C.~R. (2013).
\newblock \emph{Matrix analysis}.
\newblock Cambridge University Press, 2nd edn.

\bibitem[{Huang \& Tauchen(2005)}]{HT2005}
Huang, X. \& Tauchen, G. (2005).
\newblock The relative contribution of jumps to total price variance.
\newblock \emph{Journal of Financial Econometrics} \textbf{3}, 456--499.

\bibitem[{Hyt{\"o}nen \emph{et~al.}(2016)Hyt{\"o}nen, {v}an Neerven, Veraar \&
  Weis}]{HvNVW16}
Hyt{\"o}nen, T., {v}an Neerven, J., Veraar, M. \& Weis, L. (2016).
\newblock \emph{Analysis in {B}anach spaces --- {V}olume {I}: Martingales and
  {L}ittlewood-{P}aley theory}.
\newblock Springer.

\bibitem[{Jacod \& Protter(2012)}]{JP2012}
Jacod, J. \& Protter, P. (2012).
\newblock \emph{Discretization of processes}.
\newblock Springer.

\bibitem[{Jacod \& Shiryaev(2003)}]{JS}
Jacod, J. \& Shiryaev, A.~N. (2003).
\newblock \emph{Limit theorems for stochastic processes}.
\newblock Springer, 2nd edn.

\bibitem[{Kallenberg \& Sztencel(1991)}]{KaSz91}
Kallenberg, O. \& Sztencel, R. (1991).
\newblock Some dimension-free features of vector-valued martingales.
\newblock \emph{Probab. Theory Related Fields} \textbf{88}, 215--247.

\bibitem[{Kawai \& Masuda(2011)}]{KM2011}
Kawai, R. \& Masuda, H. (2011).
\newblock On simulation of tempered stable random variates.
\newblock \emph{J. Comput. Appl. Math.} \textbf{235}, 2873--2887.

\bibitem[{Kong(2017)}]{Kong17}
Kong, X.-B. (2017).
\newblock On the number of common factors with high-frequency data.
\newblock \emph{Biometrika} \textbf{104}, 397--410.

\bibitem[{Kong(2018)}]{Kong18}
Kong, X.-B. (2018).
\newblock On the systematic and idiosyncratic volatility with large panel
  high-frequency data.
\newblock \emph{Ann. Statist.} \textbf{46}, 1077--1108.

\bibitem[{Kong \emph{et~al.}(2023)Kong, Lin, Liu \& Liu}]{KLLL23}
Kong, X.-B., Lin, J.-G., Liu, C. \& Liu, G.-Y. (2023).
\newblock Discrepancy between global and local principal component analysis on
  large-panel high-frequency data.
\newblock \emph{J. Amer. Statist. Assoc.} \textbf{118}, 1333--1344.

\bibitem[{Kong \emph{et~al.}(2019)Kong, Liu \& Zhou}]{KLZ19}
Kong, X.-B., Liu, Z. \& Zhou, W. (2019).
\newblock A rank test for the number of factors with high-frequency data.
\newblock \emph{J. Econometrics} \textbf{211}, 439--460.

\bibitem[{Lata{\l}a(2005)}]{La05}
Lata{\l}a, R. (2005).
\newblock Some estimates of norms of random matrices.
\newblock \emph{Proc. Amer. Math. Soc.} \textbf{133}, 1273--1282.

\bibitem[{Li \emph{et~al.}(2023)Li, Chen \& Linton}]{LCL23}
Li, Y.-N., Chen, J. \& Linton, O. (2023).
\newblock Estimation of common factors for microstructure noise and efficient
  price in a high-frequency dual factor model.
\newblock \emph{J. Econometrics (forthcoming)} .

\bibitem[{Lust-Piquard \& Pisier(1991)}]{LPP91}
Lust-Piquard, F. \& Pisier, G. (1991).
\newblock Non commutative {K}hintchine and {P}aley inequalities.
\newblock \emph{Ark. Mat.} \textbf{29}, 241--260.

\bibitem[{Mackey \emph{et~al.}(2014)Mackey, Jordan, Chen, Farrell \&
  Tropp}]{MJCFT14}
Mackey, L., Jordan, M.~I., Chen, R.~Y., Farrell, B. \& Tropp, J.~A. (2014).
\newblock Matrix concentration inequalities via the method of exchangeable
  pairs.
\newblock \emph{Ann. Probab.} \textbf{42}, 906--945.

\bibitem[{Magnus \& Neudecker(2019)}]{MaNe19}
Magnus, J.~R. \& Neudecker, H. (2019).
\newblock \emph{Matrix differential calculus with applications in statistics
  and econometrics}.
\newblock Wiley, 3rd edn.

\bibitem[{Moon \& Weidner(2017)}]{MoWe17}
Moon, H.~R. \& Weidner, M. (2017).
\newblock Dynamic linear panel regression models with interactive fixed
  effects.
\newblock \emph{Econometric Theory} \textbf{33}, 158--195.

\bibitem[{Onatski(2010)}]{On10}
Onatski, A. (2010).
\newblock Determining the number of factors from empirical distribution of
  eigenvalues.
\newblock \emph{Review of Economics and Statistics} \textbf{92}, 1004--1016.

\bibitem[{Pelger(2019)}]{Pe19}
Pelger, M. (2019).
\newblock Large-dimensional factor modeling based on high-frequency
  observations.
\newblock \emph{J. Econometrics} \textbf{208}, 23--42.

\bibitem[{Pelger(2020)}]{Pe20}
Pelger, M. (2020).
\newblock Understanding systematic risk: A high-frequency approach.
\newblock \emph{Journal of Finance} \textbf{75}, 2179--2220.

\bibitem[{Pinelis(1994)}]{Pi94}
Pinelis, I. (1994).
\newblock Optimum bounds for the distributions of martingales in {B}anach
  spaces.
\newblock \emph{Ann. Probab.} \textbf{22}, 1679--1706.

\bibitem[{Pisier(2016)}]{Pi16}
Pisier, G. (2016).
\newblock \emph{Martingales in {B}anach spaces}.
\newblock Cambridge University Press.

\bibitem[{Randrianantoanina(2002)}]{Ra02}
Randrianantoanina, N. (2002).
\newblock Non-commutative martingale transforms.
\newblock \emph{J. Funct. Anal.} \textbf{194}, 181--212.

\bibitem[{Ren \& Tian(2003)}]{ReTi03}
Ren, Y.-F. \& Tian, F.-J. (2003).
\newblock On the {R}osenthal's inequality for locally square integrable
  martingales.
\newblock \emph{Stochastic Process. Appl.} \textbf{104}, 107--116.

\bibitem[{Resnick(1987)}]{Res1987}
Resnick, S.~I. (1987).
\newblock \emph{Extreme values, regular variation, and point processes}.
\newblock Springer.

\bibitem[{Rudelson(1999)}]{Ru99}
Rudelson, M. (1999).
\newblock Random vectors in the isotropic position.
\newblock \emph{J. Funct. Anal.} \textbf{164}, 60--72.

\bibitem[{Sun \emph{et~al.}(2023)Sun, Xu \& Zhang}]{SXZ23}
Sun, Y., Xu, W. \& Zhang, C. (2023).
\newblock Identifying latent factors based on high-frequency data.
\newblock \emph{J. Econometrics} \textbf{233}, 251--270.

\bibitem[{Tropp(2015)}]{Tr15}
Tropp, J.~A. (2015).
\newblock \emph{An introduction to matrix concentration inequalities}, vol.~8.
\newblock Foundations and Trends{\textregistered} in Machine Learning.

\bibitem[{Trzcinka(1986)}]{Tr86}
Trzcinka, C. (1986).
\newblock On the number of factors in the arbitrage pricing model.
\newblock \emph{Journal of Finance} \textbf{41}, 347--368.

\bibitem[{Uematsu \& Yamagata(2023{\natexlab{a}})}]{UeYa23a}
Uematsu, Y. \& Yamagata, T. (2023{\natexlab{a}}).
\newblock Estimation of sparsity-induced weak factor models.
\newblock \emph{J. Bus. Econom. Statist.} \textbf{41}, 213--227.

\bibitem[{Uematsu \& Yamagata(2023{\natexlab{b}})}]{UeYa23b}
Uematsu, Y. \& Yamagata, T. (2023{\natexlab{b}}).
\newblock Inference in sparsity-induced weak factor models.
\newblock \emph{J. Bus. Econom. Statist.} \textbf{41}, 126--139.

\bibitem[{{v}an~de Geer(2000)}]{vdG00}
{v}an~de Geer, S. (2000).
\newblock \emph{Empirical processes in {M}-estimation}.
\newblock Cambridge University Press.

\bibitem[{{v}an~de Geer \& Lederer(2013)}]{vdGL13}
{v}an~de Geer, S. \& Lederer, J. (2013).
\newblock The {B}ernstein--{O}rlicz norm and deviation inequalities.
\newblock \emph{Probab. Theory Related Fields} \textbf{157}, 225--250.

\bibitem[{{v}an~der Vaart \& Wellner(1996)}]{VW1996}
{v}an~der Vaart, A.~W. \& Wellner, J.~A. (1996).
\newblock \emph{Weak convergence and empirical processes}.
\newblock Springer.

\bibitem[{van Handel(2017)}]{vH17}
van Handel, R. (2017).
\newblock Structured random matrices.
\newblock In E.~Carlen, M.~Madiman \& E.~M. Werner, eds., \emph{Convexity and
  concentration}. Springer, pp. 107--156.

\bibitem[{Vershynin(2012)}]{Ve12}
Vershynin, R. (2012).
\newblock Introduction to the non-asymptotic analysis of random matrices.
\newblock In Y.~C. Eldar \& G.~Kutyniok, eds., \emph{Compressed sensing},
  chap.~5. Cambridge University Press, pp. 210--268.

\bibitem[{Wellner(2017)}]{We17}
Wellner, J.~A. (2017).
\newblock The {B}ennett-{O}rlicz norm.
\newblock \emph{Sankhya A} \textbf{79}, 355--383.

\bibitem[{Zheng \& Li(2011)}]{ZL2011}
Zheng, X. \& Li, Y. (2011).
\newblock On the estimation of integrated covariance matrices of high
  dimensional diffusion processes.
\newblock \emph{Ann. Statist.} \textbf{39}, 3121--3151.

\end{thebibliography}
}

\end{document}